\documentclass[preprint,12pt,authoryear]{elsarticle}

\usepackage{amssymb,amsthm,amsmath,amsfonts,mathrsfs}
\usepackage{microtype}
\usepackage{stmaryrd}
\SetSymbolFont{stmry}{bold}{U}{stmry}{m}{n}
\usepackage{bm}
\usepackage{cancel}
\usepackage{enumerate}
\usepackage{wasysym}

\usepackage{multirow}
\usepackage{boldline}
\usepackage{array}
\usepackage{hhline}

\usepackage{chessboard}
\usepackage{pst-all}

\usepackage{algorithm}
\usepackage{algpseudocode}

\usepackage[hidelinks]{hyperref}
\usepackage[capitalise]{cleveref}

\newcommand{\GL}{{G^\mathcal{L}}}
\newcommand{\GR}{{G^\mathcal{R}}}
\newcommand{\N}{\mathcal{N}}
\renewcommand{\P}{\mathcal{P}}
\renewcommand{\L}{\mathcal{L}}
\newcommand{\R}{\mathcal{R}}
\newcommand{\cgstar}{\mathord\ast}
\newcommand{\cgup}{\mathord\uparrow}
\newcommand{\cgupstar}{\cgup\cgstar}
\newcommand{\cgdoubleup}{\mathord\Uparrow}
\newcommand{\cgallstar}{\mathord\circledast}
\newcommand{\os}{\mathbin:}
\newcommand{\cggfuz}{\mathrel{\rule[-.05ex]{.05ex}{1.3ex}\hspace{.8pt}\rhd}}
\newcommand{\cglfuz}{\mathrel{\lhd\hspace{.8pt}\rule[-.05ex]{.05ex}{1.3ex}}}
\newcommand{\cgfuzzy}{\mathrel\|}
\newcommand{\rank}{{\rm rank}}
\renewcommand{\o}{\mathrm{o}}
\newcommand{\so}{\hat{\mathrm{o}}}
\newcommand{\RO}{\mbox{RO}}
\newcommand{\LO}{\mbox{LO}}
\newcommand{\mex}{\mathrm{mex}}
\newcommand{\moon}{\scalebox{1.1}{\leftmoon}}
\newcommand{\moonsmall}{\scalebox{0.9}{\leftmoon}}
\newcommand{\nym}{%
  \infty\hspace{-0.52em}\raisebox{-0.2ex}{\rule{0.1ex}{1.4ex}}\,\,\,%
}

\newtheorem{theorem}{Theorem}[section] 

\theoremstyle{definition}
\newtheorem{definition}[theorem]{Definition}

\theoremstyle{remark}

\newtheorem*{claim*}{Claim}

\newtheorem{observation}[theorem]{Observation}

\makeatletter
\newenvironment{breakablealgorithm}  
  {
   \begin{center}
     \refstepcounter{algorithm}
     \hrule height.8pt depth0pt \kern2pt
     \renewcommand{\caption}[2][\relax]{
       {\raggedright\textbf{\fname@algorithm~\thealgorithm} ##2\par}%
       \ifx\relax##1\relax 
         \addcontentsline{loa}{algorithm}{\protect\numberline{\thealgorithm}##2}%
       \else 
         \addcontentsline{loa}{algorithm}{\protect\numberline{\thealgorithm}##1}%
       \fi
       \kern2pt\hrule\kern2pt
     }
  }{
     \kern2pt\hrule\relax
   \end{center}
  }
\makeatother

\emergencystretch=\maxdimen
\usetikzlibrary{decorations.markings} 
\tikzset{
  midway/.style={
    -{}, 
    postaction={
      decorate,
      decoration={
        markings,
        mark=at position 0.5 with {\arrow{stealth}} 
      }
    }
  }
}

\journal{arXiv}

\begin{document}

\begin{frontmatter}

\title{Cyclic impartial games with carry-on moves}

\author{Tomoaki Abuku\fnref{fn1}}
\fntext[fn1]{Funded by JSPS KAKENHI Grant Number JP22K13953.}
\ead{buku3416@gmail.com }
\affiliation{organization={Gifu University},
            country={Japan}}

\author{Alda Carvalho\fnref{fn2}}
\fntext[fn2]{Funded by national funds through the FCT, I.P., under the scope of the project UID/06522/2025 (CEMAPRE/ISEG Research).}
\ead{alda.carvalho@uab.pt}
\affiliation{organization={DCeT, ABERTA University \& CEMAPRE/ISEG Research, ULISBOA},
country={Portugal}}

\author{Urban Larsson}
\ead{larsson@iitb.ac.in}
\affiliation{organization={Indian Institute of Technology Bombay},
            country={India}}

\author{Richard J. Nowakowski}
\ead{r.nowakowski@dal.ca}
\affiliation{organization={Dalhousie University}, country={Canada}}

\author{\texorpdfstring{Carlos P. Santos\corref{correspondingauthor}}{Carlos P.
Santos}}
\cortext[correspondingauthor]{Corresponding author. Funded by national funds through the FCT, I.P., under the scope of the projects
UID/297/2025 and UID/PRR/297/2025 (Center for Mathematics and Applications).}
\ead{Pereira.Santos@iscte-iul.pt}
\affiliation{organization={ISCTE-University Institute of Lisbon \& NovaMath, FCT NOVA}, country={Portugal}}

\author{Koki Suetsugu}
\ead{suetsugu.koki@gmail.com}
\affiliation{organization={Waseda University \& Osaka Metropolitan University \& Toyo University},
            country={Japan}}

\begin{abstract}
In an impartial combinatorial game, both players have the same options in the game and all its subpositions. The classical Sprague-Grundy Theory was developed for short impartial games, where players have a finite number of options, there are no special moves, and an infinite run is not possible. Subsequently, many generalizations have been proposed, particularly the Smith-Frankel-Perl Theory devised for games where the infinite run is possible, and the Larsson-Nowakowski-Santos Theory able to deal with entailing moves that disrupt the logic of the disjunctive sum. This work presents a generalization that combines these two theories, suitable for analyzing cyclic impartial games with carry-on moves, which are particular cases of entailing moves where the entailed player has no freedom of choice in their response. This generalization is illustrated with {\sc green-lime hackenbush}, a game inspired by the classic {\sc green hackenbush}.
\end{abstract}


\end{frontmatter}

\section{Introduction}
\label{sec:intro}
We assume that the reader is acquainted with the fundamental concepts of
two-player combinatorial games as presented in any of
\cite{ANW019,BCG001,Con001,Sie013}. Here, we are exclusively interested in \emph{impartial games} where the allowable moves depend only on the position and not on which player is making the move -- both players have the same options in all followers of the game. Using the underlying game digraph, an approach used in this paper, an impartial game can be seen as a game played with pieces occupying the nodes, and the moves consist of choosing a piece and playing it to a node neighbor. In this document, only \emph{finite digraphs with a finite number of pieces} are considered, meaning that a player can only have a finite number of options. Furthermore, the winning convention is \emph{normal play}, where the player who is deprived of moves on their turn loses.

Over time, some mathematical theories have been proposed to analyze impartial games, depending on the characteristics of the game digraphs. When the moves have no special properties and a position cannot be a follower of itself, meaning the game digraph has no cycles, Sprague-Grundy Theory (SGT) should be used \citep{Bou02,Spra35,Gru39}; this was the first theory to be proposed.
Later, \cite{Smi66,Fra73} generalized SGT to encompass cyclic impartial games, where a position can be a follower of itself, i.e., the game digraph can have cycles\footnote{In the specialized literature, the designation ``loopy impartial games'' is prevalent. However, because we believe that the word ``cyclic'' is considerably more precise, we have taken the liberty of using it here.} (SFPT). On the other hand, even when considering game digraphs without cycles, the moves may have properties that make SGT insufficient for their analysis. Paradigmatic are entailing moves \citep{BCG001}; once a player makes such a move in a component $G$ to a position $G'$, the opponent is forced to respond locally with some $G''$ belonging to the set of available options in $G'$. When this set of options has at most one element, we have a particular case of entailing move, called carry-on move; if it is empty, the opponent loses immediately (they are compelled to play in a position where there are no options); if it is singular, the move to $G''$ is mandatory, and, in practice, the initial move causes a ``leap'' from $G$ to $G''$, while maintaining the right to play. Regarding carry-on moves, sometimes the rules of real-life games do not even consider the intervention of the opponent, as they might be a machine devoid of intelligence -- it is more common to say that a player can move from $G$ to $G''$, \emph{having to play again}. \cite{Lar021} elucidated the algebra of a generalization of SGT suitable for analyzing impartial games with entailing moves (LNST). We also suggest the motivational prelude in \cite{Lar024}.

Both SFPT and LNST are generalizations of SGT; however, they are independent generalizations. This means that neither of them can be used to analyze cyclic impartial games with entailing moves. A theory general enough to cover both cycles and all types of entailing moves would necessarily be very heavy. However, by restricting to just the case where the entailing moves are \emph{carry-on moves}, it is possible to propose an effective solution. We note that \cite{Sue24} constitutes a prior study in this regard. However, Section \ref{sec:carry}, the main contribution of this paper, presents a more comprehensive proposal, with Theorem \ref{th:cycliccarrysum}  being a more general result than what was presented in that previous work. Following that, Section \ref{sec:hackenbush} presents a case study application, the game {\sc green-lime hackenbush}.

Before moving forward, it is important to clarify the meaning we are assigning to the word \emph{theory}.
In Figure \ref{fig:fig1}, on the left, there is a digraph\footnote{This digraph has been used as an example several times by Aviezri Frankel -- see \cite{Con001}.} with three pieces, and, on the right, the same digraph exhibits the Grundy values assigned to each node. Since the digraph has cycles, SFPT is used to determine these values, as will be discussed and detailed later on. The fact that there are three pieces makes the position a disjunctive sum with three components. Note that the pieces do not conflict in any way, as they can share nodes. For now, the point is that a theory must be able to algorithmically assign a Grundy value to each node (1), provide a way to determine the Grundy value of a disjunctive sum from the Grundy values of the components (2), and provide a way to determine the outcome of a position from its Grundy value~(3). In this work, the theories are presented accordingly: an algorithm for Grundy values determination and algebra to handle disjunctive sums and outcomes. In examples similar to that of Figure \ref{fig:fig1}, the small numbers next to the nodes indicate the step of the algorithm in which a node is \emph{assigned} with data different from $\infty$. Typically, nodes remain unassigned during part of the algorithm, changing their state at some moment. Those that remain unassigned until the end maintain the symbol $\infty$, forming part of a cyclic zone. In the figures, bold nodes constitute cyclic zones and have infinite rank. In the position shown, the Grundy values of the components are $1$, $3$, and $\infty_{2,3}$, and the Grundy value of the disjunctive sum is $\infty_{0,1}$. Taking this into account, it can be concluded that the position is an $\mathcal{N}$-position, and the Next player wins by removing the piece with Grundy value $\infty_{2,3}$ from the cyclic zone with a move to the node whose Grundy value is $2$. This move achieves a $\mathcal{P}$-position. This analysis exemplifies the purpose of a theory like SFPT.

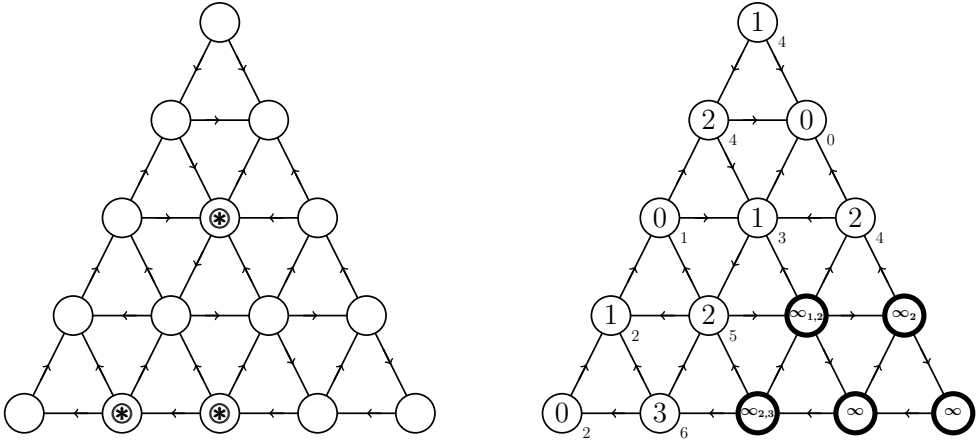
\begin{figure}[htb!]
\hspace{-0.65cm}
\scalebox{0.65}{
\begin{tikzpicture}
\clip(15.,-13.) rectangle (36.,-3.);
\draw (20.,-12.+0.35) node[anchor=north] {\scalebox{1.4}{$\bm{\circledast}$}};
\draw (18.,-12.+0.35) node[anchor=north] {\scalebox{1.4}{$\bm{\circledast}$}};
\draw (20.,-8.+0.35) node[anchor=north] {\scalebox{1.4}{$\bm{\circledast}$}};
\draw [line width=1.pt] (20.,-4.) circle (0.4cm);
\draw [line width=1.pt] (19.,-6.) circle (0.4cm);
\draw [line width=1.pt] (21.,-6.) circle (0.4cm);
\draw [line width=1.pt] (18.,-8.) circle (0.4cm);
\draw [line width=1.pt] (20.,-8.) circle (0.4cm);
\draw [line width=1.pt] (22.,-8.) circle (0.4cm);
\draw [line width=1.pt] (17.,-10.) circle (0.4cm);
\draw [line width=1.pt] (19.,-10.) circle (0.4cm);
\draw [line width=1.pt] (21.,-10.) circle (0.4cm);
\draw [line width=1.pt] (23.,-10.) circle (0.4cm);
\draw [line width=1.pt] (16.,-12.) circle (0.4cm);
\draw [line width=1.pt] (18.,-12.) circle (0.4cm);
\draw [line width=1.pt] (20.,-12.) circle (0.4cm);
\draw [line width=1.pt] (22.,-12.) circle (0.4cm);
\draw [line width=1.pt] (24.,-12.) circle (0.4cm);
\draw [->,line width=1.pt] (19.821114561799998,-4.357770876399992) -- (19.5,-5.);
\draw [line width=1.pt] (19.660557280899994,-4.678885438199998)-- (19.178885438199988,-5.642229123600017);
\draw [->,line width=1.pt] (20.17888543820001,-4.357770876400029) -- (20.5,-5.);
\draw [line width=1.pt] (20.339442719100006,-4.6788854382000284)-- (20.82111456180001,-5.642229123600028);
\draw [->,line width=1.pt] (19.4,-6.) -- (19.999999999999893,-6.);
\draw [line width=1.pt] (19.7,-6.)-- (20.59999999999983,-6.);
\draw [->,line width=1.pt] (18.178885438199984,-7.642229123600014) -- (18.5,-7.);
\draw [line width=1.pt] (18.339442719099992,-7.321114561800002)-- (18.82111456180001,-6.357770876399966);
\draw [->,line width=1.pt] (19.178885438199984,-6.357770876399977) -- (19.5,-7.);
\draw [line width=1.pt] (19.339442719099992,-6.678885438199994)-- (19.821114561800016,-7.642229123600044);
\draw [->,line width=1.pt] (20.178885438200005,-7.6422291235999795) -- (20.5,-7.);
\draw [line width=1.pt] (20.3394427191,-7.321114561799989)-- (20.821114561799988,-6.357770876400018);
\draw [->,line width=1.pt] (21.82111456179996,-7.642229123599932) -- (21.5,-7.);
\draw [line width=1.pt] (21.660557280899965,-7.321114561799947)-- (21.178885438199988,-6.357770876399991);
\draw [->,line width=1.pt] (18.4,-8.) -- (19.,-8.);
\draw [line width=1.pt] (18.7,-8.)-- (19.6,-8.);
\draw [->,line width=1.pt] (21.5999999999999,-8.) -- (21.,-8.);
\draw [line width=1.pt] (21.3,-8.)-- (20.4,-8.);
\draw [->,line width=1.pt] (19.4,-10.) -- (20.,-10.);
\draw [line width=1.pt] (19.7,-10.)-- (20.6,-10.);
\draw [->,line width=1.pt] (21.4,-10.) -- (22.,-10.);
\draw [line width=1.pt] (21.7,-10.)-- (22.599999999999895,-10.);
\draw [->,line width=1.pt] (18.6,-10.) -- (18.,-10.);
\draw [line width=1.pt] (18.3,-10.)-- (17.4,-10.);
\draw [->,line width=1.pt] (17.1788854382,-9.642229123599991) -- (17.5,-9.);
\draw [line width=1.pt] (17.3394427191,-9.321114561799984)-- (17.82111456180001,-8.357770876399966);
\draw [->,line width=1.pt] (18.821114561800023,-9.64222912360004) -- (18.5,-9.);
\draw [line width=1.pt] (18.66055728090002,-9.321114561800032)-- (18.178885438200005,-8.357770876400005);
\draw [->,line width=1.pt] (19.8211145618,-8.357770876399993) -- (19.5,-9.);
\draw [line width=1.pt] (19.6605572809,-8.678885438199993)-- (19.178885438200002,-9.642229123599993);
\draw [->,line width=1.pt] (20.82111456179996,-9.642229123599947) -- (20.5,-9.);
\draw [line width=1.pt] (20.660557280899965,-9.32111456179996)-- (20.178885438199988,-8.357770876399998);
\draw [->,line width=1.pt] (21.178885438200044,-9.642229123599884) -- (21.5,-9.);
\draw [line width=1.pt] (21.33944271910002,-9.321114561799934)-- (21.821114561799952,-8.35777087640008);
\draw [->,line width=1.pt] (22.82111456180002,-9.642229123600044) -- (22.5,-9.);
\draw [line width=1.pt] (22.660557280900022,-9.321114561800055)-- (22.17888543820004,-8.357770876400084);
\draw [->,line width=1.pt] (17.6,-12.) -- (17.,-12.);
\draw [line width=1.pt] (17.3,-12.)-- (16.4,-12.);
\draw [->,line width=1.pt] (19.5999999999999,-12.) -- (19.,-12.);
\draw [line width=1.pt] (19.3,-12.)-- (18.4,-12.);
\draw [->,line width=1.pt] (21.5999999999999,-12.) -- (21.,-12.);
\draw [line width=1.pt] (21.3,-12.)-- (20.4000000000001,-12.);
\draw [->,line width=1.pt] (23.6,-12.) -- (23.,-12.);
\draw [line width=1.pt] (23.3,-12.)-- (22.4,-12.);
\draw [->,line width=1.pt] (16.178885438199956,-11.642229123600096) -- (16.5,-11.);
\draw [line width=1.pt] (16.339442719099967,-11.321114561800073)-- (16.821114561799998,-10.357770876400007);
\draw [->,line width=1.pt] (17.821114561800012,-11.642229123600016) -- (17.5,-11.);
\draw [line width=1.pt] (17.660557280900008,-11.321114561800012)-- (17.178885438200005,-10.357770876400004);
\draw [->,line width=1.pt] (18.1788854382,-11.642229123599991) -- (18.5,-11.);
\draw [line width=1.pt] (18.3394427191,-11.321114561799988)-- (18.8211145618,-10.357770876399979);
\draw [->,line width=1.pt] (19.821114561799973,-11.642229123599943) -- (19.5,-11.);
\draw [line width=1.pt] (19.66055728089998,-11.32111456179996)-- (19.178885438200005,-10.357770876400007);
\draw [->,line width=1.pt] (20.17888543820005,-11.642229123599892) -- (20.5,-11.);
\draw [line width=1.pt] (20.339442719100028,-11.321114561799945)-- (20.821114561799945,-10.357770876400107);
\draw [->,line width=1.pt] (21.17888543820004,-10.357770876400082) -- (21.5,-11.);
\draw [line width=1.pt] (21.339442719100028,-10.678885438200053)-- (21.821114561799984,-11.642229123599968);
\draw [->,line width=1.pt] (22.17888543819997,-11.642229123600005) -- (22.5,-11.);
\draw [line width=1.pt] (22.339442719099985,-11.321114561799977)-- (22.821114561800023,-10.357770876399893);
\draw [->,line width=1.pt] (23.17888543819997,-10.35777087639993) -- (23.5,-11.);
\draw [line width=1.pt] (23.339442719099985,-10.678885438199964)-- (23.821114561800037,-11.642229123600066);
\draw [line width=1.pt] (31.,-4.) circle (0.4cm);
\draw [line width=1.pt] (30.,-6.) circle (0.4cm);
\draw [line width=1.pt] (32.,-6.) circle (0.4cm);
\draw [line width=1.pt] (29.,-8.) circle (0.4cm);
\draw [line width=1.pt] (31.,-8.) circle (0.4cm);
\draw [line width=1.pt] (33.,-8.) circle (0.4cm);
\draw [line width=1.pt] (28.,-10.) circle (0.4cm);
\draw [line width=1.pt] (30.,-10.) circle (0.4cm);
\draw [line width=3.pt] (32.,-10.) circle (0.4cm);
\draw [line width=3.pt] (34.,-10.) circle (0.4cm);
\draw [line width=1.pt] (27.,-12.) circle (0.4cm);
\draw [line width=1.pt] (29.,-12.) circle (0.4cm);
\draw [line width=3.pt] (31.,-12.) circle (0.4cm);
\draw [line width=3.pt] (33.,-12.) circle (0.4cm);
\draw [line width=3.pt] (35.,-12.) circle (0.4cm);
\draw [->,line width=1.pt] (30.821114561799973,-4.357770876400039) -- (30.5,-5.);
\draw [line width=1.pt] (30.660557280899983,-4.678885438200019)-- (30.178885438200016,-5.642229123599958);
\draw [->,line width=1.pt] (31.178885438200005,-4.357770876400021) -- (31.5,-5.);
\draw [line width=1.pt] (31.3394427191,-4.678885438200004)-- (31.821114561799977,-5.642229123599949);
\draw [->,line width=1.pt] (30.4000000000001,-6.) -- (31.,-6.);
\draw [line width=1.pt] (30.7,-6.)-- (31.599999999999895,-6.);
\draw [->,line width=1.pt] (29.178885438200012,-7.642229123599967) -- (29.5,-7.);
\draw [line width=1.pt] (29.339442719100003,-7.321114561799986)-- (29.821114561799973,-6.3577708764000445);
\draw [->,line width=1.pt] (30.178885438200034,-6.35777087640007) -- (30.5,-7.);
\draw [line width=1.pt] (30.339442719100017,-6.678885438200039)-- (30.82111456179997,-7.642229123599943);
\draw [->,line width=1.pt] (31.178885438200027,-7.642229123599896) -- (31.5,-7.);
\draw [line width=1.pt] (31.33944271910001,-7.321114561799936)-- (31.821114561799956,-6.357770876400057);
\draw [->,line width=1.pt] (32.821114561799966,-7.642229123599961) -- (32.5,-7.);
\draw [line width=1.pt] (32.66055728089998,-7.321114561799986)-- (32.17888543820002,-6.357770876400061);
\draw [->,line width=1.pt] (34.60000000000018,-12.) -- (34.00000000000014,-12.);
\draw [line width=1.pt] (34.30000000000016,-12.)-- (33.400000000000105,-12.);
\draw [->,line width=1.pt] (32.60000000000017,-12.) -- (32.000000000000135,-12.);
\draw [line width=1.pt] (32.30000000000015,-12.)-- (31.400000000000098,-12.);
\draw [->,line width=1.pt] (30.599999999999756,-12.) -- (29.999999999999858,-12.);
\draw [line width=1.pt] (30.299999999999805,-12.)-- (29.4,-12.);
\draw [->,line width=1.pt] (28.59999999999976,-12.) -- (27.999999999999787,-12.);
\draw [line width=1.pt] (28.299999999999773,-12.)-- (27.399999999999814,-12.);
\draw [->,line width=1.pt] (27.17888543819993,-11.642229123600083) -- (27.5,-11.);
\draw [line width=1.pt] (27.339442719099935,-11.321114561800073)-- (27.82111456179995,-10.357770876400046);
\draw [->,line width=1.pt] (28.821114561800012,-11.642229123600014) -- (28.5,-11.);
\draw [line width=1.pt] (28.660557280900015,-11.321114561800018)-- (28.178885438200023,-10.35777087640003);
\draw [->,line width=1.pt] (29.178885438200027,-11.64222912359994) -- (29.5,-11.);
\draw [line width=1.pt] (29.339442719100013,-11.32111456179997)-- (29.821114561799966,-10.357770876400057);
\draw [->,line width=1.pt] (30.821114561799977,-11.642229123599972) -- (30.5,-11.);
\draw [line width=1.pt] (30.66055728089998,-11.321114561799975)-- (30.178885438199984,-10.357770876399986);
\draw [->,line width=1.pt] (31.17888543820003,-11.642229123599911) -- (31.5,-11.);
\draw [line width=1.pt] (31.339442719100013,-11.321114561799947)-- (31.821114561799963,-10.357770876400052);
\draw [->,line width=1.pt] (32.17888543820003,-10.357770876400057) -- (32.5,-11.);
\draw [line width=1.pt] (32.33944271910001,-10.678885438200028)-- (32.82111456179997,-11.642229123599943);
\draw [->,line width=1.pt] (33.178885438200005,-11.642229123599964) -- (33.5,-10.999999999999844);
\draw [line width=1.pt] (33.339442719100035,-11.321114561799904)-- (33.82111456180013,-10.357770876399725);
\draw [->,line width=1.pt] (34.17888543819987,-10.357770876399757) -- (34.5,-11.);
\draw [line width=1.pt] (34.33944271909992,-10.678885438199856)-- (34.82111456180006,-11.642229123600151);
\draw [->,line width=1.pt] (29.5999999999999,-10.) -- (29.,-10.);
\draw [line width=1.pt] (29.3,-10.)-- (28.4000000000001,-10.);
\draw [->,line width=1.pt] (30.40000000000039,-10.) -- (31.000000000000142,-10.);
\draw [line width=1.pt] (30.700000000000266,-10.)-- (31.599999999999895,-10.);
\draw [->,line width=1.pt] (32.400000000000105,-10.) -- (33.00000000000014,-10.);
\draw [line width=1.pt] (32.700000000000124,-10.)-- (33.60000000000018,-10.);
\draw [->,line width=1.pt] (28.178885438200023,-9.642229123599941) -- (28.5,-9.);
\draw [line width=1.pt] (28.339442719100013,-9.321114561799963)-- (28.821114561799984,-8.35777087640003);
\draw [->,line width=1.pt] (29.821114561799977,-9.642229123599972) -- (29.5,-9.);
\draw [line width=1.pt] (29.660557280899994,-9.321114561799996)-- (29.178885438200034,-8.357770876400071);
\draw [->,line width=1.pt] (30.821114561799952,-8.357770876400098) -- (30.5,-9.);
\draw [line width=1.pt] (30.66055728089998,-8.678885438200046)-- (30.178885438200055,-9.642229123599892);
\draw [->,line width=1.pt] (31.821114561799952,-9.642229123599893) -- (31.5,-9.);
\draw [line width=1.pt] (31.66055728089997,-9.321114561799927)-- (31.178885438200016,-8.357770876400032);
\draw [->,line width=1.pt] (32.17888543820004,-9.642229123599918) -- (32.5,-9.);
\draw [line width=1.pt] (32.33944271910002,-9.321114561799952)-- (32.82111456179997,-8.357770876400057);
\draw [->,line width=1.pt] (33.821114561800094,-9.642229123600139) -- (33.5,-9.000000000000101);
\draw [line width=1.pt] (33.66055728090008,-9.321114561800119)-- (33.17888543820005,-8.357770876400064);
\draw [->,line width=1.pt] (32.600000000000456,-8.) -- (32.00000000000028,-8.);
\draw [line width=1.pt] (32.30000000000037,-8.)-- (31.400000000000098,-8.);
\draw [->,line width=1.pt] (29.400000000000105,-8.) -- (30.,-8.);
\draw [line width=1.pt] (29.7,-8.)-- (30.599999999999895,-8.);

\draw (31,-3.65) node[anchor=north] {\scalebox{1.42}{$1$}};

\draw (31-1,-3.65-2) node[anchor=north] {\scalebox{1.42}{$2$}};
\draw (31+1,-3.65-2) node[anchor=north] {\scalebox{1.42}{$0$}};

\draw (31-1-1,-3.65-2-2) node[anchor=north] {\scalebox{1.42}{$0$}};
\draw (31-1+1,-3.65-2-2) node[anchor=north] {\scalebox{1.42}{$1$}};
\draw (31-1+1+2,-3.65-2-2) node[anchor=north] {\scalebox{1.42}{$2$}};

\draw (31-1-1-1,-3.65-2-2-2) node[anchor=north] {\scalebox{1.42}{$1$}};
\draw (31-1-1+1,-3.65-2-2-2) node[anchor=north] {\scalebox{1.42}{$2$}};
\draw (31-1-1+1+2,-3.65-2-2-2.1) node[anchor=north] {\scalebox{0.75}{$\bm{\infty_{1,2}}$}};
\draw (31-1-1+1+2+2,-3.65-2-2-2.1) node[anchor=north] {\scalebox{0.75}{$\bm{\infty_{2}}$}};

\draw (31-1-1-1-1,-3.65-2-2-2-2) node[anchor=north] {\scalebox{1.42}{$0$}};
\draw (31-1-1-1+1,-3.65-2-2-2-2) node[anchor=north] {\scalebox{1.42}{$3$}};
\draw (31-1-1-1+1+2,-3.65-2-2-2-2.1) node[anchor=north] {\scalebox{0.75}{$\bm{\infty_{2,3}}$}};
\draw (31-1-1-1+1+2+2,-3.65-2-2-2-2.1) node[anchor=north] {\scalebox{0.75}{$\bm{\infty}$}};
\draw (31-1-1-1+1+2+2+2,-3.65-2-2-2-2.1) node[anchor=north] {\scalebox{0.75}{$\bm{\infty}$}};

\draw (31+0.5,-3.65-0.5) node[anchor=north] {\scalebox{0.8}{$4$}};

\draw (31-1+0.5,-3.65-2-0.5) node[anchor=north] {\scalebox{0.8}{$4$}};
\draw (31+1+0.5,-3.65-2-0.5) node[anchor=north] {\scalebox{0.8}{$0$}};

\draw (31-1-1+0.5,-3.65-2-2-0.5) node[anchor=north] {\scalebox{0.8}{$1$}};
\draw (31-1+1+0.5,-3.65-2-2-0.5) node[anchor=north] {\scalebox{0.8}{$3$}};
\draw (31-1+1+2+0.5,-3.65-2-2-0.5) node[anchor=north] {\scalebox{0.8}{$4$}};

\draw (31-1-1-1+0.5,-3.65-2-2-2-0.5) node[anchor=north] {\scalebox{0.8}{$2$}};
\draw (31-1-1+1+0.5,-3.65-2-2-2-0.5) node[anchor=north] {\scalebox{0.8}{$5$}};
\draw (31-1-1-1-1+0.5,-3.65-2-2-2-2-0.5) node[anchor=north] {\scalebox{0.8}{$2$}};
\draw (31-1-1-1+1+0.5,-3.65-2-2-2-2-0.5) node[anchor=north] {\scalebox{0.8}{$6$}};
\end{tikzpicture}}
    \caption{Fraenkel's digraph.}
    \label{fig:fig1}
\end{figure}

The following subsections delve into SGT, SFPT, and LNST, without including proofs for known results, as these can be found in the mentioned references. However, to combine the different theories, some terminology needs to be adjusted, adapting it to graph terminology and algorithmic presentation. We suggest that readers familiarize themselves with the approach used in these subsections, as it will enable them to better follow the original content that follows. Even though these are in the introductory section, the introduction effectively concludes here, and the core of the paper begins with them, especially since they already present some original results.

\subsection{Fundamental definitions and terminology}\label{subsec:def}

\vspace{0.4cm}
Usually, an impartial game $G$ is represented by a form $\{G_1,G_2,\ldots,G_k\}$ corresponding to the set of its available options; $G'$ is used to denote a generic option of $G$. However, here, we adopt a definition based on game digraphs. This approach offers the great advantage of allowing the precise specification of each theory through the restrictions on the game digraphs, i.e., describing which type of digraphs should be considered.

\begin{definition}\label{def:cycliccarrygame}
A \emph{cyclic impartial game with entailing moves} is a pair $G=(V,x)$, where $V$ is a finite digraph \emph{potentially with cycles} and $x$ is a node of $V$. The nodes can be \emph{white} or \emph{gray}.
\end{definition}

The idea behind the notation $(V,x)$ is precisely the gameplay with pieces, as mentioned before. In other words, $(V,x)$ means that there is a piece on the node $x$. The existence of white and gray nodes is due to the possibility of entailing moves. Whenever a player moves a piece to a gray node, \emph{the opponent is compelled to move that very same piece in response}. Naturally, entailing moves only gain a special status in disjunctive sums, which will be defined next (positions with more than one piece). It is clear that, when there is a single piece, the player must always move it on their turn. The possibility of choice arises only when there are multiple pieces, and this is where the color of the node comes into play. When a piece is moved to a gray node, the usual logic of the disjunctive sum is lost as the mandatory local response eliminates the possibility of choice. Before formalizing the disjunctive sum, it is important to establish some terminology.

\begin{definition}\label{def:options}
Given a game $G=(V,x)$, an \emph{instance} of $G$ is a game $\overline{G}=(V,x^*)$ that shares the underlying digraph $V$ with $G$. The \emph{options} of $G$ are games of the form $G'=(V,x')$, where there is an arc $\langle x,x'\rangle$ from $x$ to $x'$. The options of $G$, the options of options, etc., are the \emph{followers} of $G$, with the generic follower denoted by $G^\downarrow$.
\end{definition}

\noindent
\textbf{Notation:}
When writing $(G,x)$, $x$ can be either a white or a gray node. Whenever we wish to specify, we write $x^\circ$ for a white node and $x^{\textcolor{gray}{\bullet}}$ for a gray node\footnote{Strictly speaking, it would have been more accurate to write $G=(V_w,V_g,x)$. However, we prefer the shorter form $G=(V,x)$, assuming that $V=V_w\cup V_g$ and writing $x^\circ$ or $x^{\textcolor{gray}{\bullet}}$ whenever we want to specify the color of a vertex.}.

\begin{definition}\label{def:cartiproduct}
Let $V$ and $U$ be finite digraphs potentially with cycles with white and gray nodes. The \emph{cartproduct}\footnote{
In the case of short games, without cycles or gray nodes, this product is the cartesian product of digraphs, hence the choice of the name.} $V\ast U$ of the digraphs $V$ and $U$ is the digraph fulfilling the following:
\begin{enumerate}
  \item[\scalebox{0.5}{$\bullet$}] The vertex set $V\ast U$ is $$V\times U\backslash \{(x^{\textcolor{gray}{\bullet}},y^{\textcolor{gray}{\bullet}}):x^{\textcolor{gray}{\bullet}}\in V \text{ and }y^{\textcolor{gray}{\bullet}}\in U\},$$ where $\times$ denotes the cartesian product of the nodes of $V$ and $U$.
  \item[\scalebox{0.5}{$\bullet$}] Any two vertices $(x,y)$ and $(x^*,y^*)$ are connected by the edge $\langle (x,y),(x^*,y^*)\rangle$  in $V\ast U$ if and only if either $\langle x,x^*\rangle$ is an arc in $V$ and $y^\circ=y^{*\circ}$ or $\langle y,y^*\rangle$ is an arc in $U$ and $x^\circ=x^{*\circ}$.
\end{enumerate}
\end{definition}

\begin{observation}
The definition of cartproduct captures the core concept of the disjunctive sum, where, when faced with two pieces on the digraph, a player selects one to move ($\langle x,x^*\rangle$ is an arc in $V$), leaving the other untouched ($y^\circ=y^{*\circ}$). However, notice that $y^\circ=y^{*\circ}$ was written instead of $y=y^*$. This is because the piece that remains untouched must be in a white node, \emph{otherwise the player would be compelled to move that piece}. Also, regarding the vertex set, the pairs $(x^{\textcolor{gray}{\bullet}},y^{\textcolor{gray}{\bullet}})$ were excluded, since at no moment in the play can two pieces occupy two gray nodes.
\end{observation}

\begin{definition}\label{def:disjunctivesum}
Given two games $G=(V,x)$ and $H=(U,y)$ where the nodes $x$ and $y$ are not both gray, the disjunctive sum of $G$ and $H$ is the game $G+H=(V\ast U,(x,y))$.
\end{definition}

\begin{observation}
In practice, it is common to play games \emph{on the same board with more than one piece}, so
often, the cartproduct of interest is $V\ast V \ast V \ast\ldots$ However, it is important to emphasize that finding the cartproduct is a highly impractical process, due to its gigantism. Therefore, the goal should be to find the Grundy value of the sum from the Grundy values of the components without going through the process. In a way, that is the fundamental purpose of all the theories analyzed here. In Figure \ref{fig:fig2}, on the left, there is a digraph $V$ with two pieces. On the right, there is the cartproduct $V\ast V$, which consists of a digraph with a single piece. Playing the position on the left with the two pieces is, in practical terms, the same as playing the one on the right with just one. It is also important to observe that the options of a game $G+H$ are of the form $G'+H$ (provided the player is not entailed in $H$) or of the form $G+H'$ (provided the player is not entailed in $G$), as typically found in the standard definition of disjunctive sum.
\end{observation}

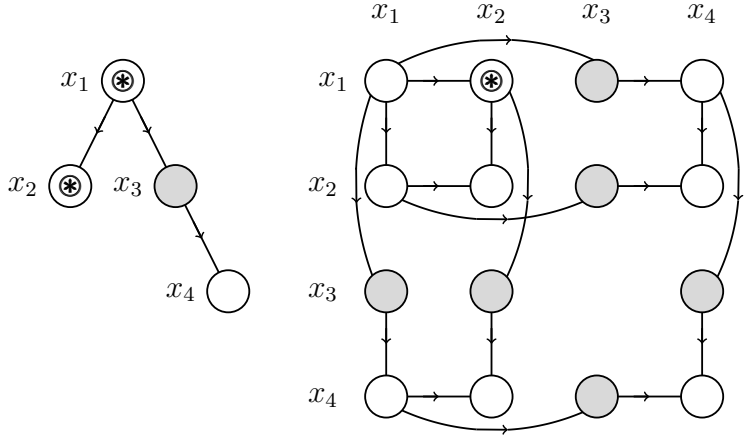
\begin{figure}[htb!]
\begin{center}
\scalebox{0.7}{
\begin{tikzpicture}
\clip(-7.3,-3.) rectangle (7.,5.7);
\draw (-6,2+0.35) node[anchor=north] {\scalebox{1.4}{$\bm{\circledast}$}};
\draw (-5,4+0.35) node[anchor=north] {\scalebox{1.4}{$\bm{\circledast}$}};
\draw (2,4+0.35) node[anchor=north] {\scalebox{1.4}{$\bm{\circledast}$}};
\draw [line width=1.pt] (0.,4.) circle (0.4cm);
\draw [line width=1.pt] (0.,2.) circle (0.4cm);
\draw [line width=1.pt,fill=gray!30] (0.,0.) circle (0.4cm);
\draw [line width=1.pt] (0.,-2.) circle (0.4cm);
\draw [line width=1.pt] (2.,4.) circle (0.4cm);
\draw [line width=1.pt] (2.,2.) circle (0.4cm);
\draw [line width=1.pt, fill=gray!30] (2.,0.) circle (0.4cm);
\draw [line width=1.pt] (2.,-2.) circle (0.4cm);
\draw [line width=1.pt, fill=gray!30] (4.,4.) circle (0.4cm);
\draw [line width=1.pt, fill=gray!30] (4.,2.) circle (0.4cm);
\draw [line width=1.pt, fill=gray!30] (4.,-2.) circle (0.4cm);
\draw [line width=1.pt] (6.,4.) circle (0.4cm);
\draw [line width=1.pt] (6.,2.) circle (0.4cm);
\draw [line width=1.pt,fill=gray!30] (6.,0.) circle (0.4cm);
\draw [line width=1.pt] (6.,-2.) circle (0.4cm);
\draw [->,line width=1.pt] (0.,3.6) -- (0.,3.);
\draw [line width=1.pt] (0.,3.3)-- (0.,2.4);
\draw [->,line width=1.pt] (0.,-0.4) -- (0.,-1.);
\draw [line width=1.pt] (0.,-0.7)-- (0.,-1.6);
\draw [->,line width=1.pt] (2.,3.6) -- (2.,3.);
\draw [line width=1.pt] (2.,3.3)-- (2.,2.4);
\draw [->,line width=1.pt] (2.,-0.4) -- (2.,-1.);
\draw [line width=1.pt] (2.,-0.7)-- (2.,-1.6);
\draw [->,line width=1.pt] (6.,3.6) -- (6.,3.);
\draw [line width=1.pt] (6.,3.3)-- (6.,2.4);
\draw [->,line width=1.pt] (6.,-0.4) -- (6.,-1.);
\draw [line width=1.pt] (6.,-0.7)-- (6.,-1.6);
\draw [->,line width=1.pt] (0.4,4.) -- (1.,4.);
\draw [line width=1.pt] (0.7,4.)-- (1.6,4.);
\draw [->,line width=1.pt] (0.4,-2.) -- (1.,-2.);
\draw [line width=1.pt] (0.7,-2.)-- (1.6,-2.);
\draw [->,line width=1.pt] (4.4,-2.) -- (5.,-2.);
\draw [line width=1.pt] (4.7,-2.)-- (5.6,-2.);
\draw [->,line width=1.pt] (4.4,2.) -- (5.,2.);
\draw [line width=1.pt] (4.7,2.)-- (5.6,2.);
\draw [->,line width=1.pt] (4.4,4.) -- (5.,4.);
\draw [line width=1.pt] (4.7,4.)-- (5.6,4.);
\draw [->,line width=1.pt] (1.77869855965319,4.766633715765787) -- (2.3782366253564273,4.777546053133112);
\draw [shift={(2.1568036630969463,0.468203864738077)},line width=1.pt]  plot[domain=1.6585340385614944:2.0340912994602194,variable=\t]({1.*4.315027537971956*cos(\t r)+0.*4.315027537971956*sin(\t r)},{0.*4.315027537971956*cos(\t r)+1.*4.315027537971956*sin(\t r)});
\draw [shift={(2.1568036630969463,0.46820386473808123)},line width=1.pt]  plot[domain=1.143899818877642:1.5194570797763676,variable=\t]({1.*4.315027537971952*cos(\t r)+0.*4.315027537971952*sin(\t r)},{0.*4.315027537971952*cos(\t r)+1.*4.315027537971952*sin(\t r)});
\draw [shift={(2.0558639671130794,5.654233019794052)},line width=1.pt]  plot[domain=1.6505853116845957:1.9920065613774014,variable=\t]({1.*4.293783219123118*cos(\t r)+0.*4.293783219123118*sin(\t r)},{0.*4.293783219123118*cos(\t r)+-1.*4.293783219123118*sin(\t r)});
\draw [shift={(2.055863967113083,5.654233019794038)},line width=1.pt]  plot[domain=1.1693987137032824:1.5108199633960897,variable=\t]({1.*4.293783219123103*cos(\t r)+0.*4.293783219123103*sin(\t r)},{0.*4.293783219123103*cos(\t r)+-1.*4.293783219123103*sin(\t r)});
\draw [->,line width=1.pt] (1.7136307581352108,1.3741102688094502) -- (2.3132351040004338,1.3681702075244289);
\draw [shift={(2.055863967113083,1.654233019794038)},line width=1.pt]  plot[domain=1.1693987137032824:1.5108199633960897,variable=\t]({1.*4.293783219123103*cos(\t r)+0.*4.293783219123103*sin(\t r)},{0.*4.293783219123103*cos(\t r)+-1.*4.293783219123103*sin(\t r)});
\draw [shift={(2.0558639671130794,1.6542330197940522)},line width=1.pt]  plot[domain=1.6505853116845957:1.9920065613774014,variable=\t]({1.*4.293783219123118*cos(\t r)+0.*4.293783219123118*sin(\t r)},{0.*4.293783219123118*cos(\t r)+-1.*4.293783219123118*sin(\t r)});
\draw [->,line width=1.pt] (1.7136307581352108,-2.62588973119055) -- (2.3132351040004338,-2.631829792475571);
\draw [->,line width=1.pt] (6.694767196028232,2.3284690422154872) -- (6.685958410775506,1.7289116059317438);
\draw [shift={(2.474227953487658,2.0906343918840906)},line width=1.pt]  plot[domain=0.05629218589618569:0.4147902774633528,variable=\t]({1.*4.227235126974038*cos(\t r)+0.*4.227235126974038*sin(\t r)},{0.*4.227235126974038*cos(\t r)+1.*4.227235126974038*sin(\t r)});
\draw [shift={(2.4742279534875826,2.0906343918840853)},line width=1.pt]  plot[domain=5.839012852187291:6.197510943754451,variable=\t]({1.*4.227235126974116*cos(\t r)+0.*4.227235126974116*sin(\t r)},{0.*4.227235126974116*cos(\t r)+1.*4.227235126974116*sin(\t r)});
\draw [shift={(4.465492327213462,2.057965703610385)},line width=1.pt]  plot[domain=3.210103419652631:3.4973193345667193,variable=\t]({1.*5.044726034097549*cos(\t r)+0.*5.044726034097549*sin(\t r)},{0.*5.044726034097549*cos(\t r)+1.*5.044726034097549*sin(\t r)});
\draw [shift={(4.4654923272134655,2.0579657036103747)},line width=1.pt]  plot[domain=2.803933880230609:3.0911497951446973,variable=\t]({1.*5.044726034097551*cos(\t r)+0.*5.044726034097551*sin(\t r)},{0.*5.044726034097551*cos(\t r)+1.*5.044726034097551*sin(\t r)});
\draw [->,line width=1.pt] (-0.5709176374859224,2.347509183209414) -- (-0.5648644334197765,1.6774770443141405);
\draw [line width=1.pt] (-3.,0.) circle (0.4cm);
\draw [line width=1.pt,fill=gray!30] (-4.,2.) circle (0.4cm);
\draw [line width=1.pt] (-5.,4.) circle (0.4cm);
\draw [line width=1.pt] (-6.,2.) circle (0.4cm);
\draw [->,line width=1.pt] (-5.178885438199986,3.642229123600029) -- (-5.5,3.);
\draw [line width=1.pt] (-5.339442719099994,3.3211145618000124)-- (-5.821114561800018,2.3577708763999645);
\draw [->,line width=1.pt] (-4.821114561800019,3.642229123600038) -- (-4.5,3.);
\draw [line width=1.pt] (-4.66055728090001,3.32111456180002)-- (-4.178885438199983,2.357770876399965);
\draw [->,line width=1.pt] (-3.8211145618000173,1.6422291236000341) -- (-3.5,1.);
\draw [line width=1.pt] (-3.660557280900009,1.3211145618000173)-- (-3.178885438199984,0.35777087639996685);
\draw [->,line width=1.pt] (0.4,2.) -- (1.,2.);
\draw [line width=1.pt] (0.7,2.)-- (1.6,2.);
\draw [->,line width=1.pt] (2.6947671960282316,2.328469042215484) -- (2.685958410775506,1.7289116059317464);
\draw [shift={(-1.525772046512397,2.0906343918840777)},line width=1.pt]  plot[domain=0.05629218589618727:0.41479027746335034,variable=\t]({1.*4.227235126974093*cos(\t r)+0.*4.227235126974093*sin(\t r)},{0.*4.227235126974093*cos(\t r)+1.*4.227235126974093*sin(\t r)});
\draw [shift={(-1.52577204651243,2.090634391884086)},line width=1.pt]  plot[domain=5.839012852187292:6.197510943754452,variable=\t]({1.*4.2272351269741275*cos(\t r)+0.*4.2272351269741275*sin(\t r)},{0.*4.2272351269741275*cos(\t r)+1.*4.2272351269741275*sin(\t r)});

\draw (-5.9,4.3) node[anchor=north] {\scalebox{1.42}{$x_1$}};
\draw (-5.9-1,4.3-2) node[anchor=north] {\scalebox{1.42}{$x_2$}};
\draw (-5.9+1,4.3-2) node[anchor=north] {\scalebox{1.42}{$x_3$}};
\draw (-5.9+1+1,4.3-2-2) node[anchor=north] {\scalebox{1.42}{$x_4$}};

\draw (-1,4.3) node[anchor=north] {\scalebox{1.42}{$x_1$}};
\draw (-1.2,4.3-2) node[anchor=north] {\scalebox{1.42}{$x_2$}};
\draw (-1.2,4.3-2-2) node[anchor=north] {\scalebox{1.42}{$x_3$}};
\draw (-1.2,4.3-2-2-2) node[anchor=north] {\scalebox{1.42}{$x_4$}};
\draw (-1+1,4.3+1.3) node[anchor=north] {\scalebox{1.42}{$x_1$}};
\draw (-1+1+2,4.3+1.3) node[anchor=north] {\scalebox{1.42}{$x_2$}};
\draw (-1+1+2+2,4.3+1.3) node[anchor=north] {\scalebox{1.42}{$x_3$}};
\draw (-1+1+2+2+2,4.3+1.3) node[anchor=north] {\scalebox{1.42}{$x_4$}};
\end{tikzpicture}}
\end{center}
    \caption{On the left: a digraph $V$ with pieces at nodes $x_1$ and $x_2$; on the right, the game ($V\ast V, (x_1,x_2))$.}
    \label{fig:fig2}
\end{figure}

Regarding the outcomes, the definition is as usual. When all options are losing moves, the $\mathscr{P}$revious player wins; when there is at least one winning option, the $\mathscr{N}$ext player wins; when there are no winning moves but at least one option results in a tie, the game is a $\mathscr{D}$raw.\\

\begin{definition}\label{def:outcomes}
If $G$ is an impartial game, then the \emph{outcome} of $G$, denoted by $o(G)$, is determined by
\begin{enumerate}
  \item[\scalebox{0.5}{$\bullet$}] $o(G)=\mathscr{P}$ if $G$ is terminal or if all options $G'$ have outcome $\mathscr{N}$;
  \item[\scalebox{0.5}{$\bullet$}] $o(G)=\mathscr{N}$ if there exists an option $G'$ with outcome $\mathscr{P}$;
  \item[\scalebox{0.5}{$\bullet$}] $o(G)=\mathscr{D}$ if there are no options $G'$ with outcome $\mathscr{P}$, but there exists an option $G'$ with outcome $\mathscr{D}$.
\end{enumerate}
The outcome classes $\mathcal{P},\mathcal{N}$, and $\mathcal{D}$ are the sets of all games with
the respective outcome, so that we write $G\in \mathcal{P}$ when $o(G)=\mathscr{P}$, calling it a $\mathcal{P}$-\emph{position}.
\end{definition}

\subsection{Short impartial games}

\vspace{0.4cm}
The classical restriction and the first to be studied concerns \emph{short impartial games}, corresponding to digraphs without cycles or gray nodes \citep{Bou02,Spra35,Gru39}. {\sc nim} \citep{Bou02} and {\sc wythoff nim} \citep{Wyt07} are examples of real-life short impartial games, as long as they are played with a finite number of stones. Given such a digraph $V$, the algorithm for determining the Grundy values of the games $G=(V,x)$ makes use of the \emph{mex rule} (minimum excluded value), which is a set function that returns the smallest nonnegative integer that does not belong to the set
(Algorithm \ref{alg:Gshort}, exemplified in Figure \ref{fig:fig3}). The mex rule is used in all theories analyzed in this document.

\begin{algorithm}[htb!]
\caption{Grundy values: short impartial games}\label{alg:Gshort}
\begin{algorithmic}

\Require finite digraph $V$ without cycles or gray nodes\\
\State \textbf{Step 0:}\Comment{Initialization}\\

\vspace{-0.3cm}
\If{$G=(V,x)$ is a game such that $x$ is terminal} $\mathcal{G}_0(G)=0$
\Else{} $\mathcal{G}_0(G)=\infty$
\EndIf\\

\State \textbf{Step n+1:}\\

\vspace{-0.3cm}
\If{all options of $G=(V,x)$ have a finite Grundy value assigned, i.e., if $\mathcal{G}_{n}(G')<\infty$ for every option $G'$} $\mathcal{G}_{n+1}(G)=\mex\{\mathcal{G}_{n}(G')\}$
\EndIf\\

\vspace{-0.3cm}
\If{$\mathcal{G}_{n+1}(G)=\mathcal{G}_{n}(G)$ for every node $x$} \textbf{End}
\EndIf \Comment{Ending condition}\\\\

\Return the Grundy value of each game $G=(V,x)$ is $\mathcal{G}(G)=\mathcal{G}_{t}(G)$, where step $t$ is the final step
\end{algorithmic}
\end{algorithm}

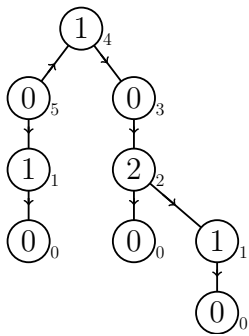
\begin{figure}[htb!]
\begin{center}
\scalebox{0.7}{
\begin{tikzpicture}
\clip(-2.5,-1.8) rectangle (2.4,4.42);
\draw [line width=1.pt] (-1.,4.) circle (0.4cm);
\draw [line width=1.pt] (-2.,2.68) circle (0.4cm);
\draw [line width=1.pt] (0.,2.68) circle (0.4cm);
\draw [->,line width=1.pt] (-0.7584569248018114,3.6811631407383913) -- (-0.5,3.34);
\draw [line width=1.pt] (-0.6292284624009042,3.510581570369194)-- (-0.24154307519818285,2.9988368592616017);
\draw [->,line width=1.pt] (-1.7584569248018158,2.9988368592616035) -- (-1.5,3.34);
\draw [line width=1.pt] (-1.6292284624009088,3.1694184296308006)-- (-1.2415430751981877,3.6811631407383913);
\draw [line width=1.pt] (-2.,1.32) circle (0.4cm);
\draw [line width=1.pt] (0.,1.32) circle (0.4cm);
\draw [line width=1.pt] (-2.,-0.04) circle (0.4cm);
\draw [line width=1.pt] (0.,-0.04) circle (0.4cm);
\draw [line width=1.pt] (1.58,-0.04) circle (0.4cm);
\draw [line width=1.pt] (1.58,-1.4) circle (0.4cm);
\draw [->,line width=1.pt] (-2.,2.28) -- (-2.,2.);
\draw [line width=1.pt] (-2.,2.14)-- (-2.,1.72);
\draw [->,line width=1.pt] (-2.,0.92) -- (-2.,0.64);
\draw [line width=1.pt] (-2.,0.78)-- (-2.,0.36);
\draw [->,line width=1.pt] (0.,2.28) -- (0.,2.);
\draw [line width=1.pt] (0.,2.14)-- (0.,1.72);
\draw [->,line width=1.pt] (0.,0.92) -- (0.,0.64);
\draw [line width=1.pt] (0.,0.78)-- (0.,0.36);
\draw [->,line width=1.pt] (0.30316022260619674,1.0590519602883373) -- (0.79,0.64);
\draw [line width=1.pt] (0.5465801113030977,0.8495259801441692)-- (1.2768397773938007,0.2209480397116652);
\draw [->,line width=1.pt] (1.58,-0.44) -- (1.58,-0.72);
\draw [line width=1.pt] (1.58,-0.58)-- (1.58,-1.);
\draw (-2.,0.24+0.07) node[anchor=north] {\scalebox{1.42}{$0$}};
\draw (0.,0.24+0.07) node[anchor=north] {\scalebox{1.42}{$0$}};
\draw (1.58,-1.12+0.07) node[anchor=north] {\scalebox{1.42}{$0$}};
\draw (-2.,1.6+0.07) node[anchor=north] {\scalebox{1.42}{$1$}};
\draw (1.58,0.24+0.07) node[anchor=north] {\scalebox{1.42}{$1$}};
\draw (0.,1.6+0.07) node[anchor=north] {\scalebox{1.42}{$2$}};
\draw (0.,2.96+0.07) node[anchor=north] {\scalebox{1.42}{$0$}};
\draw (-2.,2.96+0.07) node[anchor=north] {\scalebox{1.42}{$0$}};
\draw (-1.,4.28+0.07) node[anchor=north] {\scalebox{1.42}{$1$}};

\draw (-2.+0.5,0.24-0.3+0.07) node[anchor=north] {\scalebox{0.8}{$0$}};
\draw (0.+0.5,0.24+0.07-0.3) node[anchor=north] {\scalebox{0.8}{$0$}};
\draw (1.58+0.5,-1.12+0.07-0.3) node[anchor=north] {\scalebox{0.8}{$0$}};
\draw (-2.+0.5,1.6+0.07-0.3) node[anchor=north] {\scalebox{0.8}{$1$}};
\draw (1.58+0.5,0.24+0.07-0.3) node[anchor=north] {\scalebox{0.8}{$1$}};
\draw (0.+0.5,1.6+0.07-0.3) node[anchor=north] {\scalebox{0.8}{$2$}};
\draw (0.+0.5,2.96+0.07-0.3) node[anchor=north] {\scalebox{0.8}{$3$}};
\draw (-2.+0.5,2.96+0.07-0.3) node[anchor=north] {\scalebox{0.8}{$5$}};
\draw (-1.+0.5,4.28+0.07-0.3) node[anchor=north] {\scalebox{0.8}{$4$}};
\end{tikzpicture}}
\end{center}
\vspace{-0.5cm}
    \caption{Example of Algorithm \ref{alg:Gshort} execution.}
    \label{fig:fig3}
\end{figure}

Regarding the disjunctive sum, the fundamental theorem makes use of the {\sc nim}-\emph{sum}, denoted by $\oplus$, which is the bitwise exclusive (or {\sc xor}, or addition without carrying) of the binary representations of the summands (nonnegative integers). For example, $7\oplus 14=9$ ($7=_2 111$ and $14=_2 1110$). The {\sc nim}-sum is used in all theories analyzed in this document.

\begin{theorem}[$\mathcal{G}$-value of a disjunctive sum of short impartial games]\label{th:classicalsum}$\,$\\
If $G=(V,x^\circ)$ and $H=(U,y^\circ)$ are two short impartial games, then\linebreak $\mathcal{G}(G+H)=\mathcal{G}(G)\oplus\mathcal{G}(H)$.
\end{theorem}

Regarding outcomes, it is worth noting that, in the absence of cycles, the possibility of a game ending in a tie ceases to exist. The outcome classes are $\mathcal{P}$ and $\mathcal{N}$, defined by the first two items of Definition~\ref{def:outcomes}. The link between the Grundy value of a game and its outcome can be described as follows.

\begin{theorem}[Outcome of a short impartial game with a given $\mathcal{G}$-value]\label{th:classicaloutcome} $\,$\\
If $G=(V,x^\circ)$ is a short impartial game, then
\begin{enumerate}
  \item[\scalebox{0.5}{$\bullet$}] $o(G)=\mathscr{P}$ if and only if $\mathcal{G}(G)=0$;
  \item[\scalebox{0.5}{$\bullet$}] $o(G)=\mathscr{N}$ if and only if $\mathcal{G}(G)\neq 0$.
\end{enumerate}
\end{theorem}

\subsection{Cyclic impartial games}\label{subsec:cyclic}

\vspace{0.4cm}
If a game digraph without gray nodes has cycles, then we have a \emph{cyclic impartial game} \citep{Smi66,Fra73}. {\sc additional subtraction games} and {\sc adders-and-ladders} \citep{BCG001} are examples of real-life cyclic impartial games. Given such a digraph $V$, the mex rule alone is no longer sufficient; it becomes necessary to apply simultaneously the mex rule and reversibility. Expanding further, during the execution of the algorithm, a finite Grundy value is assigned to $G=(V,x)$ only if the minimum excluded value of the ``already known'' finite Grundy numbers of the options of $G$ is precisely that value, and every other move reverts to a position with that value.

Additionally, after the algorithm finishes, there may be nodes that retain their initial $\infty$, constituting cyclic zones. In regard to these nodes, it is essential to retain information about the possible \emph{exits} from the cyclic zones. This information \emph{can make the difference between tying or winning a game}, as a player may opt to exit a cycle if it leads to victory. Hence, the Grundy values must contain this information, which explains the output of Algorithm~\ref{alg:Gcyclic}.

\begin{algorithm}[htb!]
\caption{Grundy values: cyclic impartial games}\label{alg:Gcyclic}
\begin{algorithmic}

\Require finite digraph $V$ with cycles without gray nodes\\

\vspace{-0.3cm}
\State \textbf{Step 0:}\Comment{Initialization}\\

\vspace{-0.3cm}
\If{$G=(V,x)$ is a game such that $x$ is terminal} $\mathcal{G}_0(G)=0$
\Else{} $\mathcal{G}_0(G)=\infty$
\EndIf\\

\vspace{-0.3cm}
\State \textbf{Step n+1:} for $G=(V,x)$ under analysis, let $m=\mex\{\mathcal{G}_{n}(G')\}$\\

\vspace{-0.3cm}
\If{any option of $G$ whose assigned value is higher than $m$ reverts to $m$, i.e., if for each $G'$ such that $\mathcal{G}_{n}(G')>m$ there is $G''$ with $\mathcal{G}_{n}(G'')=m$} $\mathcal{G}_{n+1}(G)=m$
\EndIf\\

\vspace{-0.3cm}
\If{$\mathcal{G}_{n+1}(G)=\mathcal{G}_{n}(G)$ for every node $x$} \textbf{End}
\EndIf \Comment{Ending condition}\\\\

\Return for each $G=(V,x)$, and the step $t$ being the final step,
\If{$\mathcal{G}_{t}(G)<\infty$} the (non-cyclic) Grundy value of the game $G$ is $\mathcal{G}(G)=\mathcal{G}_{t}(G)$
\Else{} the (cyclic) Grundy value of $G$ is $\mathcal{G}(G)=\infty_{\mathbf{D}}$, where $\mathbf{D}$ is the set $\{a\in\mathbb{Z}_0^{+}:\mathcal{G}(G')=a\text{ for some option }G'\}$
\EndIf
\end{algorithmic}
\end{algorithm}

\vspace{-0.3cm}
As for the disjunctive sums, when there are two or more pieces in cyclic zones, it is impossible to exit, as a player cannot make more than one move in their turn. However, if there is only one piece in a cyclic zone, a ``winning exit'' is a possibility. In this document, if $m\in \mathbb{Z}_0^{+}$ and $S$ is a set, then $m\oplus S=\{m\oplus s\,:\,s\in S\}$. It is clear that if $S$ is empty, $m\oplus S$ is empty as well. The same logic applies to $S\oplus S'$, where $S$ and $S'$ are two sets\footnote{Here, $\mathbb{Z}_0^{+}$ denotes the nonnegative integers.}.

\begin{figure}[htb!]
\begin{center}
\scalebox{0.7}{
\begin{tikzpicture}
\clip(-4.5,-1.9) rectangle (2.5,4.5);
\draw [line width=3.pt] (-1.,4.) circle (0.4cm);
\draw [line width=3.pt] (-2.,2.68) circle (0.4cm);
\draw [line width=3.pt] (0.,2.68) circle (0.4cm);
\draw [->,line width=1.pt] (-0.7584569248018114,3.6811631407383913) -- (-0.5,3.34);
\draw [line width=1.pt] (-0.6292284624009042,3.510581570369194)-- (-0.24154307519818285,2.9988368592616017);
\draw [->,line width=1.pt] (-1.7584569248018158,2.9988368592616035) -- (-1.5,3.34);
\draw [line width=1.pt] (-1.6292284624009088,3.1694184296308006)-- (-1.2415430751981877,3.6811631407383913);
\draw [line width=1.pt] (-2.,1.32) circle (0.4cm);
\draw [line width=1.pt] (0.,1.32) circle (0.4cm);
\draw [line width=1.pt] (-2.,-0.04) circle (0.4cm);
\draw [line width=1.pt] (0.,-0.04) circle (0.4cm);
\draw [line width=1.pt] (1.58,-0.04) circle (0.4cm);
\draw [line width=1.pt] (1.58,-1.4) circle (0.4cm);
\draw [->,line width=1.pt] (-2.,2.28) -- (-2.,2.);
\draw [line width=1.pt] (-2.,2.14)-- (-2.,1.72);
\draw [->,line width=1.pt] (-2.,0.92) -- (-2.,0.64);
\draw [line width=1.pt] (-2.,0.78)-- (-2.,0.36);
\draw [->,line width=1.pt] (0.,2.28) -- (0.,2.);
\draw [line width=1.pt] (0.,2.14)-- (0.,1.72);
\draw [->,line width=1.pt] (0.,0.92) -- (0.,0.64);
\draw [line width=1.pt] (0.,0.78)-- (0.,0.36);
\draw [->,line width=1.pt] (0.30316022260619674,1.0590519602883373) -- (0.79,0.64);
\draw [line width=1.pt] (0.5465801113030977,0.8495259801441692)-- (1.2768397773938007,0.2209480397116652);
\draw [->,line width=1.pt] (1.58,-0.44) -- (1.58,-0.72);
\draw [line width=1.pt] (1.58,-0.58)-- (1.58,-1.);
\draw (-2.,0.24+0.07) node[anchor=north] {\scalebox{1.42}{$0$}};
\draw (0.,0.24+0.07) node[anchor=north] {\scalebox{1.42}{$0$}};
\draw (1.58,-1.12+0.07) node[anchor=north] {\scalebox{1.42}{$0$}};
\draw (-2.,1.6+0.07) node[anchor=north] {\scalebox{1.42}{$1$}};
\draw (1.58,0.24+0.07) node[anchor=north] {\scalebox{1.42}{$1$}};
\draw (0.,1.6+0.07) node[anchor=north] {\scalebox{1.42}{$2$}};
\draw (-2.,2.96-0.04) node[anchor=north] {\scalebox{1}{$\infty$}};
\draw (0.,2.96-0.04) node[anchor=north] {\scalebox{1}{$\infty$}};
\draw (-1.,4.28-0.04) node[anchor=north] {\scalebox{1}{$\infty$}};

\draw (-2.+0.5,0.24-0.3+0.07) node[anchor=north] {\scalebox{0.8}{$0$}};
\draw (0.+0.5,0.24+0.07-0.3) node[anchor=north] {\scalebox{0.8}{$0$}};
\draw (1.58+0.5,-1.12+0.07-0.3) node[anchor=north] {\scalebox{0.8}{$0$}};
\draw (-2.+0.5,1.6+0.07-0.3) node[anchor=north] {\scalebox{0.8}{$1$}};
\draw (1.58+0.5,0.24+0.07-0.3) node[anchor=north] {\scalebox{0.8}{$1$}};
\draw (0.+0.5,1.6+0.07-0.3) node[anchor=north] {\scalebox{0.8}{$2$}};

\draw [->,line width=1.pt] (-0.4,2.68) -- (-1.,2.68);
\draw [line width=1.pt] (-0.7,2.68)-- (-1.6,2.68);

\draw [line width=1.pt] (-4.,1.32) circle (0.4cm);
\draw [->,line width=1.pt] (-3.669229399172189,1.5449240085629117) -- (-3.,2.);
\draw [line width=1.pt] (-3.3346146995860946,1.7724620042814558)-- (-2.3307706008278113,2.4550759914370888);
\draw [->,line width=1.pt] (-3.6692293991721843,1.0950759914370856) -- (-3.,0.64);
\draw [line width=1.pt] (-3.3346146995860915,0.8675379957185423)-- (-2.330770600827812,0.18492400856291247);
\draw (-4.,1.6+0.07) node[anchor=north] {\scalebox{1.42}{$1$}};
\draw (-4.+0.55,1.6+0.07-0.25) node[anchor=north] {\scalebox{0.8}{$2$}};
\end{tikzpicture}}
\scalebox{0.7}{
\begin{tikzpicture}
\clip(-4.5,-1.9) rectangle (2.5,4.5);
\draw [line width=3.pt] (-1.,4.) circle (0.4cm);
\draw [line width=3.pt] (-2.,2.68) circle (0.4cm);
\draw [line width=3.pt] (0.,2.68) circle (0.4cm);
\draw [->,line width=1.pt] (-0.7584569248018114,3.6811631407383913) -- (-0.5,3.34);
\draw [line width=1.pt] (-0.6292284624009042,3.510581570369194)-- (-0.24154307519818285,2.9988368592616017);
\draw [->,line width=1.pt] (-1.7584569248018158,2.9988368592616035) -- (-1.5,3.34);
\draw [line width=1.pt] (-1.6292284624009088,3.1694184296308006)-- (-1.2415430751981877,3.6811631407383913);
\draw [line width=1.pt] (-2.,1.32) circle (0.4cm);
\draw [line width=1.pt] (0.,1.32) circle (0.4cm);
\draw [line width=1.pt] (-2.,-0.04) circle (0.4cm);
\draw [line width=1.pt] (0.,-0.04) circle (0.4cm);
\draw [line width=1.pt] (1.58,-0.04) circle (0.4cm);
\draw [line width=1.pt] (1.58,-1.4) circle (0.4cm);
\draw [->,line width=1.pt] (-2.,2.28) -- (-2.,2.);
\draw [line width=1.pt] (-2.,2.14)-- (-2.,1.72);
\draw [->,line width=1.pt] (-2.,0.92) -- (-2.,0.64);
\draw [line width=1.pt] (-2.,0.78)-- (-2.,0.36);
\draw [->,line width=1.pt] (0.,2.28) -- (0.,2.);
\draw [line width=1.pt] (0.,2.14)-- (0.,1.72);
\draw [->,line width=1.pt] (0.,0.92) -- (0.,0.64);
\draw [line width=1.pt] (0.,0.78)-- (0.,0.36);
\draw [->,line width=1.pt] (0.30316022260619674,1.0590519602883373) -- (0.79,0.64);
\draw [line width=1.pt] (0.5465801113030977,0.8495259801441692)-- (1.2768397773938007,0.2209480397116652);
\draw [->,line width=1.pt] (1.58,-0.44) -- (1.58,-0.72);
\draw [line width=1.pt] (1.58,-0.58)-- (1.58,-1.);
\draw (-2.,0.24+0.07) node[anchor=north] {\scalebox{1.42}{$0$}};
\draw (0.,0.24+0.07) node[anchor=north] {\scalebox{1.42}{$0$}};
\draw (1.58,-1.12+0.07) node[anchor=north] {\scalebox{1.42}{$0$}};
\draw (-2.,1.6+0.07) node[anchor=north] {\scalebox{1.42}{$1$}};
\draw (1.58,0.24+0.07) node[anchor=north] {\scalebox{1.42}{$1$}};
\draw (0.,1.6+0.07) node[anchor=north] {\scalebox{1.42}{$2$}};
\draw (-2.,2.96-0.04) node[anchor=north] {\scalebox{1}{$\infty_1$}};
\draw (0.,2.96-0.04) node[anchor=north] {\scalebox{1}{$\infty_2$}};
\draw (-1.,4.28-0.04) node[anchor=north] {\scalebox{1}{$\infty$}};

\draw [->,line width=1.pt] (-0.4,2.68) -- (-1.,2.68);
\draw [line width=1.pt] (-0.7,2.68)-- (-1.6,2.68);

\draw [line width=1.pt] (-4.,1.32) circle (0.4cm);
\draw [->,line width=1.pt] (-3.669229399172189,1.5449240085629117) -- (-3.,2.);
\draw [line width=1.pt] (-3.3346146995860946,1.7724620042814558)-- (-2.3307706008278113,2.4550759914370888);
\draw [->,line width=1.pt] (-3.6692293991721843,1.0950759914370856) -- (-3.,0.64);
\draw [line width=1.pt] (-3.3346146995860915,0.8675379957185423)-- (-2.330770600827812,0.18492400856291247);
\draw (-4.,1.6+0.07) node[anchor=north] {\scalebox{1.42}{$1$}};
\end{tikzpicture}}
\end{center}
\vspace{-0.5cm}
    \caption{Example of Algorithm \ref{alg:Gcyclic} execution.
On the left, the state after the ending condition is satisfied (third step); on the right, the assigned $\mathcal{G}$-values.}
    \label{fig:fig4}
\end{figure}
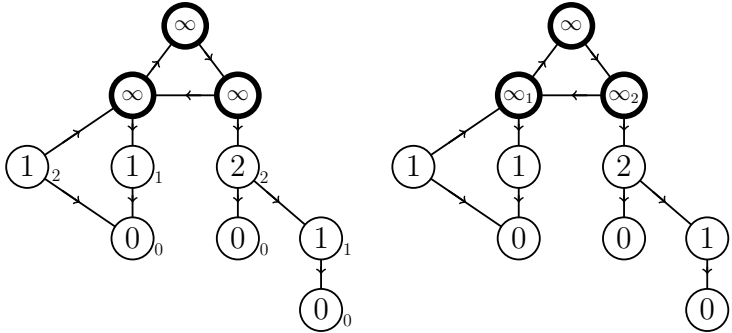

\begin{theorem}[$\mathcal{G}$-value of a disjunctive sum of cyclic impartial games]\label{th:cyclicsum} If $G=(V,x^\circ)$ and $H=(U,y^\circ)$ are two cyclic impartial games, then $$\mathcal{G}(G+H)=\begin{cases}
m\oplus k \quad\quad\,\,\,\,\textrm{ if }\mathcal{G}(G)=m<\infty\textrm{ and }\mathcal{G}(H)=k<\infty;\\
\infty_{m\oplus\mathbf{D}} \quad\quad\,\,\,\,\textrm{ if }\mathcal{G}(G)=m<\infty\textrm{ and }\mathcal{G}(H)=\infty_{\mathbf{D}};\\
\infty \quad\quad\quad\quad\,\,\textrm{ otherwise.}
\end{cases}$$
\end{theorem}

Regarding outcomes, the potential existence of cycles allows ties to occur, making all three outcomes from Definition~\ref{def:outcomes} possible. A tie only occurs when the Grundy value of a game $G=(V,x)$ is $\mathcal{G}(G)=\infty_{\mathbf{D}}$ and $0\not \in \mathbf{D}$. This is because $x$ is in a cyclic zone and there is no winning exit.\\

\begin{theorem}[Outcome of a cyclic impartial game with a given $\mathcal{G}$-value]\label{th:cyclicoutcome} If $G=(V,x^\circ)$ is a cyclic impartial game, then
\begin{enumerate}
  \item[\scalebox{0.5}{$\bullet$}] $o(G)=\mathscr{P}$ if and only if $\mathcal{G}(G)=0$;
  \item[\scalebox{0.5}{$\bullet$}] $o(G)=\mathscr{N}$ if and only if $\mathcal{G}(G)\in\mathbb{Z}^{+}$ or $\mathcal{G}(G)=\infty_\mathbf{D}$ where $0\in\mathbf{D}$;
  \item[\scalebox{0.5}{$\bullet$}] $o(G)=\mathscr{D}$ if and only if $\mathcal{G}(G)=\infty_{\mathbf{D}}$ and $0\not\in\mathbf{D}$.
\end{enumerate}
\end{theorem}

\subsection{Acyclic impartial games with entailing moves}\label{subsec:entailing}

\vspace{0.4cm}
If a game digraph without cycles has gray nodes, then we have an \emph{acyclic impartial game with entailing moves} \citep{Lar021}. {\sc nimstring} \citep{Ber000} and {\sc top entails} \citep{BCG001} are examples of real-life games of this type. As in the case of cyclic games, the mex rule alone is no longer sufficient to determine the Grundy values of the positions. However, the reason for this lies not in the occurrence of cycles, but rather in the occurrence of forcing sequences. Regarding short games, and using the ordinary nimbers, when a player has a winning move in $G+*m$, it follows that $G=*k\neq *m$, and the winning option is to $*k+*k$ or to $*m+*m$, depending on whether $m>k$ or $k>m$. When entailing moves come into play, there is a third possibility: the winning strategy may be a sequence of entailing moves initiated from $G$, forcing that, at a given moment, one of two things occur in that component; either an opponent's move to a nimber other than $*m$ or a player's move to $*m+*m$. When this type of strategy exists, we say that there is \emph{protection} against the presence of $*m$ in the disjunctive sum (or, for short, just ``against $m$''), meaning that there is a forcing sequence from $G$ that guarantees a win in $G+*m$. The set containing all the non-harmful integers $m$ under these conditions is denoted by $\mathbf{P}(G)$. Thus, in the case of acyclic impartial games with entailing moves, the mex rule should be applied not only to $\{\mathcal{G}(G'): G'=(V,{x'}^\circ)\}$, but to the union of this set with $\mathbf{P}(G)$.

\begin{definition}\label{def:protected}
Let $G$ be a game as defined in Definition \ref{def:cycliccarrygame}. The \emph{protected values} in $G$, $\mathbf{P}(G)$, are determined in the following away:
\begin{enumerate}
  \item [\scalebox{0.7}{$\bullet$}] If $G$ has a terminal gray option $G'=(V,{x'}^{\textcolor{gray}{\bullet}})$, then $\mathbf{P}(G)=\mathbb{Z}_0^{+}$;
  \item [\scalebox{0.7}{$\bullet$}] Otherwise, $m\in \mathbf{P}(G)$ if $G$ has a nonterminal gray option $G'=(V,{x'}^{\textcolor{gray}{\bullet}})$ such that, for every $G''$, at least one of the following items is satisfied:
      \begin{enumerate}
        \item [\scalebox{0.5}{$\bullet$}] $\mathcal{G}(G'')=k<\infty$ and $k\oplus m\neq 0$;
        \item [\scalebox{0.5}{$\bullet$}] there is a $G'''$ such that $\mathcal{G}(G''')=m$;
        \item [\scalebox{0.5}{$\bullet$}] $m\in \mathbf{P}(G'')$.
      \end{enumerate}
\end{enumerate}
\end{definition}

Consider the game $G=(V,x)$ in Figure \ref{fig:fig5}, where $x$ is the node where a single piece is located. For example, there is protection against the presence of $*5$. In $G+*5$, the first move is the entailing option. If the response is to the node whose Grundy value is $1$, the first subitem of the second item is satisfied. If the response is to the unlabeled node, there is again protection against the presence of $*5$ (via the second entailing move), and the third subitem of the second item is satisfied. As another example, observe also that there is protection against the presence of $0$ ($G$ played in isolation). In $G$, the first move is again the entailing option. If the response is to the node whose Grundy value is $1$, the first subitem of the second item is again satisfied. If the response is to the unlabeled node, it follows a move to the node whose Grundy value is $0$, and the second subitem of the second item is satisfied. It is easy to check that $\mathbf{P}(G)=\mathbb{Z}_0^{+}\backslash \{1\}$ and, as we will detail next, $\mathcal{G}(G)=\mex(\{2\}\,\cup\,\mathbb{Z}_0^{+}\backslash \{1\})=1$. The unlabeled white node has a Grundy value specific to this theory, which we will discuss shortly.

\begin{figure}[htb!]
\begin{center}
\scalebox{0.7}{
\begin{tikzpicture}
\clip(-4.7,-3.6) rectangle (5.7,3.6);
\draw (-1,3+0.35) node[anchor=north] {\scalebox{1.4}{$\bm{\circledast}$}};
\draw [line width=1.pt] (-1.,3.) circle (0.4cm);
\draw [line width=1.pt] (-2.,1.) circle (0.4cm);
\draw [line width=1.pt, fill=gray!30] (0.,1.) circle (0.4cm);
\draw [line width=1.pt] (-1.,-1.) circle (0.4cm);
\draw [line width=1.pt] (-2.,-3.) circle (0.4cm);
\draw [line width=1.pt] (1.,-1.) circle (0.4cm);
\draw [line width=1.pt] (0.,-3.) circle (0.4cm);
\draw [line width=1.pt,fill=gray!30] (3.,-2.) circle (0.4cm);
\draw [line width=1.pt] (5.,-3.) circle (0.4cm);
\draw [line width=1.pt] (-3.,-1.) circle (0.4cm);
\draw [line width=1.pt] (-4.,-3.) circle (0.4cm);
\draw [line width=1.pt] (-4.,0.) circle (0.4cm);
\draw [->,line width=1.pt] (-1.1788854381999847,2.642229123600032) -- (-1.5,2.);
\draw [line width=1.pt] (-1.339442719099993,2.3211145618000155)-- (-1.8211145618000169,1.3577708763999665);
\draw [->,line width=1.pt] (-2.3577708763999667,0.8211145618000166) -- (-3.,0.5);
\draw [line width=1.pt] (-2.6788854381999836,0.6605572809000082)-- (-3.642229123600034,0.17888543819998293);
\draw [->,line width=1.pt] (-2.178885438199983,0.6422291236000335) -- (-2.5,0.);
\draw [line width=1.pt] (-2.3394427190999916,0.32111456180001685)-- (-2.8211145618000164,-0.642229123600033);
\draw [->,line width=1.pt] (-3.1788854381999836,-1.3577708763999663) -- (-3.5,-2.);
\draw [line width=1.pt] (-3.339442719099992,-1.6788854381999823)-- (-3.821114561800016,-2.6422291236000297);
\draw [->,line width=1.pt] (-0.8211145618000159,2.642229123600031) -- (-0.5,2.);
\draw [line width=1.pt] (-0.6605572809000078,2.321114561800015)-- (-0.17888543819998326,1.3577708763999663);
\draw [->,line width=1.pt] (-0.17888543819998326,0.6422291236000335) -- (-0.5,0.);
\draw [line width=1.pt] (-0.33944271909999163,0.32111456180001674)-- (-0.8211145618000167,-0.6422291236000335);
\draw [->,line width=1.pt] (-1.1788854381999834,-1.3577708763999665) -- (-1.5,-2.);
\draw [line width=1.pt] (-1.3394427190999916,-1.6788854381999831)-- (-1.8211145618000169,-2.642229123600033);
\draw [->,line width=1.pt] (0.1788854381999832,0.6422291236000337) -- (0.5,0.);
\draw [line width=1.pt] (0.3394427190999916,0.3211145618000169)-- (0.8211145618000169,-0.6422291236000335);
\draw [->,line width=1.pt] (0.8211145618000171,-1.3577708763999667) -- (0.5,-2.);
\draw [line width=1.pt] (0.6605572809000084,-1.6788854381999838)-- (0.17888543819998248,-2.6422291236000355);
\draw [->,line width=1.pt] (1.3577708763999663,-1.1788854381999836) -- (2.,-1.5);
\draw [line width=1.pt] (1.6788854381999831,-1.339442719099992)-- (2.6422291236000337,-1.8211145618000173);
\draw [->,line width=1.pt] (3.3577708763999663,-2.178885438199983) -- (4.,-2.5);
\draw [line width=1.pt] (3.678885438199984,-2.339442719099992)-- (4.642229123600037,-2.8211145618000186);

\draw (-4,0.24+0.1) node[anchor=north] {\scalebox{1.42}{$0$}};
\draw (-2,0.24+0.1+1) node[anchor=north] {\scalebox{1.42}{$2$}};
\draw (-2-1,0.24+0.1+1-2) node[anchor=north] {\scalebox{1.42}{$1$}};
\draw (-2-1-1,0.24+0.1+1-2-2) node[anchor=north] {\scalebox{1.42}{$0$}};
\draw (-2-1-1+2,0.24+0.1+1-2-2) node[anchor=north] {\scalebox{1.42}{$0$}};
\draw (-2-1-1+2+2,0.24+0.1+1-2-2) node[anchor=north] {\scalebox{1.42}{$0$}};
\draw (-2-1-1+2+1,0.24+0.1+1-2) node[anchor=north] {\scalebox{1.42}{$1$}};
\draw (-2-1-1+2+2+5,0.24+0.1+1-2-2) node[anchor=north] {\scalebox{1.42}{$0$}};
\end{tikzpicture}}
\end{center}
\vspace{-0.5cm}
    \caption{A game $G=(V,x)$ whose set of protected values $\mathbf{P}(G)$ is $\mathbb{Z}_0^{+}\backslash \{1\}$.}
    \label{fig:fig5}
\end{figure}
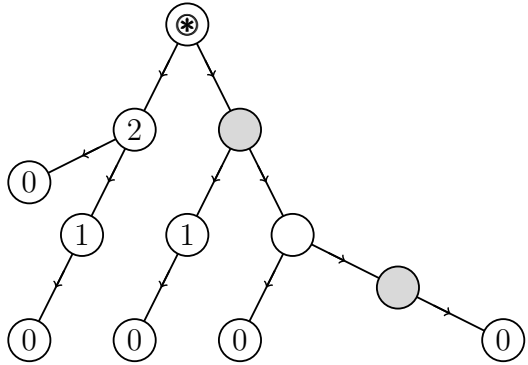

The following function, to be used in step $n+1$, allows to check whether an entailing move ensures protection for the presence of $*k$. Theorem~\ref{def:protected3} ensures that the said function only needs to be used a finite number of times in Algorithm~\ref{alg:Gentailing}. Since Theorem~\ref{def:protected3} is not included in \cite{Lar021}, it is presented here with a proof. The logic of the function is clear. If, when an entailing move is made, there are only responses to white nodes whose Grundy values are different from $k$ and there responses to gray nodes that do not provide protection for the presence of $*k$, then it is a satisfactory move, offering protection for the presence of $*k$.

Afterward, Algorithm \ref{alg:Gentailing} for determining Grundy values is presented (exemplified in Figure \ref{fig:fig6}). White nodes are assigned with Grundy values and gray nodes are assigned with sets of covered values (Definition \ref{def:protected2}). Grundy values can be either nonnegative integers or the special value $\moon\!$ whose presence in a disjunctive sum makes it an $\mathcal{N}$-position. In \cite{Lar021}, $\moon\!\!$ can be a game form like $\{0,\{\infty\,|\,0\}\,|\,0,\{0\,|\,\overline{\infty}\}\}$, with the symbols $\infty$ and $\overline{\infty}$ being used with a different meaning. Here, we are combining theories, and it is more appropriate for $\moon\!\!$ to be considered a Grundy value.

\begin{algorithm}[htb!]
\caption{{\sc boolean} function Cr to evaluate whether an entailing move (a move to a gray node) results in $k$ being a protected value}\label{alg:protected}
\begin{algorithmic}

\Require an integer $k\geqslant 0$ and $G=(V,x^{\textcolor{gray}{\bullet}})$ whose options are previously assigned\\
\If{$x^{\textcolor{gray}{\bullet}}$ is a terminal node} \Return {\sc true} $\rightarrow$ Cr$(G,k)$
\Else{}
\If{for all $G'$, $\mathcal{G}_n(G'=(V,{x'}^\circ))\neq k$ or Cr$(G'=(V,{x'}^{\textcolor{gray}{\bullet}}),k)=${\sc false},}
    \Return {\sc true} $\rightarrow$ Cr$(G,k)$
\Else{} \Return {\sc false} $\rightarrow$ Cr$(G,k)$
\EndIf
\EndIf
\end{algorithmic}
\end{algorithm}

\begin{definition}\label{def:protected2}
Let $G=(V,x^{\textcolor{gray}{\bullet}})$ with the options assigned by step $n$. The set of \emph{covered values} by $G$ is given by $\mathcal{C}_{n+1}(G)=\{k\in \mathbb{Z}_0^{+}\,|\,\text{Cr}(G,k)=${\sc true}$\}$.
\end{definition}

\begin{theorem}\label{def:protected3}
If $G$ is a game whose options are previously assigned, then $\mathcal{C}_{n+1}(G)$ is either finite or infinite of the form of $\mathbb{Z}_0^{+}\backslash F$, where $F$ is finite. We also have the following:
\begin{enumerate}
  \item [\scalebox{0.5}{$\bullet$}] If there is at least one $G'=(V,{x'}^{\textcolor{gray}{\bullet}})$ such that $\mathcal{C}_{n}(G')$ is infinite, then let $l=\max(F_1\cup\ldots\cup F_j)$ where $\mathbb{Z}_0^{+}\backslash F_1,\ldots,\mathbb{Z}_0^{+}\backslash F_j$ are the infinite sets with respect to the gray options of $G$, and we have \emph{Cr}$(G,k)=${\sc false} for all $k>l$.
  \item [\scalebox{0.5}{$\bullet$}] Otherwise, let $l$ be the least nonnegative integer exceeding all integers in
  
  \vspace{-0.6cm}
  $$\cup_{G'=(V,{x'}^{\textcolor{gray}{\bullet}})} \mathcal{C}_{n}(G')\,\,\cup\,\,\{\mathcal{G}_{n}(G')\,|\,G'=(V,{x'}^\circ)\},$$

  \vspace{-0.2cm}
  and we have \emph{Cr}$(G,k)=${\sc true} for all $k\geqslant l$.
\end{enumerate}
\end{theorem}

\begin{proof}
If $x^{\textcolor{gray}{\bullet}}$ is terminal, we have $\mathcal{C}_{n+1}(G)=\mathbb{Z}_0^{+}$ and the result is true. If the first item is verified then, by induction, $G$ has at least one gray option $G'$ such that $\mathcal{C}_{n}(G')=\mathbb{Z}_0^{+}\backslash F$. Moreover, if $k>l$, then Cr$(G',k)=${\sc true}. Hence, Cr$(G,k)=${\sc false}. In this case, $\mathcal{C}_{n+1}(G)$ is finite, since this set does not contain any element of $\{l+1,l+2,\ldots\}$. If the second item is verified and $k\geqslant l$, then, if $G'=(V,{x'}^\circ)$, we have $\mathcal{G}_{n}(G')\neq k$, and, if $G'=(V,{x'}^{\textcolor{gray}{\bullet}})$, we have Cr$(G',k)=${\sc false}. Thus, Cr$(G,k)=${\sc true} and $\{l,l+1,\ldots\}\subseteq\mathcal{C}_{n+1}(G)$. In this case, we have an infinite set in the described form.
\end{proof}
\clearpage
\begin{algorithm}[htb!]
\caption{Grundy values: acyclic impartial games with entailing moves}\label{alg:Gentailing}
\begin{algorithmic}

\Require acyclic digraph $V$ with gray nodes\\

\vspace{-0.3cm}
\State \textbf{Step 0:}\Comment{Initialization}\\

\vspace{-0.3cm}
\If{$G=(V,x^\circ)$ is a game such that $x^\circ$ is a terminal node} $\mathcal{G}_0(G)=0$
\Else{}, if $x^\circ$ is not terminal, $\mathcal{G}_0(G)=\infty$
\EndIf
\If{$G=(V,x^{\textcolor{gray}{\bullet}})$ is a game such that $x^{\textcolor{gray}{\bullet}}$ is a terminal node} $\mathcal{C}_0(G)=\mathbb{Z}_0^{+}$
\Else{}, if $x^{\textcolor{gray}{\bullet}}$ is not terminal, $\mathcal{C}_0(G)=\infty$
\EndIf\\

\vspace{-0.3cm}
\State \textbf{Step n+1:}\\

\vspace{-0.3cm}
\If{$G=(V,x^\circ)$ is a game whose options are previously assigned}
\State \textbf{let} $D=\{\mathcal{G}_{n}(G'): G'=(V,{x'}^\circ)\}$ and \textbf{let} $P=\cup_{G'=(V,{x'}^{\textcolor{gray}{\bullet}})}\,\mathcal{C}_{n}(G')$
\If{$\mathbb{Z}_0^{+}\subseteq D\cup P$} $\mathcal{G}_{n+1}(G)=\moon$
\Else{} $\mathcal{G}_{n+1}(G)=\mex(D\cup P)$
\EndIf
\EndIf\\

\vspace{-0.3cm}
\If{$G=(V,x^{\textcolor{gray}{\bullet}})$ is a game whose options are previously assigned}
\State \textbf{let} $\mathcal{C}_{n+1}(G)=\emptyset$ and \textbf{let} $k=0$
\If{there is at least one $G'=(V,{x'}^{\textcolor{gray}{\bullet}})$ such that $\mathcal{C}_{n}(G')$ is infinite}
\While{$k\leqslant l$ (Theorem \ref{def:protected3})}
    \If{Cr$(G,k)=${\sc true}} \textbf{Add}$(\mathcal{C}_{n+1}(G),k)$
    \EndIf
\EndWhile
\Else{}
\While{$k\leqslant l$ (Theorem \ref{def:protected3})}
    \If{Cr$(G,k)=${\sc true}} \textbf{Add}$(\mathcal{C}_{n+1}(G),k)$
    \EndIf
\EndWhile
\State \textbf{Append}$(\mathcal{C}_{n+1}(G),\{l+1,l+2,\ldots\})$
\EndIf
\EndIf\\

\vspace{-0.3cm}
\If{$\mathcal{G}_{n+1}(G)=\mathcal{G}_{n}(G)$ and $\mathcal{C}_{n+1}(G)=\mathcal{C}_{n}(G)$, $\forall G=(V,x)$} \textbf{End}
\EndIf \Comment{Ending condition}\\\\

\vspace{-0.3cm}
\Return Given $G=(V,x^\circ)$, the Grundy value of $G$ is $\mathcal{G}(G)=\mathcal{G}_{t}(G)$, where step $t$ is the final step
\end{algorithmic}
\end{algorithm}

\clearpage

\begin{figure}[htb!]
\begin{center}
\scalebox{0.7}{
\begin{tikzpicture}
\clip(-2.5,-1.9) rectangle (2.4,4.5);
\draw [line width=1.pt] (-1.,4.) circle (0.4cm);
\draw [line width=1.pt, fill=gray!30] (-2.,2.68) circle (0.4cm);
\draw [line width=1.pt] (0.,2.68) circle (0.4cm);
\draw [->,line width=1.pt] (-0.7584569248018114,3.6811631407383913) -- (-0.5,3.34);
\draw [line width=1.pt] (-0.6292284624009042,3.510581570369194)-- (-0.24154307519818285,2.9988368592616017);
\draw [->,line width=1.pt] (-1.2415430751981877,3.6811631407383913) -- (-1.6292284624009088,3.34);
\draw [line width=1.pt] (-1.6292284624009088+0.015,3.34+0.015)-- (-1.6292284624009088-0.29,3.08);
\draw [line width=1.pt] (-2.,1.32) circle (0.4cm);
\draw [line width=1.pt] (0.,1.32) circle (0.4cm);
\draw [line width=1.pt] (-2.,-0.04) circle (0.4cm);
\draw [line width=1.pt] (0.,-0.04) circle (0.4cm);
\draw [line width=1.pt, fill=gray!30] (1.58,-0.04) circle (0.4cm);
\draw [line width=1.pt] (1.58,-1.4) circle (0.4cm);
\draw [->,line width=1.pt] (-2.,2.28) -- (-2.,2.);
\draw [line width=1.pt] (-2.,2.14)-- (-2.,1.72);
\draw [->,line width=1.pt] (-2.,0.92) -- (-2.,0.64);
\draw [line width=1.pt] (-2.,0.78)-- (-2.,0.36);
\draw [->,line width=1.pt] (0.,2.28) -- (0.,2.);
\draw [line width=1.pt] (0.,2.14)-- (0.,1.72);
\draw [->,line width=1.pt] (0.,0.92) -- (0.,0.64);
\draw [line width=1.pt] (0.,0.78)-- (0.,0.36);
\draw [->,line width=1.pt] (0.30316022260619674,1.0590519602883373) -- (0.79,0.64);
\draw [line width=1.pt] (0.5465801113030977,0.8495259801441692)-- (1.2768397773938007,0.2209480397116652);
\draw [->,line width=1.pt] (1.58,-0.44) -- (1.58,-0.72);
\draw [line width=1.pt] (1.58,-0.58)-- (1.58,-1.);
\draw (-2.,0.24+0.07) node[anchor=north] {\scalebox{1.42}{$0$}};
\draw (0.,0.24+0.07) node[anchor=north] {\scalebox{1.42}{$0$}};
\draw (1.58,-1.12+0.07) node[anchor=north] {\scalebox{1.42}{$0$}};
\draw (-2.,1.6+0.07) node[anchor=north] {\scalebox{1.42}{$1$}};
\draw (0.05,1.63+0.07) node[anchor=north] {\scalebox{1.42}{$\moon$}};
\draw (0.,2.96+0.07) node[anchor=north] {\scalebox{1.42}{$0$}};
\draw (-1.,4.28+0.07) node[anchor=north] {\scalebox{1.42}{$1$}};
\draw (1.55,0.24+0.08) node[anchor=north] {\scalebox{1.42}{$A$}};
\draw (-2.-0.02,2.96+0.08) node[anchor=north] {\scalebox{1.42}{$B$}};

\draw (-2.+0.5,0.24-0.3+0.07) node[anchor=north] {\scalebox{0.8}{$0$}};
\draw (0.+0.5,0.24+0.07-0.3) node[anchor=north] {\scalebox{0.8}{$0$}};
\draw (1.58+0.5,-1.12+0.07-0.3) node[anchor=north] {\scalebox{0.8}{$0$}};
\draw (-2.+0.5,1.6+0.07-0.3) node[anchor=north] {\scalebox{0.8}{$1$}};
\draw (0.+0.5,1.6+0.07-0.3) node[anchor=north] {\scalebox{0.8}{$2$}};
\draw (0.+0.5,2.96+0.07-0.3) node[anchor=north] {\scalebox{0.8}{$3$}};
\draw (-1.+0.5,4.28+0.07-0.3) node[anchor=north] {\scalebox{0.8}{$4$}};
\draw (1.58+0.5,0.24+0.07-0.3) node[anchor=north] {\scalebox{0.8}{$1$}};
\draw (-2+0.5,2.96+0.07-0.3) node[anchor=north] {\scalebox{0.8}{$2$}};

\end{tikzpicture}}
\end{center}
\vspace{-0.5cm}
    \caption{Example of Algorithm \ref{alg:Gentailing} execution: $A=\mathbb{Z}_0^{+}\backslash \{0\}$ and $B=\mathbb{Z}_0^{+}\backslash \{1\}$.}
    \label{fig:fig6}
\end{figure}
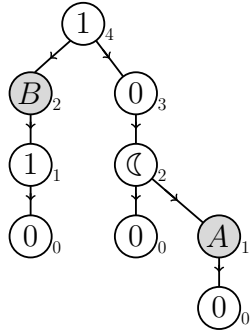

Theorem \ref{th:entailingsum}, concerning the disjunctive sum of two games $G=(V,x^\circ)$ and $H=(V,y^{\circ})$, expresses the fact that $\moon\!$ is absorbing. Theorem~\ref{th:entailingoutcome}, concerning the outcome of a game $G=(V,x^\circ)$, asserts what was expected.

\begin{theorem}[$\mathcal{G}$-value of a disjunctive sum of acyclic impartial games with entailing moves]\label{th:entailingsum} If $G=(V,x^\circ)$ and $H=(V,y^{\circ})$ are two games, then $$\mathcal{G}(G+H)=\begin{cases}
m\oplus k \quad\quad\,\,\,\textrm{ if }\mathcal{G}(G)=m\textrm{ and }\mathcal{G}(H)=k\!;\\
\moon \quad\quad\quad\quad\,\,\textrm{ otherwise.}
\end{cases}$$
\end{theorem}

\begin{theorem}[Outcome of an acyclic impartial game with entailing moves with a given $\mathcal{G}$-value]\label{th:entailingoutcome}If $G=(V,x^\circ)$ is a game, then
\begin{enumerate}
  \item[\scalebox{0.5}{$\bullet$}] $o(G)=\mathscr{P}$ if and only if $\mathcal{G}(G)=0$;
  \item[\scalebox{0.5}{$\bullet$}] $o(G)=\mathscr{N}$ if and only if $\mathcal{G}(G)\neq 0$.
\end{enumerate}
\end{theorem}

\section{Cyclic impartial games with carry-on moves}
\label{sec:carry}

If a game digraph with cycles has gray nodes with \emph{the outdegree of them being at most one}, then we have a \emph{cyclic impartial game with carry-on moves}. Observe that when a gray node has exactly one outgoing edge, that is a carry-on move, as when a player uses it, the opponent responds in a mandatory way and the player \emph{moves again}. The mex rule and the {\sc nim}-sum continue to play a crucial role, as the non-cyclic zones of the digraph behaves in exactly the same way as in previous theories.

Regarding the assignment of white nodes, once again, the mex rule is applied to the union of the Grundy values of the already assigned white options with the Grundy values protected by the already assigned gray options. As one would expect, given the potential occurrence of carry-on moves, the concept of protection used in LNST is needed. Furthermore, regarding the unassigned options, similar to SFPT, it is necessary to check reversibility. For example, when the mex rule points to an integer $m$, if for any unassigned option, there is an immediate response to a white node assigned with $m$ or an immediate response to a gray node that offers \emph{protection for the presence of} $m$, reversibility happens. This second possibility is not considered in SFPT, and in this new theory, the boolean function Cr has to be used. However, given the potential occurrence of cycles, the part ``whose options are previously assigned'' should be removed. Additionally, it should be established that {\sc false} occurs when {\sc true} \emph{is not reached prior to an infinite run}. The function adapted for the case of digraphs with cycles and carry-on moves is as follows (the meaning of $\fullmoon$ will be clarified soon).

\begin{algorithm}[htb!]
\caption{{\sc boolean} function Cr to evaluate whether a carry-on move (a move to a gray node) results in $k$ being a protected value}\label{alg:protected}
\begin{algorithmic}

\Require an integer $k\geqslant 0$ and $G=(V,x^{\textcolor{gray}{\bullet}})$\\
\If{$x^{\textcolor{gray}{\bullet}}$ is a terminal node} \Return {\sc true} $\rightarrow$ Cr$(G,k)$
\Else{}
\If{all options $G'$ fall into one of the following cases:
\State i) $G'=(V,{x'}^\circ)$ and $\mathcal{G}_n(G')\in \{\moon,\fullmoon\}\cup\mathbb{Z}_0^{+}\backslash \{k\}$
\State ii) there is some $G''=(V,{x''}^\circ)$ such that $\mathcal{G}_n(G'')=k$
\State iii) there is some $G''=(V,{x''}^{\textcolor{gray}{\bullet}})$ such that Cr$(G'',k)=${\sc true}}
    \Return {\sc true} $\rightarrow$ Cr$(G,k)$
\Else{} \Return {\sc false} $\rightarrow$ Cr$(G,k)$
\EndIf \quad\Comment{The condition may fail due to an infinite run; in this case, the function returns {\sc false} and not an error message}
\EndIf
\end{algorithmic}
\end{algorithm}

In Figure \ref{fig:fig7}, considering $G=(V,x^\circ)$ where $x^\circ$ is the node where a single piece is placed, the assigned options point to the minimum excluded value $1$. Moreover, it is straightforward to verify that, regarding the only unassigned option, it has a carry-on move that provides protection for the presence of $1$, i.e., it reverts to a node whose Grundy value is $1$ after a forcing sequence. Thus, $G$ must be assigned with the Grundy value $1$. We will see later that the unlabeled nodes form a cyclic zone. In this example, reversibility is a kind of ``hit and run'', a carry-on move is made within a cyclic zone and then a subsequent move exits it immediately.

\begin{figure}[htb!]
\begin{center}
\scalebox{0.7}{
\begin{tikzpicture}
\clip(-1.5,-0.6) rectangle (9.5,4.5);
\draw (1,4+0.35) node[anchor=north] {\scalebox{1.4}{$\bm{\circledast}$}};
\draw [line width=1.pt] (1.,4.) circle (0.4cm);
\draw [line width=1.pt] (-1.,2.) circle (0.4cm);
\draw [line width=1.pt, fill=gray!30] (1.,2.) circle (0.4cm);
\draw [line width=1.pt] (1.,0.) circle (0.4cm);
\draw [line width=1.pt] (-1.,0.) circle (0.4cm);
\draw [line width=3.pt] (3.,4.) circle (0.4cm);
\draw [line width=3.pt, fill=gray!30] (5.,4.) circle (0.4cm);
\draw [line width=1.pt] (5.+2,2.) circle (0.4cm);
\draw [line width=1.pt] (5.+2,0.) circle (0.4cm);
\draw [line width=3.pt] (7.,4.) circle (0.4cm);
\draw [line width=3.pt] (9.,2.) circle (0.4cm);
\draw [line width=3.pt] (9.,4.) circle (0.4cm);
\draw [->,line width=1.pt] (0.7171572875253812,3.717157287525381) -- (0.,3.);
\draw [line width=1.pt] (0.3585786437626906,3.3585786437626908)-- (-0.7171572875253812,2.282842712474619);
\draw [->,line width=1.pt] (1.,3.6) -- (1.,3.);
\draw [line width=1.pt] (1.,3.3)-- (1.,2.4);
\draw [->,line width=1.pt] (1.,1.6) -- (1.,1.);
\draw [line width=1.pt] (1.,1.3)-- (1.,0.4);
\draw [->,line width=1.pt] (0.6,0.) -- (0.,0.);
\draw [line width=1.pt] (0.3,0.)-- (-0.6,0.);
\draw [->,line width=1.pt] (1.4,4.) -- (2.,4.);
\draw [line width=1.pt] (1.7,4.)-- (2.6,4.);
\draw [->,line width=1.pt] (3.4,4.) -- (4.,4.);
\draw [line width=1.pt] (3.7,4.)-- (4.6,4.);
\draw [->,line width=1.pt] (5.+2,3.6) -- (5.+2,3.);
\draw [line width=1.pt] (5.+2,3.3)-- (5.+2,2.4);
\draw [->,line width=1.pt] (5.+2,1.6) -- (5.+2,1.);
\draw [line width=1.pt] (5.+2,1.3)-- (5.+2,0.4);
\draw [->,line width=1.pt] (5.4,4.) -- (6.,4.);
\draw [line width=1.pt] (5.7,4.)-- (6.6,4.);
\draw [->,line width=1.pt] (7.2828427124746105,3.71715728752539) -- (8.,3.);
\draw [line width=1.pt] (7.641421356237302,3.358578643762698)-- (8.717157287525376,2.282842712474622);
\draw [->,line width=1.pt] (9.,2.4) -- (9.,3.);
\draw [line width=1.pt] (9.,2.7)-- (9.,3.6);
\draw [->,line width=1.pt] (8.6,4.) -- (8.,4.);
\draw [line width=1.pt] (8.3,4.)-- (7.4,4.);

\draw (-1,2+0.33) node[anchor=north] {\scalebox{1.42}{$0$}};
\draw (-1,2+0.33-2) node[anchor=north] {\scalebox{1.42}{$0$}};
\draw (-1+2-0.02,2+0.33+0.02) node[anchor=north] {\scalebox{1.42}{$A$}};
\draw (-1+2,2+0.33-2) node[anchor=north] {\scalebox{1.42}{$1$}};
\draw (-1+2+4+2,2+0.33) node[anchor=north] {\scalebox{1.42}{$1$}};
\draw (-1+2+4+2,2+0.33-2) node[anchor=north] {\scalebox{1.42}{$0$}};
\end{tikzpicture}}
\end{center}
\vspace{-0.5cm}
    \caption{A game $G=(V,x^\circ)$ whose Grundy value is $1$ ($A=\mathbb{Z}_0^{+}\backslash \{1\}$).}
    \label{fig:fig7}
\end{figure}
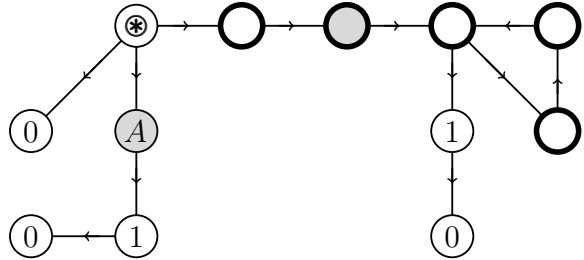

In the theory of cyclic impartial games with carry-on moves, absorbing situations can occur. When a player has to play at a terminal gray node, they immediately lose regardless of what other components may be, because \emph{it is mandatory to play} and that node has no outgoing edges. These gray nodes are assigned with the value $\newmoon$ (``new moon''); only a gray node can be a new moon. Hence, when a player has a new moon as an option, they immediately win regardless of what other components may be. A node in these circumstances is assigned with the value $\fullmoon$ (``full moon''); both a white node and a gray node can be a full moon. The values $\newmoon$ and $\fullmoon$ are designated as \emph{absorbing values}. Figure \ref{fig:fig8} presents two examples\footnote{In \cite{BCG001}, similar situations are referred to as ``sunny'' and ``loony''.}.

\begin{figure}[htb!]
\begin{center}
\scalebox{0.7}{
\begin{tikzpicture}
\clip(-2.5,2.5) rectangle (12.5,3.5);
\draw [line width=1.pt] (-2.,3.) circle (0.4cm);
\draw [line width=1.pt, fill=gray!30] (0.,3.) circle (0.4cm);
\draw [line width=1.pt, fill=gray!30] (2.,3.) circle (0.4cm);
\draw [line width=1.pt, fill=gray!30] (4.,3.) circle (0.4cm);
\draw [line width=1.pt] (8.,3.) circle (0.4cm);
\draw [line width=1.pt, fill=gray!30] (10.,3.) circle (0.4cm);
\draw [line width=1.pt, fill=gray!30] (12.,3.) circle (0.4cm);
\draw [->,line width=1.pt] (-1.6,3.) -- (-1.,3.);
\draw [line width=1.pt] (-1.3,3.)-- (-0.4,3.);
\draw [->,line width=1.pt] (0.4,3.) -- (1.,3.);
\draw [line width=1.pt] (0.7,3.)-- (1.6,3.);
\draw [->,line width=1.pt] (2.4,3.) -- (3.,3.);
\draw [line width=1.pt] (2.7,3.)-- (3.6,3.);
\draw [->,line width=1.pt] (8.4,3.) -- (9.,3.);
\draw [line width=1.pt] (8.7,3.)-- (9.6,3.);
\draw [->,line width=1.pt] (10.4,3.) -- (11.,3.);
\draw [line width=1.pt] (10.7,3.)-- (11.6,3.);

\draw (-2.,3.+0.34) node[anchor=north] {\scalebox{1.42}{$\fullmoon$}};
\draw (-2.+2,3.+0.34) node[anchor=north] {\scalebox{1.42}{$\newmoon$}};
\draw (-2.+2+2,3.+0.34) node[anchor=north] {\scalebox{1.42}{$\fullmoon$}};
\draw (-2.+2+2+2,3.+0.34) node[anchor=north] {\scalebox{1.42}{$\newmoon$}};

\draw (-2.+2+2+2+4,3.+0.34) node[anchor=north] {\scalebox{1.42}{$0$}};
\draw (-2.+2+2+2+4+2,3.+0.34) node[anchor=north] {\scalebox{1.42}{$\fullmoon$}};
\draw (-2.+2+2+2+4+4,3.+0.34) node[anchor=north] {\scalebox{1.42}{$\newmoon$}};
\end{tikzpicture}}
\end{center}
\vspace{-0.5cm}
    \caption{New moon ($\newmoon$) and full moon ($\fullmoon$); two examples.}
    \label{fig:fig8}
\end{figure}

Still regarding the gray nodes, just like in LNST, it is crucial to retain information about the protection they provide. Thus, in addition to the absorbing values, a gray node can be assigned with $\mathbb{Z}_0^{+}\backslash F$ or with $F$, depending on the information provided by its single \emph{previously assigned} option. It is easy to understand how to do this in Algorithm \ref{alg:Gcarry}, so we do not further develop this point. It is only emphasized that a gray node, at the end of the execution, may maintain its initial $\infty$, meaning that it is in a cyclic zone.

One last important factor regarding this new theory concerns the exit from cyclic zones. In the subsection on SFPT, one can read ``when there are two or more pieces in cyclic zones, it is impossible to exit, as a player cannot make more than one move in the same turn''. However, with both cycles and carry-on moves, \emph{this statement is no longer true}. If there are two pieces in a cyclic zone, there is the possibility of executing a forcing sequence with one piece, exiting the cycle, and \emph{forcing the opponent} to move to a white node assigned with $k$. Afterwards, \emph{since the right to play is preserved}, exit with the other piece, possibly obtaining $k\oplus k=0$. Naturally, forcing sequences can be used for a larger number of pieces in cyclic zones. On the other hand, it is very interesting to note that in the theory of cyclic impartial games with carry-on moves, \emph{not all moons are equal}. If there is one piece in a cyclic zone and another piece in a white node assigned with $\moon\!\!$, one way to win may be ``preparing the exit by adjusting the moon''. Starting with the lunar node, a player can use carry-on moves to force the opponent to move to a white node assigned with $k$, and only then exit with the other piece from the cycle. Again, the process can be iterated for a larger number of pieces. Given a white node $G=(V,x^\circ)$ whose value is $\infty$ or $\moon\!\!$, the values accessible to the current player, $\mathbf{D}(G)$, and the values that the opponent of the current player cannot avoid, $\mathbf{F}(G)$, must be taken into account. If $G=(V,x^\circ)$, in a cyclic zone, is assigned with $\infty$, then $\mathbf{F}(G)$ is singular or empty; otherwise $G$ would be assigned with $\moon\!\!$ and not with $\infty$. If $\mathbf{F}(G)$ is empty then the information concerning that node should include the set $\mathbf{D}(G)$, else should include only $\mathbf{F}(G)=\{f\}$. Additionally, the information concerning $G=(V,x^\circ)$ assigned with $\moon\!\!$ should include the set $\mathbf{F}(G)$, which this time may have more than one element (\emph{lunar phase}). Not all moons are equal because their behavior in a disjunctive sum varies with their lunar phase. If $G$ is a moon, it remains true that $G+*n$ is an $\mathcal{N}$-position. However, it ceases to be true that it is sufficient for one component of a disjunctive sum to be a moon for the position to be in $\mathcal{N}$, since it is still possible for the other component to be in a cyclic zone. What is guaranteed is that the position is not in $\mathcal{P}$, although \emph{it can be a tie}. In other words, the assertion ``Next player wins'' should be replaced by ``Next player does not lose,'' which is not the same thing.

Definition \ref{def:directvalues} and Algorithm \ref{alg:direct} work for all games, even if the entailing moves are not carry-on moves (with the remark regarding the possibility of infinite run). The determination of $\mathbf{D}(G)$ can be made with a {\sc while-end} using Algorithm \ref{alg:direct}, where the {\sc end} is determined by the highest Grundy value $\mathcal{G}(G^\downarrow)$. Therefore, $\mathbf{D}(G)$ is used in Algorithm \ref{alg:Gcarry}.\\

\begin{definition}\label{def:directvalues}
Let $G=(V,x)$ be an impartial game with entailing moves. The \emph{direct values} in $G$, $\mathbf{D}(G)$, are determined in the following away:
\begin{enumerate}
  \item[\scalebox{0.5}{$\bullet$}] If there is $G'=(V,{x'}^\circ)$ such that $\mathcal{G}(G')=k$ then $k\in\mathbf{D}(G)$;
  \item[\scalebox{0.5}{$\bullet$}] If there is $G'=(V,{x'}^{\textcolor{gray}{\bullet}})$ such that $k\in \mathbf{D}(G'')$ for every $G''=(V,x'')$, then $k\in\mathbf{D}(G)$.
\end{enumerate}
\end{definition}

\vspace{0.8cm}
\begin{algorithm}[htb!]
\caption{{\sc boolean} function Dr to evaluate the possibility of exiting a node to a white node with Grundy value $k$}\label{alg:direct}
\begin{algorithmic}
\Require an integer $k\geqslant 0$ and $G=(V,x)$\\

\vspace{-0.3cm}
\If{there exists $G'=(V,{x'}^\circ)$ such that $\mathcal{G}(G')=k$ or there exists $G'=(V,{x'}^{\textcolor{gray}{\bullet}})$ such that Dr$(G'',k)=${\sc true} for all $G''$} \Return {\sc true} $\rightarrow$ Dr$(G,k)$
\Else{} \Return {\sc false} $\rightarrow$ Dr$(G,k)$
\EndIf
\end{algorithmic}
\end{algorithm}

\vspace{0.8cm}
Definition \ref{def:forcingvalues} and Algorithm \ref{alg:forcingvalues} work only for carry-on moves. Again, the determination of $\mathbf{F}(G)$ can be made with a {\sc while-end} using Algorithm \ref{alg:forcingvalues}. Essentially, the algorithm checks if it is possible to force $k$ with only one carry-on move or if a sequence of carry-on moves is required, which, for the majority of digraphs under study, has no computational cost. Thus, $\mathbf{F}(G)$ is also used in Algorithm \ref{alg:Gcarry}.

\begin{definition}\label{def:forcingvalues}
Let $G=(V,x)$ be an impartial game with carry-on moves. The \emph{forcing values} from $G$, $\mathbf{F}(G)$, are determined in the following away:
\begin{enumerate}
  \item[\scalebox{0.5}{$\bullet$}] If there is $G'=(V,{x'}^{\textcolor{gray}{\bullet}})$ with $G''=(V,{x''}^\circ)$ such that $\mathcal{G}(G'')=k$ then $k\in\mathbf{F}(G)$;
  \item[\scalebox{0.5}{$\bullet$}] If there is $G'=(V,{x'}^{\textcolor{gray}{\bullet}})$ with $G''=(V,x'')$ such that $k\in \mathbf{F}(G'')$, then $k\in\mathbf{F}(G)$.
\end{enumerate}
\end{definition}

\clearpage
Next, we present Algorithm \ref{alg:Gcarry} for determining Grundy values (exemplified in Figure \ref{fig:fig9}). In what follows, the adorned value of a lunar white node is $\mathcal{G}(G)=\moon_{\!\!a,b,\ldots}$, where $\mathbf{F}(G)=\{a,b,\ldots\}$. On the other hand, a white node that ends the first stage with $\infty$ has value $\mathcal{G}(G)=\infty_{a,b,\ldots}$ or $\mathcal{G}(G)=\nym_{f}$, depending on whether $\mathbf{F}(G)$ is empty or singular. Here, $\mathbf{D}(G)=\{a,b,\ldots\}$, and $\mathbf{F}(G)=\emptyset$ or $\mathbf{F}(G)=\{f\}$ are being considered. Moreover, in the first case, the absence of a subscript indicates that $\mathbf{D}(G)$ is an empty set.

The value represented by $\nym$ is called a \emph{nymphet}, and if we want to particularize the value $f$, we say ``\emph{ninny} $f$''. These designations are analogous to \emph{nimber} and \emph{star} $f$. What happens is that, when the opponent can be forced to move to $*f$, in practice the player can move in $*f$. In this sense, a nymphet and a nimber behave almost in the same way, the only difference being that the former is in a cyclic zone, while the latter is not. Because of that, $*n+*n \in \mathcal{P}$, whereas $*n+\nym_{n} \in \mathcal{D}$.
Accordingly, the ``nim'' sound was intentionally preserved.\\

\begin{algorithm}[htb!]
\caption{{\sc boolean} function Fr to evaluate the possibility of exiting a node and forcing the opponent to move to a white node with Grundy value $k$.}\label{alg:forcingvalues}
\begin{algorithmic}

\Require an integer $k\geqslant 0$ and $G=(V,x)$\\

\vspace{-0.3cm}
\If{there exists $G'=(V,{x'}^{\textcolor{gray}{\bullet}})$ with $G''=(V,x'')$, where (${x''}^\circ$ is white and $\mathcal{G}(G'')=k$) or Fr$(G'',k)=${\sc true}} \Return {\sc true} $\rightarrow$ Fr$(G,k)$
\Else{} \Return {\sc false} $\rightarrow$ Fr$(G,k)$
\EndIf
\end{algorithmic}
\end{algorithm}

\vspace{1cm}
\begin{breakablealgorithm}
\caption{Grundy values: cyclic impartial games with carry-on moves}\label{alg:Gcarry}
\begin{algorithmic}
\Require cyclic digraph $V$ with gray nodes whose outdegree is at most $1$\\

\vspace{-0.3cm}
\State \textbf{Step 0:}\Comment{Initialization}\\

\vspace{-0.3cm}
\If{$G=(V,x^\circ)$ is a game such that $x^\circ$ is a terminal node} $\mathcal{G}_0(G)=0$
\Else{}, if $x^\circ$ is not terminal, $\mathcal{G}_0(G)=\infty$
\EndIf
\If{$G=(V,x^{\textcolor{gray}{\bullet}})$ is a game such that $x^{\textcolor{gray}{\bullet}}$ is a terminal node} $\mathcal{C}_0(G)=\newmoon$
\Else{}, if $x^{\textcolor{gray}{\bullet}}$ is not terminal, $\mathcal{C}_0(G)=\infty$
\EndIf\\

\vspace{-0.3cm}
\State \textbf{Step n+1:}\\
\vspace{-0.3cm}
\If{$G=(V,x^\circ)$}\Comment{Assignment of white nodes}
\If{there is $G'=(V,{x'}^{\textcolor{gray}{\bullet}})$ with $\mathcal{C}_{n}(G')=\newmoon$} $\mathcal{G}_{n+1}(G)=\fullmoon$
\Else{}
\State \textbf{let} $D=\{\mathcal{G}_{n}(G'): G'=(V,{x'}^\circ)\text{ is assigned with an element of }\mathbb{Z}_0^{+}\}$
\State \textbf{let} $P=\cup_{G'=(V,{x'}^{\textcolor{gray}{\bullet}})\text{ assigned and not fullmoon}}\,\mathcal{C}_{n}(G')$
\If{$\mathbb{Z}_0^{+}\subseteq D\cup P$} $\mathcal{G}_{n+1}(G)=\moon$
\Else{} \State \textbf{let} $m=\mex(D\cup P)$
\If{for every unassigned $G'$, there is $G''=(V,{x''}^\circ)$ with $\mathcal{G}(G'')=m$ or
there is $G''=(V,{x''}^{\textcolor{gray}{\bullet}})$ such that Cr$(G'',m)=\,${\sc true}} $\mathcal{G}_{n+1}(G)=m$
\EndIf
\EndIf
\EndIf
\EndIf\\

\vspace{-0.3cm}
\If{$G=(V,x^{\textcolor{gray}{\bullet}})$ and $G'=(V,{x'}^\circ)$}\Comment{Assignment of gray nodes}
\State \textbf{if} $\mathcal{G}_{n}(G')=\fullmoon$ \textbf{then} $\mathcal{C}_{n+1}(G)=\newmoon$ \textbf{end if}
\State \textbf{if} $\mathcal{G}_{n}(G')=k\in\mathbb{Z}_0^{+}$ \textbf{then} $\mathcal{C}_{n+1}(G)=\mathbb{Z}_0^{+}\backslash \{k\}$ \textbf{end if}
\State \textbf{if} $\mathcal{G}_{n}(G')=\moon\!$ \textbf{then} $\mathcal{C}_{n+1}(G)=\mathbb{Z}_0^{+}$ \textbf{end if}
\EndIf\\

\vspace{-0.3cm}
\If{$G=(V,x^{\textcolor{gray}{\bullet}})$ and $G'=(V,{x'}^{\textcolor{gray}{\bullet}})$}
\State \textbf{if} $\mathcal{C}_{n}(G')=\fullmoon$ \textbf{then} $\mathcal{C}_{n+1}(G)=\newmoon$ \textbf{end if}
\State \textbf{if} $\mathcal{C}_{n}(G')=\newmoon$ \textbf{then} $\mathcal{C}_{n+1}(G)=\fullmoon$ \textbf{end if}
\State \textbf{if} $\mathcal{C}_{n}(G')=F$ \textbf{then} $\mathcal{C}_{n+1}(G)=\mathbb{Z}_0^{+}\backslash F$ \textbf{end if}
\State \textbf{if} $\mathcal{C}_{n}(G')=\mathbb{Z}_0^{+}\backslash F$ \textbf{then} $\mathcal{C}_{n+1}(G)=F$ \textbf{end if}
\EndIf

\If{$\mathcal{G}_{n+1}(G)=\mathcal{G}_{n}(G)$ and $\mathcal{C}_{n+1}(G)=\mathcal{C}_{n}(G)$, $\forall G=(V,x)$} \textbf{End}
\EndIf \Comment{Ending condition}\\\\

\vspace{-0.3cm}
\Return for each $G=(V,x^\circ)$, and the step $t$ being the final step,
\If{$\mathcal{G}_{t}(G)\in \mathbb{Z}_0^{+}\cup \{\fullmoon\}$} $\mathcal{G}(G)=\mathcal{G}_{t}(G)$
\ElsIf{$\mathcal{G}_{t}(G)=\moon$} $\mathcal{G}(G)=\moon\!\!_{a,b,\ldots}$, where $\mathbf{F}(G)=\{a,b,\ldots\}$
\ElsIf{$\mathcal{G}_{t}(G)=\infty$ and $\mathbf{F}(G)=\{f\}$} $\mathcal{G}(G)=\nym_{f}$
\ElsIf{$\mathcal{G}_{t}(G)=\infty$ and $\mathbf{F}(G)=\emptyset$} $\mathcal{G}(G)=\infty_{a,b,\ldots}$,
\State \quad\quad\quad\quad\quad\quad\quad\quad\quad\quad\quad\quad where $\mathbf{D}(G)=\{a,b,\ldots\}$
\EndIf
\end{algorithmic}
\end{breakablealgorithm}

\clearpage

\begin{figure}[htb!]
\begin{center}
\scalebox{1}{
\begin{tikzpicture}
\clip(0.,0.) rectangle (14.,10.);
\draw [line width=1.pt] (1.,1.) circle (0.4cm);
\draw [line width=1.pt] (3.,1.) circle (0.4cm);
\draw [line width=1.pt] (5.,1.) circle (0.4cm);
\draw [line width=1.pt] (7.,1.) circle (0.4cm);
\draw [line width=1.pt, fill=gray!30] (9.,1.) circle (0.4cm);
\draw [line width=1.pt] (11.,1.) circle (0.4cm);
\draw [line width=1.pt] (13.,1.) circle (0.4cm);
\draw [line width=3.pt] (1.,3.) circle (0.4cm);
\draw [line width=1.pt] (3.,3.) circle (0.4cm);
\draw [line width=1.pt] (5.,3.) circle (0.4cm);
\draw [line width=3.pt] (7.,3.) circle (0.4cm);
\draw [line width=1.pt] (9.,3.) circle (0.4cm);
\draw [line width=1.pt, fill=gray!30] (11.,3.) circle (0.4cm);
\draw [line width=1.pt, fill=gray!30] (13.,3.) circle (0.4cm);
\draw [line width=3.pt, fill=gray!30] (1.,5.) circle (0.4cm);
\draw [line width=1.pt, fill=gray!30] (3.,5.) circle (0.4cm);
\draw [line width=3.pt] (5.,5.) circle (0.4cm);
\draw [line width=3.pt, fill=gray!30] (7.,5.) circle (0.4cm);
\draw [line width=3.pt] (9.,5.) circle (0.4cm);
\draw [line width=1.pt] (11.,5.) circle (0.4cm);
\draw [line width=1.pt, fill=gray!30] (13.,5.) circle (0.4cm);
\draw [line width=1.pt] (1.,7.) circle (0.4cm);
\draw [line width=1.pt, fill=gray!30] (3.,7.) circle (0.4cm);
\draw [line width=1.pt] (5.,7.) circle (0.4cm);
\draw [line width=1.pt, fill=gray!30] (7.,7.) circle (0.4cm);
\draw [line width=1.pt] (9.,7.) circle (0.4cm);
\draw [line width=1.pt] (11.,7.) circle (0.4cm);
\draw [line width=1.pt] (13.,7.) circle (0.4cm);
\draw [line width=1.pt] (1.,9.) circle (0.4cm);
\draw [line width=1.pt] (3.,9.) circle (0.4cm);
\draw [line width=1.pt] (5.,9.) circle (0.4cm);
\draw [line width=1.pt, fill=gray!30] (7.,9.) circle (0.4cm);
\draw [line width=1.pt] (9.,9.) circle (0.4cm);
\draw [line width=1.pt] (11.,9.) circle (0.4cm);
\draw [line width=1.pt] (13.,9.) circle (0.4cm);
\draw [->,line width=1.pt] (2.6,9.) -- (2.,9.);
\draw [line width=1.pt] (2.3,9.)-- (1.4,9.);
\draw [->,line width=1.pt] (8.6,9.) -- (8.,9.);
\draw [line width=1.pt] (8.3,9.)-- (7.4,9.);
\draw [->,line width=1.pt] (4.6,5.) -- (4.,5.);
\draw [line width=1.pt] (4.3,5.)-- (3.4,5.);
\draw [->,line width=1.pt] (7.4,5.) -- (8.,5.);
\draw [line width=1.pt] (7.7,5.)-- (8.6,5.);
\draw [->,line width=1.pt] (9.4,5.) -- (10.,5.);
\draw [line width=1.pt] (9.7,5.)-- (10.6,5.);
\draw [->,line width=1.pt] (11.4,5.) -- (12.,5.);
\draw [line width=1.pt] (11.7,5.)-- (12.6,5.);
\draw [->,line width=1.pt] (12.6,3.) -- (12.,3.);
\draw [line width=1.pt] (12.3,3.)-- (11.4,3.);
\draw [->,line width=1.pt] (7.4,3.) -- (8.,3.);
\draw [line width=1.pt] (7.7,3.)-- (8.6,3.);
\draw [->,line width=1.pt] (11.4,1.) -- (12.,1.);
\draw [line width=1.pt] (11.7,1.)-- (12.6,1.);
\draw [->,line width=1.pt] (10.6,1.) -- (10.,1.);
\draw [line width=1.pt] (10.3,1.)-- (9.4,1.);
\draw [->,line width=1.pt] (6.6,1.) -- (6.,1.);
\draw [line width=1.pt] (6.3,1.)-- (5.4,1.);
\draw [->,line width=1.pt] (3.4,1.) -- (4.,1.);
\draw [line width=1.pt] (3.7,1.)-- (4.6,1.);
\draw [->,line width=1.pt] (4.6,3.) -- (4.,3.);
\draw [line width=1.pt] (4.3,3.)-- (3.4,3.);
\draw [->,line width=1.pt] (1.,8.6) -- (1.,8.);
\draw [line width=1.pt] (1.,8.3)-- (1.,7.4);
\draw [->,line width=1.pt] (3.,8.6) -- (3.,8.);
\draw [line width=1.pt] (3.,8.3)-- (3.,7.4);
\draw [->,line width=1.pt] (5.,7.4) -- (5.,8.);
\draw [line width=1.pt] (5.,7.7)-- (5.,8.6);
\draw [->,line width=1.pt] (3.282842712474623,8.717157287525378) -- (4.,8.);
\draw [line width=1.pt] (3.6414213562373194,8.358578643762682)-- (4.717157287525408,7.282842712474593);
\draw [->,line width=1.pt] (5.282842712474628,8.717157287525371) -- (6.,8.);
\draw [line width=1.pt] (5.6414213562373146,8.358578643762685)-- (6.717157287525374,7.282842712474628);
\draw [->,line width=1.pt] (13.,7.4) -- (13.,8.);
\draw [line width=1.pt] (13.,7.7)-- (13.,8.6);
\draw [->,line width=1.pt] (11.,7.4) -- (11.,8.);
\draw [line width=1.pt] (11.,7.7)-- (11.,8.6);
\draw [->,line width=1.pt] (13.,5.4) -- (13.,6.);
\draw [line width=1.pt] (13.,5.7)-- (13.,6.6);
\draw [->,line width=1.pt] (13.,1.4) -- (13.,2.);
\draw [line width=1.pt] (13.,1.7)-- (13.,2.6);
\draw [->,line width=1.pt] (11.,5.4) -- (11.,6.);
\draw [line width=1.pt] (11.,5.7)-- (11.,6.6);
\draw [->,line width=1.pt] (9.,5.4) -- (9.,6.);
\draw [line width=1.pt] (9.,5.7)-- (9.,6.6);
\draw [->,line width=1.pt] (9.,7.4) -- (9.,8.);
\draw [line width=1.pt] (9.,7.7)-- (9.,8.6);
\draw [->,line width=1.pt] (7.,8.6) -- (7.,8.);
\draw [line width=1.pt] (7.,8.3)-- (7.,7.4);
\draw [->,line width=1.pt] (3.,6.6) -- (3.,6.);
\draw [line width=1.pt] (3.,6.3)-- (3.,5.4);
\draw [->,line width=1.pt] (1.,1.4) -- (1.,2.);
\draw [line width=1.pt] (1.,1.7)-- (1.,2.6);
\draw [->,line width=1.pt] (1.,3.4) -- (1.,4.);
\draw [line width=1.pt] (1.,3.7)-- (1.,4.6);
\draw [->,line width=1.pt] (3.,2.6) -- (3.,2.);
\draw [line width=1.pt] (3.,2.3)-- (3.,1.4);
\draw [->,line width=1.pt] (5.,2.6) -- (5.,2.);
\draw [line width=1.pt] (5.,2.3)-- (5.,1.4);
\draw [->,line width=1.pt] (3.,4.6) -- (3.,4.);
\draw [line width=1.pt] (3.,4.3)-- (3.,3.4);
\draw [->,line width=1.pt] (5.,4.6) -- (5.,4.);
\draw [line width=1.pt] (5.,4.3)-- (5.,3.4);
\draw [->,line width=1.pt] (5.178885438199979,4.6422291236000435) -- (6.,3.);
\draw [line width=1.pt] (5.5894427190999885,3.821114561800023)-- (6.821114561800019,1.357770876399963);
\draw [->,line width=1.pt] (6.717157287525373,1.2828427124746282) -- (6.,2.);
\draw [line width=1.pt] (6.358578643762682,1.6414213562373188)-- (5.28284271247461,2.7171572875253904);
\draw [->,line width=1.pt] (11.,3.4) -- (11.,4.);
\draw [line width=1.pt] (11.,3.7)-- (11.,4.6);
\draw [->,line width=1.pt] (9.,1.4) -- (9.,2.);
\draw [line width=1.pt] (9.,1.7)-- (9.,2.6);
\draw [->,line width=1.pt] (7.,3.4) -- (7.,4.);
\draw [line width=1.pt] (7.,3.7)-- (7.,4.6);
\draw [->,line width=1.pt] (7.,1.4) -- (7.,2.);
\draw [line width=1.pt] (7.,1.7)-- (7.,2.6);
\draw [shift={(4.99856318834274,10.030187762640395)},line width=1.pt]  plot[domain=4.335989333062406:4.742855436060141,variable=\t]({1.*9.890182578869336*cos(\t r)+0.*9.890182578869336*sin(\t r)},{0.*9.890182578869336*cos(\t r)+1.*9.890182578869336*sin(\t r)});
\draw [shift={(4.998563188342742,10.030187762640386)},line width=1.pt]  plot[domain=4.682186889242518:5.089052992240254,variable=\t]({1.*9.890182578869327*cos(\t r)+0.*9.890182578869327*sin(\t r)},{0.*9.890182578869327*cos(\t r)+1.*9.890182578869327*sin(\t r)});
\draw [->,line width=1.pt] (4.699904401929208,0.14451558657323105) -- (5.299835385400387,0.14459488681091415);
\draw [shift={(3.136200084875491,0.8341193912057764)},line width=1.pt]  plot[domain=1.2208422499112235:1.654683601001732,variable=\t]({1.*4.788967422817719*cos(\t r)+0.*4.788967422817719*sin(\t r)},{0.*4.788967422817719*cos(\t r)+1.*4.788967422817719*sin(\t r)});
\draw [shift={(3.136200084875487,0.834119391205794)},line width=1.pt]  plot[domain=1.5293751373864943:1.9632164884770047,variable=\t]({1.*4.788967422817701*cos(\t r)+0.*4.788967422817701*sin(\t r)},{0.*4.788967422817701*cos(\t r)+1.*4.788967422817701*sin(\t r)});
\draw [->,line width=1.pt] (3.3345080935452427,5.618979149914157) -- (2.7147800055532327,5.604508693301928);
\draw [shift={(3.112146771878832,5.780762598636403)},line width=1.pt]  plot[domain=3.6740611484508:3.9930606472312737,variable=\t]({1.*2.2343184712408974*cos(\t r)+0.*2.2343184712408974*sin(\t r)},{0.*2.2343184712408974*cos(\t r)+1.*2.2343184712408974*sin(\t r)});
\draw [shift={(0.8498214989272379,3.53976199856965)},line width=1.pt]  plot[domain=0.8584528087502594:1.1938047477215414,variable=\t]({1.*1.4537409823253697*cos(\t r)+0.*1.4537409823253697*sin(\t r)},{0.*1.4537409823253697*cos(\t r)+1.*1.4537409823253697*sin(\t r)});
\draw [shift={(11.112146771878825,5.780762598636402)},line width=1.pt]  plot[domain=3.674061148450801:3.9930606472312764,variable=\t]({1.*2.2343184712408917*cos(\t r)+0.*2.2343184712408917*sin(\t r)},{0.*2.2343184712408917*cos(\t r)+1.*2.2343184712408917*sin(\t r)});
\draw [shift={(8.849821498927241,3.5397619985696553)},line width=1.pt]  plot[domain=0.8584528087502584:1.1938047477215425,variable=\t]({1.*1.4537409823253635*cos(\t r)+0.*1.4537409823253635*sin(\t r)},{0.*1.4537409823253635*cos(\t r)+1.*1.4537409823253635*sin(\t r)});
\draw [line width=1.pt] (3.74,2.26)-- (3.2828427124746193,2.7171572875253807);
\draw [line width=1.pt] (4.43,1.57)-- (4.717157287525384,1.2828427124746158);
\draw [line width=1.pt] (4.31,2.31)-- (4.717157287525384,2.717157287525384);
\draw [line width=1.pt] (3.2828427124746193,1.2828427124746193)-- (3.56,1.56);
\draw [->,line width=1.pt] (3.74,2.26) -- (4.43,1.57);
\draw [->,line width=1.pt] (4.31,2.31) -- (3.56,1.56);
\draw [shift={(1.8126315789473681,4.254736842105263)},line width=1.pt]  plot[domain=-2.41080252372329:0.19650508813510978,variable=\t]({1.*0.23183000745065485*cos(\t r)+0.*0.23183000745065485*sin(\t r)},{0.*0.23183000745065485*cos(\t r)+1.*0.23183000745065485*sin(\t r)});
\draw [shift={(9.812631578947368,4.254736842105264)},line width=1.pt]  plot[domain=-2.4108025237232846:0.19650508813510661,variable=\t]({1.*0.23183000745065477*cos(\t r)+0.*0.23183000745065477*sin(\t r)},{0.*0.23183000745065477*cos(\t r)+1.*0.23183000745065477*sin(\t r)});
\draw [->,line width=1.pt] (10.04,4.3) -- (9.8,4.64);
\draw [->,line width=1.pt] (2.04,4.3) -- (1.8,4.64);
\draw [->,line width=1.pt] (1.1788854381999836,1.3577708763999672) -- (2.,3.);
\draw [line width=1.pt] (1.5894427190999922,2.178885438199985)-- (2.8211145618000186,4.642229123600037);

\draw (1-0.06,1+0.36) node[anchor=north] {\scalebox{1.42}{$\moon$}};\draw (1+0.12,1+0.04) node[anchor=north] {\scalebox{0.42}{$\mathbf{0},\mathbf{2}$}};
\draw (3,1+0.3) node[anchor=north] {\scalebox{1.1}{$1$}};
\draw (5,1+0.3) node[anchor=north] {\scalebox{1.1}{$0$}};
\draw (7,1+0.3) node[anchor=north] {\scalebox{1.1}{$1$}};
\draw (9+0.01,1+0.27) node[anchor=north] {\scalebox{0.6}{$\mathbb{Z}^+$}};
\draw (11,1+0.36) node[anchor=north] {\scalebox{1.42}{$\moon$}};\draw (11+0.12,1+0.04) node[anchor=north] {\scalebox{0.42}{$\mathbf{0}$}};
\draw (13,1+0.3) node[anchor=north] {\scalebox{1.1}{$0$}};

\draw (1,3+0.22) node[anchor=north] {\scalebox{0.9}{$\infty$}};
\draw (3,3+0.3) node[anchor=north] {\scalebox{1.1}{$2$}};
\draw (5,3+0.3) node[anchor=north] {\scalebox{1.1}{$3$}};
\draw (7-0.15,3+0.22) node[anchor=north] {\scalebox{0.9}{$\infty$}};\draw (7+0.17,3+0.11) node[anchor=north] {\scalebox{0.42}{$0,1$}};
\draw (9,3+0.3) node[anchor=north] {\scalebox{1.1}{$0$}};
\draw (11,3+0.3) node[anchor=north] {\scalebox{0.6}{$\mathbb{Z}_0^+$}};
\draw (13,3+0.32) node[anchor=north] {\scalebox{0.9}{$\emptyset$}};

\draw (1,5+0.22) node[anchor=north] {\scalebox{0.9}{$\infty$}};
\draw (3+0.01,5+0.27) node[anchor=north] {\scalebox{0.5}{$\mathbb{Z}_0^+\backslash \{2\}$}};
\draw (5,5+0.26) node[anchor=north] {\scalebox{0.9}{$\nym_{2}$}};
\draw (7,5+0.22) node[anchor=north] {\scalebox{0.9}{$\infty$}};
\draw (9-0.05,5+0.22) node[anchor=north] {\scalebox{0.9}{$\infty$}};\draw (9+0.2,5+0.11) node[anchor=north] {\scalebox{0.42}{$1$}};
\draw (11,5+0.36) node[anchor=north] {\scalebox{1.42}{$\moon$}};\draw (11+0.12,5+0.04) node[anchor=north] {\scalebox{0.42}{$\mathbf{1}$}};
\draw (13+0.01,5+0.27) node[anchor=north] {\scalebox{0.5}{$\mathbb{Z}_0^+\backslash \{1\}$}};

\draw (1,7+0.3) node[anchor=north] {\scalebox{1.1}{$0$}};
\draw (3,7+0.25) node[anchor=north] {\scalebox{0.6}{$\{2\}$}};
\draw (5,7+0.3) node[anchor=north] {\scalebox{1.1}{$0$}};
\draw (7,7+0.3) node[anchor=north] {\scalebox{1.1}{$\newmoon$}};
\draw (9,7+0.3) node[anchor=north] {\scalebox{1.1}{$1$}};
\draw (11,7+0.3) node[anchor=north] {\scalebox{1.1}{$1$}};
\draw (13,7+0.3) node[anchor=north] {\scalebox{1.1}{$1$}};

\draw (1,9+0.3) node[anchor=north] {\scalebox{1.1}{$1$}};
\draw (3,9+0.3) node[anchor=north] {\scalebox{1.1}{$3$}};
\draw (5,9+0.3) node[anchor=north] {\scalebox{1.1}{$\fullmoon$}};
\draw (7,9+0.3) node[anchor=north] {\scalebox{1.1}{$\fullmoon$}};
\draw (9,9+0.3) node[anchor=north] {\scalebox{1.1}{$0$}};
\draw (11,9+0.3) node[anchor=north] {\scalebox{1.1}{$0$}};
\draw (13,9+0.3) node[anchor=north] {\scalebox{1.1}{$0$}};

\draw (1.+0.4,1-0.3+0.07) node[anchor=north] {\scalebox{0.8}{$4$}};
\draw (3.+0.4,1-0.3+0.07) node[anchor=north] {\scalebox{0.8}{$1$}};
\draw (5.+0.4,1-0.3+0.07) node[anchor=north] {\scalebox{0.8}{$0$}};
\draw (7.+0.4,1-0.3+0.14) node[anchor=north] {\scalebox{0.8}{$4$}};
\draw (9.+0.4,1-0.3+0.07) node[anchor=north] {\scalebox{0.8}{$1$}};
\draw (11.+0.4,1-0.3+0.07) node[anchor=north] {\scalebox{0.8}{$2$}};
\draw (13.+0.4,1-0.3+0.07) node[anchor=north] {\scalebox{0.8}{$1$}};

\draw (3.+0.3,3-0.3+0.05) node[anchor=north] {\scalebox{0.8}{$2$}};
\draw (5.+0.3,3-0.3+0.05) node[anchor=north] {\scalebox{0.8}{$3$}};
\draw (9.+0.4,3-0.3+0.07) node[anchor=north] {\scalebox{0.8}{$0$}};
\draw (11.+0.4,3-0.3+0.07) node[anchor=north] {\scalebox{0.8}{$4$}};
\draw (13.+0.4,3-0.3+0.07) node[anchor=north] {\scalebox{0.8}{$5$}};

\draw (3.+0.4,5-0.3+0.07) node[anchor=north] {\scalebox{0.8}{$3$}};
\draw (11.+0.4,5-0.3+0.07) node[anchor=north] {\scalebox{0.8}{$3$}};
\draw (13.+0.4,5-0.3+0.07) node[anchor=north] {\scalebox{0.8}{$2$}};

\draw (1.+0.4,7-0.3+0.07) node[anchor=north] {\scalebox{0.8}{$0$}};
\draw (3.+0.4,7-0.3+0.07) node[anchor=north] {\scalebox{0.8}{$4$}};
\draw (5.+0.4,7-0.3+0.07) node[anchor=north] {\scalebox{0.8}{$2$}};
\draw (7.+0.4,7-0.3+0.07) node[anchor=north] {\scalebox{0.8}{$0$}};
\draw (9.+0.4,7-0.3+0.07) node[anchor=north] {\scalebox{0.8}{$3$}};
\draw (11.+0.4,7-0.3+0.07) node[anchor=north] {\scalebox{0.8}{$1$}};
\draw (13.+0.4,7-0.3+0.07) node[anchor=north] {\scalebox{0.8}{$1$}};

\draw (1.+0.4,9-0.3+0.03) node[anchor=north] {\scalebox{0.8}{$1$}};
\draw (3.+0.3,9-0.3+0.03) node[anchor=north] {\scalebox{0.8}{$5$}};
\draw (5.+0.3,9-0.3+0.07) node[anchor=north] {\scalebox{0.8}{$1$}};
\draw (7.+0.4,9-0.3+0.07) node[anchor=north] {\scalebox{0.8}{$1$}};
\draw (9.+0.4,9-0.3+0.07) node[anchor=north] {\scalebox{0.8}{$2$}};
\draw (11.+0.4,9-0.3+0.07) node[anchor=north] {\scalebox{0.8}{$0$}};
\draw (13.+0.4,9-0.3+0.07) node[anchor=north] {\scalebox{0.8}{$0$}};
\end{tikzpicture}}
\end{center}
\vspace{-0.5cm}
    \caption{Example of Algorithm \ref{alg:Gcarry} execution.}
    \label{fig:fig9}
\end{figure}
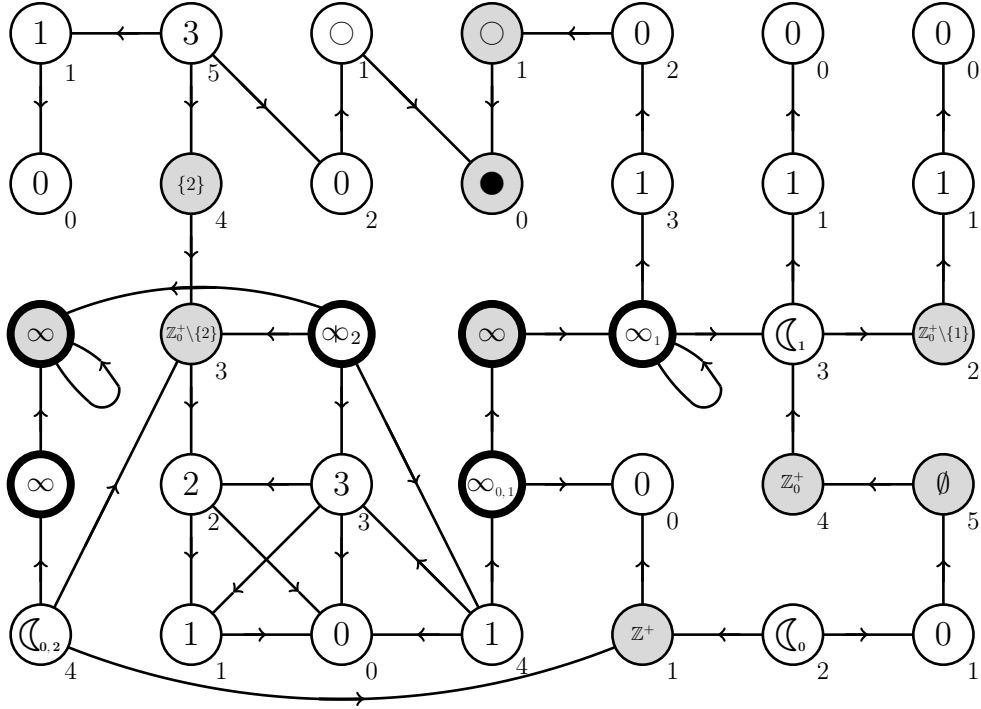

The following result confirms the effectiveness of the algorithm.

\begin{theorem}[Algorithm \ref{alg:Gcarry} halting theorem]\label{th:haltingcarry} Given a cyclic digraph $V$ with gray nodes whose outdegree is at most $1$, Algorithm \ref{alg:Gcarry} ends at a step $t\leqslant\#V$.
\end{theorem}

\begin{proof}
Let us begin by proving that once an assignment occurs, it is permanent and does not change anymore during the execution of the algorithm. Observe first that the assignment of a gray node is done for a terminal node or from its option, \emph{provided that this have been previously assigned}. This means that if such an assignment occurs, by induction on the assignment step, it is permanent.

On the other hand, the assignment of a white node is also done for a terminal node or from information gathered \emph{from all its options}. It takes into account the mex rule applied to previously assigned options and reversibility of unassigned ones to previously assigned nodes. The only exception occurs when the assignment is with $\moon\!\!$, when $\mathbb{Z}_0^{+}\subseteq D\cup P$, in which case there is no need to consider the unassigned options, \emph{either in that step or in subsequent ones}. Therefore, again by induction on the assignment step, any assignment of a white node is also permanent. In summary, during the algorithm execution, if an initial $\infty$ changes to data distinct from $\infty$, this occurs \emph{exactly once}.

Given the finiteness of the digraph $V$, the ending condition of Algorithm \ref{alg:Gcarry} must occur in a finite number of steps $t$, where the slowest case occurs when, at each step, \emph{exactly one node} is assigned with data distinct from $\infty$. Hence, $t$ cannot exceed $\#V$.
\end{proof}

As usual, a theorem follows that relates the $\mathcal{G}$-value of a game $G=(V,x^\circ)$ (white node) to its outcome. However, the information $\mathcal{C}(G)$ respecting a game $G=(V,x^{\textcolor{gray}{\bullet}})$  (gray node) deserves a comment. Not having an absorbing value $\newmoon$ or $\fullmoon$, if it is a set $S$, this information indicates that if a player has to play at this node, they lose in the presence of an element of $S$ and win in the presence of an element of $\mathbb{Z}_0^{+}\backslash S$. In contrast, if $\mathcal{C}(G)$ is $\infty$, there are elements of $\mathbb{Z}_0^{+}$ whose presence results in a tie.

\begin{theorem}[Outcome of a cyclic impartial game with carry-on moves and a given $\mathcal{G}$-value]\label{th:cycliccarryoutcome} If $G=(V,x^\circ)$ is a game, then
\begin{enumerate}
  \item[\scalebox{0.5}{$\bullet$}] $o(G)=\mathscr{P}$ if and only if $\mathcal{G}(G)=0$;
  \item[\scalebox{0.5}{$\bullet$}] $o(G)=\mathscr{N}$ if and only if $\mathcal{G}(G)\in\{\moon\!\!_{a,b,\ldots},\emph{\fullmoon}\}\cup\mathbb{Z}^{+}$ or\\

       \vspace{-0.3cm}
      $\,$\hspace{4.55cm}$\mathcal{G}(G)=\infty_{a,b,\ldots}$, with $0\in \{a,b.\ldots\}$ or\\

       \vspace{-0.3cm}
      $\,$\hspace{4.55cm}$\mathcal{G}(G)=\nym_{f}$, with $f\neq 0$;
  \item[\scalebox{0.5}{$\bullet$}] $o(G)=\mathscr{D}$ if and only if $\mathcal{G}(G)=\infty_{a,b,\ldots}$, with $0\not\in \{a,b.\ldots\}$ or\\

       \vspace{-0.3cm}
      $\,$\hspace{4.4cm}$\mathcal{G}(G)=\nym_{0}$.
\end{enumerate}
\end{theorem}

\begin{proof}
It suffices to prove the three implications
\begin{itemize}
  \item[\scalebox{0.5}{$\bullet$}] $\mathcal{G}(G)=0\implies o(G)=\mathscr{P}$,
  \item[\scalebox{0.5}{$\bullet$}] $\mathcal{G}(G)\in\{\moon\!\!_{a,b,\ldots},\emph{\fullmoon}\}\cup\mathbb{Z}^{+}$ or $\mathcal{G}(G)=\infty_{a,b,\ldots}$, with $0\in \{a,b.\ldots\}$ or
  \item[] $\mathcal{G}(G)=\nym_{f}$, with $f\neq 0$ $\implies o(G)=\mathscr{N}$, and
  \item[\scalebox{0.5}{$\bullet$}] $\mathcal{G}(G)=\infty_{a,b,\ldots}$, with $0\not\in \{a,b.\ldots\}$ or $\mathcal{G}(G)=\nym_{0}$ $\implies o(G)=\mathscr{D}$.
\end{itemize} Since the hypotheses of these implications cover all possible $\mathcal{G}$-values, if the converse of one of them were not satisfied, it would imply the failure of one of the others.

Let us start by proving the first implication. If $G=(V,x^\circ)$ is terminal, the implication is trivially satisfied. Otherwise, at the assignment step, the following occurs:
\begin{itemize}
  \item[] From the assigned options, we have $\mex(D\cup P)=0$;
  \item[] From any unassigned option there is a forcing sequence -- possibly an immediate reply -- resulting in a move to $0$ or to $\newmoon$, i.e., reversibility happens.
\end{itemize}

Thus, after a player makes a move in $G$, it is the opponent who can force a move either to $0$ or to $\newmoon$. If the move is to an assigned node, this happens because $\mex(D\cup P)=0$; if the move is to an unassigned node, it occurs due to reversibility. The opponent's move to $0$ is, by induction on the assignment step, a move to a $\mathcal{P}$-position and therefore a winning move, whereas the opponent's move to $\newmoon$ is trivially winning. Hence, the player who makes a move in $G$ loses.

Moving on to the second implication, if the Grundy value of $G$ is $\fullmoon$, the implication is satisfied since the next player has a move to a node with Grundy value $\newmoon$. If the Grundy value of $G$ is $\moon\!\!_{a,b,\ldots}$, then at the assignment step we have $\mathbb{Z}_0^{+} \subseteq D \cup P$ for the assigned options. Consequently, the next player can force a move to $0$. The argument extends unchanged to the cases $\mathcal{G}(G) \in \mathbb{Z}^{+}$ and $\mathcal{G}(G) = \infty_{a,b,\ldots}$, with $0 \in \{a,b,\ldots\}$. In these cases as well, the next player can force a move to $0$. Since nodes with Grundy value $0$ have already been shown to be $\mathcal{P}$-positions, the next player wins. Finally, if $\mathcal{G}(G) = \nym_f$, with $f \neq 0$, the next player can force \emph{the opponent} to move to a node whose Grundy value is a positive integer, which has already been shown to be an $\mathcal{N}$-position. Since it is the opponent who moves to that node, the next player wins as well.

Proving the third implication can no longer be based on induction on the assignment step, since, when facing a cyclic zone without a winning exit, the optimal play requires moves to unassigned nodes. The argument must therefore be a play-based argument.

First of all, if $\mathcal{G}(G) = \infty_{a,b,\ldots}$ with $0 \notin \{a,b,\ldots\}$, or if $\mathcal{G}(G) = \nym_0$, then the next player cannot force a move to $0$. For the first case, this is due to the fact that $0 \notin \{a,b,\ldots\}$. For the second case, if the next player could force a move to a node with Grundy value $0$, then we would have $\mathbb{Z}_0^+ \subseteq \mathbf{D}(G) \cup \mathbf{P}(G)$, and $G$ would be a moon rather than a nymphet. On the other hand, the next player cannot force a move to $\newmoon$ either; if this were possible, the Grundy value of $G$ would be $\fullmoon$.

It follows that the next player can potentially force moves only to nodes with Grundy values $\fullmoon$, $\moon_{\!\!\!a,b,\ldots}$, $n>0$, $\infty_{a,b,\ldots}$, or $\nym_f$. As we have seen, the first three cases lead to a loss, since they are $\mathcal{N}$-positions. The last two cases also lead to a loss if $0 \in \{a,b,\ldots\}$ or $f \neq 0$, as was previously observed. If only these cases existed, $G$ would be a $\mathcal{P}$-position and, as such, its Grundy value would be zero, which is not true. Therefore, there must be at least one way for the next player to force a move to a node with Grundy value $\infty_{a,b,\ldots}$ with $0 \notin \{a,b,\ldots\}$ or to a node with Grundy value $\nym_0$, and this must be the choice to avoid a loss. But in that case, the choice leads to a node \emph{under exactly the same conditions as the previous one}, and in response the opponent must repeat the strategy to avoid defeat. In other words, the ``saving strategy'' will always be repeated, leading the game to a draw. Hence, we have $o(G) = \mathscr{D}$.
\end{proof}

Regarding the following theorem, for simplicity, we write $\moon\!\!_{\mathbf{F}}$ instead of $\moon\!\!_{\mathbf{F}(G)}$, and $\infty_{\mathbf{D}}$ instead of $\infty_{\mathbf{D}(G)}$.

\begin{theorem}[$\mathcal{G}$-value of a disjunctive sum of cyclic impartial games with carry-on moves]\label{th:cycliccarrysum} If $G=(V,x^\circ)$ and $H=(U,y^\circ)$ are two games, then $\mathcal{G}(G+H)$ can be obtained from $\mathcal{G}(G)$ and $\mathcal{G}(H)$ using the following table:\\

\hspace{2.8cm}$\overbrace{\quad\quad\quad\quad\quad\quad\quad\quad\quad\quad\quad\quad\quad\quad\quad\quad\quad\quad\quad\,}^{\mathcal{G}(H)}$
\begin{center}
\begin{tabular}{|c||c|c|c|c|c|}
  \hline
  $+$ & $\quad\fullmoon\quad$ & $\infty_{\mathbf{D}'}$ & $\moon_{\!\!\!\mathbf{F}'}$ & $\nym_{f'}$ & $n'$ \\
\hhline{|=#=|=|=|=|=|}
  $\fullmoon$ & $\fullmoon$ & $\fullmoon$ & $\fullmoon$ & $\fullmoon$ & $\fullmoon$ \\
    \hline
  $\infty_{\mathbf{D}}$ & $\fullmoon$ & $\infty$ & $\infty_{\mathbf{D}\oplus\mathbf{F}'}$ & $\infty_{f'\oplus\mathbf{D}}$ & $\infty_{n'\oplus\mathbf{D}}$ \\
    \hline
  $\moon_{\!\!\!\mathbf{F}}$ & $\fullmoon$ & $\infty_{\mathbf{F}\oplus\mathbf{D}'}$ &  $\moon_{\!\!\!\mathbf{F}\oplus\mathbf{F}'}$ & $\moon_{\!\!\!f'\oplus\mathbf{F}}$ & $\moon_{\!\!\!n'\oplus\mathbf{F}}$   \\
    \hline
  $\nym_{f}$ & $\fullmoon$ & $\infty_{f\oplus\mathbf{D}'}$ & $\moon_{\!\!\!f\oplus\mathbf{F}'}$ & $\nym_{f\oplus f'}$ & $\nym_{f\oplus n'}$ \\
    \hline
  $n$ & $\fullmoon$ & $\infty_{n\oplus\mathbf{D}'}$ & $\moon_{\!\!\!n\oplus\mathbf{F}'}$ &  $\nym_{n\oplus f'}$  & $n\oplus n'$  \\
  \hline
\end{tabular}
\end{center}
\end{theorem}

\vspace{-2.95cm}
\hspace{-0.6cm}
$\begin{array}{r@{\hspace{0.5em}}l}
\raisebox{0em}{\scalebox{0.75}{$\mathcal{G}(G)$}} & \left\{ \rule{0pt}{3.4em} \right.
\end{array}$

\vspace{0.8cm}
\begin{proof}
If the Grundy value of one of the summands is $\fullmoon$, then there exists a sequence of carry-on moves leading to a final winning move to $\newmoon$. Naturally, this sequence can be carried out regardless of what the other components may be, which justifies the absorbing nature of $\fullmoon$, and consequently the first row of the addition table.

The simpler cases
\begin{itemize}
  \item [] $\mathcal{G}(G+H)=n\oplus m$ when  $\mathcal{G}(G)=n$ and $\mathcal{G}(H)=m$,
  \item [] $\mathcal{G}(G+H)=\moon_{\!\!\!n\oplus\mathbf{F}}$ when  $\mathcal{G}(G)=\moon_{\!\!\!\mathbf{F}}$ and $\mathcal{G}(H)=n$, and
  \item [] $\mathcal{G}(G+H)=\moon_{\!\!\!\mathbf{F}\oplus\mathbf{F}'}$ when  $\mathcal{G}(G)=\moon_{\!\!\!\mathbf{F}}$ and $\mathcal{G}(H)=\moon_{\!\!\!\mathbf{F}'}$
\end{itemize}
do not involve cyclic zones and behave essentially as in LNST. However, in that theory there is only one lunar Grundy value, and the moon is simply absorbing. In the present theory one must also take the adorns into account. When $\mathcal{G}(G)=\moon_{\!\!\!\mathbf{F}}$ and $\mathcal{G}(H)=n$, the $H$-component may allow the next player to force the opponent to move to a nimber. Yet such a move must be to $*n$ again; otherwise $H$ would not have Grundy value $n$. Thus, any forcing sequence initiated in $G$ will always coexist with a $*n$, and therefore the only {\sc nim} values to which the opponent can be forced are the elements of $n\oplus\mathbf{F}$. The case $\mathcal{G}(G)=\moon_{\!\!\!\mathbf{F}}$ and $\mathcal{G}(H)=\moon_{\!\!\!\mathbf{F}'}\,$ is similar. The set of {\sc nim} values to which the opponent may be forced is obtained by combining two independent forcing sequences, one in each component.

Let us now move on to the cases in which one of the components lies in a cyclic zone and the Grundy value of the other is a {\sc nim} value. If $\mathcal{G}(G)=\infty_{\mathbf{D}}$ and $\mathcal{G}(H)=n$, then, as already mentioned, any exit from the cyclic zone will always coexist with a $*n$. For this reason, the exits correspond to the elements of $n\oplus\mathbf{D}$. Thus, if $n \in \mathbf{D}$, the next player has a winning move, and if $n \not \in \mathbf{D}$, they do not. It turns out that, in the second case, the next player must have a way to force a draw; otherwise, the Grundy value of $G$ would be $n$. This explains why $\mathcal{G}(G+H)=\infty_{n \oplus \mathbf{D}}$. It remains only to note that the way to force a draw in this case requires maintaining the cyclic zone, but this \emph{does not mean that the drawing move must be made in} $G$; in some situations, the only way to draw is by making a move in $H$.

If $\mathcal{G}(G)=\nym_{f}$ and $\mathcal{G}(H)=n$, the analysis is very similar. The next player may exit the cyclic zone, forcing the opponent to play to $*f + *n$. This forcing sequence will only fail to be winning if $f = n$. If that is the case, the next player must have a way to force a draw while maintaining the cyclic zone; otherwise, the Grundy value of $G$ would be $f$. This explains why $\mathcal{G}(G+H)=\nym_{f \oplus n}$. The move that forces a draw must be made in $G$, so this component must allow a move to a Grundy value $\nym_{f}$ or to a Grundy value $\infty_{\mathbf{D}}$, where $n \not \in \mathbf{D}$. But note that at least one of the two possibilities must always exist;, otherwise the Grundy value of $G$ would be $n$.

The cases with both a cyclic value and a lunar value can be discussed using what we already know. If $\mathcal{G}(G)=\infty_{\mathbf{D}}$ and $\mathcal{G}(H)=\moon_{\!\!\!\mathbf{F}}$, then the next player cannot win by exiting the cyclic zone with the moon present in the disjunctive sum. Thus, in order to attempt to win, they must force the opponent to move to a node with Grundy value in $\mathbf{F}$ and exit the cyclic zone to a node with Grundy value in $\mathbf{D}$. Therefore, they will only be able to win when $0\in \mathbf{D}\oplus \mathbf{F}$, which occurs only when $\mathbf{F}\cap \mathbf{D}\neq \emptyset$. If this does not occur, then there exists $f\in \mathbf{F}$ such that $f\not \in \mathbf{D}$. In that case, as previously shown, the next player can force a draw by making the opponent move to $G+*f$. This explains why $\mathcal{G}(G+H)=\infty_{f \oplus \mathbf{D}}$.

If $\mathcal{G}(G)=\nym_{f}$ and $\mathcal{G}(H)=\moon_{\!\!\!\mathbf{F}}$, since the opponent can be forced to move to $*f$ in the first component, we have
\begin{itemize}
  \item [] $\mathbf{D}(G+H)=f\oplus \mathbf{D}(H)$ and
  \item [] $\mathbf{P}(G+H)=f\oplus \mathbf{P}(H)$.
\end{itemize}
In other words, $$\mathbf{D}(G+H)\cup \mathbf{P}(G+H) = f\oplus (\mathbf{D}(H)\cup \mathbf{P}(H)).$$ Given that the Grundy value of $H$ is $\moon_{\!\!\!\mathbf{F}}$, we know that $\mathbb{Z}_0^{+}\subseteq \mathbf{D}(H)\cup \mathbf{P}(H)$. Consequently, we also have $\mathbb{Z}_0^{+}\subseteq f\oplus (\mathbf{D}(H)\cup \mathbf{P}(H)) = \mathbf{D}(G+H)\cup \mathbf{P}(G+H)$ and $G+H$ has a lunar Grundy value. This explains why $\mathcal{G}(G+H)=\moon_{\!\!\!f\oplus \mathbf{F}}$.

It remains only to analyze the case in which both components have cyclic values. If $\mathcal{G}(G)=\nym_{f}$ and $\mathcal{G}(H)=\infty_{\mathbf{D}}$, then the next player can make a move to any Grundy value in $f\oplus \mathbf{D}$. To do this, it suffices to force the opponent to move to $*f$ in the first component and then to exit to a node whose Grundy value lies in $\mathbf{D}$ in the second. Thus, if $f \in \mathbf{D}$, the next player wins. On the other hand, if $f \not \in \mathbf{D}$, as already shown, it is enough for the next player to force the opponent to move to $*f + H$ in order to draw. This explains why $\mathcal{G}(G+H)=\infty_{f\oplus \mathbf{D}}$.

If $\mathcal{G}(G)=\nym_{f}$ and $\mathcal{G}(H)=\nym_{f'}$, then the next player can force the opponent to move to $*(f\oplus f')$ by combining two forcing sequences, one in each component. Thus, if $f \neq f'$, the next player wins. If $f' = f$, then, as we have seen earlier, the next player can draw by forcing the opponent to move to $*f + H$.  This explains why $\mathcal{G}(G+H)=\nym_{f\oplus f'}$.

If $\mathcal{G}(H)=\infty_{\mathbf{D}}$ and $\mathcal{G}(H)=\infty_{\mathbf{D}'}$, then, on the one hand, it is not possible to force the opponent to move to a nimber; on the other hand, the next player also cannot exit the cyclic zone to a nimber, because, as in the SFPL, they cannot do so from both components simultaneously in a single turn. Nevertheless, it is possible to argue that the next player has a way to force a draw. If $\mathbf{D} \neq \mathbf{D}'$, then without loss of generality we may assume that there exists $n \in \mathbf{D}$ such that $n \notin \mathbf{D}'$. In that case, as already shown, moving to $*n+H$ results in a draw. On the contrary, if $\mathbf{D} = \mathbf{D}'$, let $n = \mex(\mathbf{D})$. We begin by noting that the next player must have a way to move in $G$ to a follower with Grundy value $\nym_{n}$ or Grundy value $\infty_{\mathbf{D}''}$ with $n\not \in \mathbf{D}''$; otherwise $G$ would have Grundy value $n$ (all unsigned options would be reversible). The first possibility, as already seen, leads to a draw; the second possibility leads to a situation in which the Grundy value of the first component is $\infty_{\mathbf{D}''}$ and that of the second is $\infty_{\mathbf{D}'}$. In this latter case, however, \emph{we are exactly in the same initial conditions}, and the whole process repeats, once again leading to a draw. This explains why $\mathcal{G}(G+H)=\infty$.
\end{proof}

\section{Case study: {\sc green-lime hackenbush}}
\label{sec:hackenbush}

In this section, we introduce {\sc green-lime hackenbush}, a variant of the classical ruleset {\sc green hackenbush} \citep{BCG001,Con001}. The positions in {\sc green-lime hackenbush} are childish drawings, like the one shown in Figure~\ref{fig:fig11}. There are edges and nodes, and everything is connected to the ground, which may be thought of as a single node. Moreover, each edge joins two nodes, and its color may be either green or lime.

On their turn, players may make one of two types of moves: a \emph{removal} or a \emph{toggle}. If the player chooses to remove an edge, then after the removal anything no longer connected to the ground falls away. If the removed edge is lime, the turn passes to the opponent; if the removed edge is green \emph{and some lime edge also falls away, the player must move again} (a carry-on move). In the latter case, if no lime edge falls away, the turn passes to the opponent.

If the player chooses to toggle, \emph{there must be at least one lime edge connected to a green edge}. In that case, if they wish, the player may recolor the lime edge to green while simultaneously recoloring an adjacent green edge to lime. This type of move does not cause any edges to be removed and may give rise to a cyclic effect. This rule was inspired by the ruleset {\sc loopy hackenbush}, proposed in \cite{BCG001}. Figure~\ref{fig:fig10} illustrates a position together with its options.

\begin{figure}[htb!]
\begin{center}
\scalebox{0.9}{
\definecolor{qqwuqq}{rgb}{0.,0.39215686274509803,0.}
\definecolor{qqffqq}{rgb}{0.,1.,0.}
\begin{tikzpicture}
\clip(1.,-0.5) rectangle (16.,4.5);
\draw [line width=1.pt] (1.5,1.)-- (2.5,1.);
\draw [line width=3.6pt,color=qqffqq] (2.,1.)-- (2.,2.);
\draw [line width=3.6pt,color=qqwuqq] (2.,2.)-- (2.,3.);
\draw [line width=3.6pt,color=qqffqq] (2.,3.)-- (2.,4.);
\draw [shift={(4.855122950819671,3.24811475409836)},line width=2.pt]  plot[domain=2.420348508283319:3.5270449523955163,variable=\t]({1.*1.1386688210463614*cos(\t r)+0.*1.1386688210463614*sin(\t r)},{0.*1.1386688210463614*cos(\t r)+1.*1.1386688210463614*sin(\t r)});
\draw [line width=2.pt] (3.64,2.48)-- (3.8,2.82);
\draw [shift={(4.855122950819674,1.7518852459016416)},line width=2.pt]  plot[domain=2.7561403547840717:3.8628367988962657,variable=\t]({1.*1.1386688210463645*cos(\t r)+0.*1.1386688210463645*sin(\t r)},{0.*1.1386688210463645*cos(\t r)+1.*1.1386688210463645*sin(\t r)});
\draw [line width=2.pt] (3.64,2.52)-- (3.8,2.18);
\draw [line width=3.6pt,color=qqffqq] (5.,1.)-- (5.,2.);
\draw [line width=3.6pt,color=qqffqq] (5.,2.)-- (5.,3.);
\draw [line width=3.6pt,color=qqwuqq] (5.,3.)-- (5.,4.);
\draw [line width=1.pt] (4.5,1.)-- (5.5,1.);
\draw [line width=3.6pt,color=qqwuqq] (7.,1.)-- (7.,2.);
\draw [line width=3.6pt,color=qqffqq] (7.,2.)-- (7.,3.);
\draw [line width=3.6pt,color=qqffqq] (7.,3.)-- (7.,4.);
\draw [line width=1.pt] (6.5,1.)-- (7.5,1.);
\draw [line width=1.pt] (8.5,1.)-- (9.5,1.);
\draw [line width=3.6pt,color=qqffqq] (11.,1.)-- (11.,2.);
\draw [line width=3.6pt,color=qqwuqq] (11.,2.)-- (11.,3.);
\draw [line width=1.pt] (10.5,1.)-- (11.5,1.);
\draw [line width=3.6pt,color=qqffqq] (13.,1.)-- (13.,2.);
\draw [line width=1.pt] (12.5,1.)-- (13.5,1.);
\draw [shift={(14.14487704918033,3.2481147540983586)},line width=2.pt]  plot[domain=-0.3854522988057232:0.7212441453064751,variable=\t]({1.*1.1386688210463611*cos(\t r)+0.*1.1386688210463611*sin(\t r)},{0.*1.1386688210463611*cos(\t r)+1.*1.1386688210463611*sin(\t r)});
\draw [shift={(14.14487704918033,1.7518852459016403)},line width=2.pt]  plot[domain=-0.7212441453064748:0.38545229880572324,variable=\t]({1.*1.1386688210463607*cos(\t r)+0.*1.1386688210463607*sin(\t r)},{0.*1.1386688210463607*cos(\t r)+1.*1.1386688210463607*sin(\t r)});
\draw [line width=2.pt] (15.36,2.48)-- (15.2,2.82);
\draw [line width=2.pt] (15.36,2.52)-- (15.2,2.18);
\draw [shift={(6.,7.016666666666668)},line width=1.pt]  plot[domain=4.414609085165695:5.010168875603684,variable=\t]({1.*6.816666666666668*cos(\t r)+0.*6.816666666666668*sin(\t r)},{0.*6.816666666666668*cos(\t r)+1.*6.816666666666668*sin(\t r)});
\draw [shift={(11.5,20.76666666666665)},line width=1.pt]  plot[domain=4.54137839302828:4.883399567741099,variable=\t]({1.*20.566666666666652*cos(\t r)+0.*20.566666666666652*sin(\t r)},{0.*20.566666666666652*cos(\t r)+1.*20.566666666666652*sin(\t r)});
\begin{scriptsize}
\draw [color=black,fill=white] (2.,1.) circle (2.5pt);
\draw [color=black,fill=white] (2.,2.) circle (2.5pt);
\draw [color=black,fill=white] (2.,3.) circle (2.5pt);
\draw [color=black,fill=white] (2.,4.) circle (2.5pt);
\draw [color=black,fill=white] (5.,1.) circle (2.5pt);
\draw [color=black,fill=white] (5.,2.) circle (2.5pt);
\draw [color=black,fill=white] (5.,3.) circle (2.5pt);
\draw [color=black,fill=white] (5.,4.) circle (2.5pt);
\draw [color=black,fill=white] (7.,1.) circle (2.5pt);
\draw [color=black,fill=white] (7.,2.) circle (2.5pt);
\draw [color=black,fill=white] (7.,3.) circle (2.5pt);
\draw [color=black,fill=white] (7.,4.) circle (2.5pt);
\draw [color=black,fill=white] (11.,1.) circle (2.5pt);
\draw [color=black,fill=white] (11.,2.) circle (2.5pt);
\draw [color=black,fill=white] (11.,2.) circle (2.5pt);
\draw [color=black,fill=white] (11.,3.) circle (2.5pt);
\draw [color=black,fill=white] (13.,1.) circle (2.5pt);
\draw [color=black,fill=white] (13.,2.) circle (2.5pt);
\end{scriptsize}

\draw (2.55,2.68) node[anchor=north west] {$=$};
\draw (5.8,1.15) node[anchor=north west] {$,$};
\draw (5.8+2,1.15) node[anchor=north west] {$,$};
\draw (5.8+4,1.15) node[anchor=north west] {$,$};
\draw (5.8+6,1.15) node[anchor=north west] {$,$};
\draw (5.1,0.2) node[anchor=north west] {toggles};
\draw (5.1+5.4,0.2) node[anchor=north west] {removals};
\draw (13.25, 1.6) node[anchor=north west] {\scalebox{0.8}{\emph{carry-on}}};
\end{tikzpicture}}
\end{center}
\vspace{-0.5cm}
    \caption{A {\sc green-lime hackenbush} position and its options.}
    \label{fig:fig10}
\end{figure}
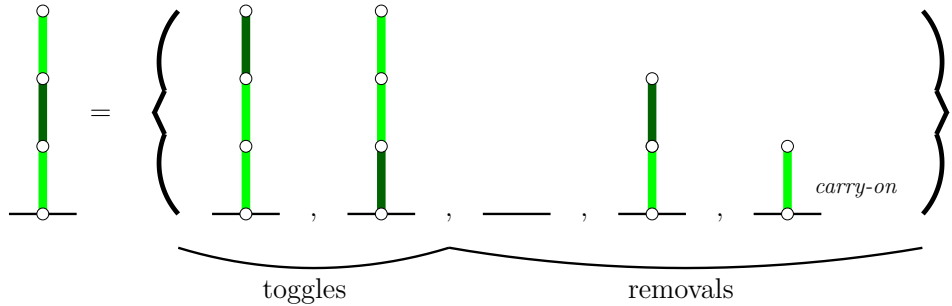

Since monochromatic {\sc green-lime hackenbush} positions allow neither carry-on moves nor toggles, the game is analogous to {\sc green hackenbush} whenever all edges are of the same color.

When a {\sc green-lime hackenbush} position has edges of different colors, to determine its Grundy value it is necessary to draw its game digraph and apply the Algorithm \ref{alg:Gcarry}. To illustrate the procedure, let us consider the challenge shown in Figure \ref{fig:fig11}, concerning a disjunctive sum with four disjoint components.

\begin{figure}[htb!]
\begin{center}
\definecolor{qqffqq}{rgb}{0.,1.,0.}
\definecolor{qqwuqq}{rgb}{0.,0.39215686274509803,0.}
\definecolor{ududff}{rgb}{0.30196078431372547,0.30196078431372547,1.}
\begin{tikzpicture}
\clip(1.4,-4.2) rectangle (11.,0.2);
\fill[line width=1.pt,fill=black,fill opacity=1.0] (2.3387520944474054,-4.) -- (3.234683283219673,-4.) -- (3.234683283219673,-2.3514233109714118) -- (2.814990844104971,-1.6032999275021507) -- (2.5221635292332794,-1.7287973481614471) -- (2.5147385779562437,-2.2474312979889164) -- (2.442744107429901,-2.5914048793925546) -- (2.298755166377215,-2.8713833758838883) -- cycle;
\fill[line width=1.pt,fill=black,fill opacity=1.0] (4.13061447199194,-4.) -- (3.2346832832196726,-4.) -- (3.2346832832196726,-2.3514233109714118) -- (3.6543757223343745,-1.6032999275021507) -- (3.9472030372060667,-1.7287973481614471) -- (3.954627988483102,-2.2474312979889164) -- (4.026622459009445,-2.5914048793925546) -- (4.170611400062131,-2.8713833758838883) -- cycle;
\fill[line width=1.pt,fill=black,fill opacity=1.0] (3.,-1.5) -- (2.986702329184491,-1.7114724618483639) -- (3.234683283219673,-1.5754829064097162) -- cycle;
\fill[line width=1.pt,fill=black,fill opacity=1.0] (3.469366566439346,-1.5) -- (3.4826642372548546,-1.7114724618483639) -- (3.234683283219673,-1.5754829064097162) -- cycle;
\fill[line width=1.pt,fill=black,fill opacity=1.0] (3.234683283219673,-1.4394933509710683) -- (3.234683283219673,-1.247508096234154) -- (3.074695570938909,-1.2075111681639636) -- (2.9547047867283376,-1.1195179264095445) -- (2.8507127737458426,-1.0315246846551254) -- (2.69872444707912,-0.839539429918211) -- (2.8007295195115076,-1.146937540511326) -- (2.9147078586581485,-1.319502566760497) -- (3.0342828582812356,-1.3942442983466161) -- cycle;
\fill[line width=1.pt,fill=black,fill opacity=1.0] (3.234683283219673,-1.4394933509710683) -- (3.234683283219673,-1.247508096234154) -- (3.394670995500437,-1.2075111681639636) -- (3.5146617797110085,-1.1195179264095445) -- (3.6186537926935034,-1.0315246846551254) -- (3.770642119360226,-0.8395394299182111) -- (3.6686370469278384,-1.146937540511326) -- (3.5546587077811975,-1.319502566760497) -- (3.4350837081581105,-1.3942442983466161) -- cycle;
\fill[line width=1.pt,fill=black,fill opacity=1.0] (3.234683283219673,-0.9035345148305158) -- (3.234683283219673,-0.7995425018480204) -- (3.0266992572546805,-0.8795363579884015) -- (2.994701714798529,-1.1755136257078111) -- (3.090694342166986,-1.2315093250060778) -- (3.1066931133950617,-0.9755289853568587) -- cycle;
\fill[line width=1.pt,fill=black,fill opacity=1.0] (3.234683283219673,-0.9035345148305158) -- (3.234683283219673,-0.7995425018480204) -- (3.4426673091846656,-0.8795363579884015) -- (3.474664851640817,-1.1755136257078111) -- (3.37867222427236,-1.2315093250060778) -- (3.3626734530442843,-0.9755289853568587) -- cycle;
\draw [shift={(-6.807574460701446,-2.9935458954057133)},line width=1.pt]  plot[domain=-0.11579515048851441:0.059764339625761984,variable=\t]({1.*8.711131683046695*cos(\t r)+0.*8.711131683046695*sin(\t r)},{0.*8.711131683046695*cos(\t r)+1.*8.711131683046695*sin(\t r)});
\draw [shift={(2.4118110011102054,-2.347330079452711)},line width=1.pt]  plot[domain=1.749887954045214:3.37749347920988,variable=\t]({1.*0.5387267254629343*cos(\t r)+0.*0.5387267254629343*sin(\t r)},{0.*0.5387267254629343*cos(\t r)+1.*0.5387267254629343*sin(\t r)});
\draw [line width=1.pt] (2.3158444833337564,-1.8172197964041)-- (2.814990844104971,-1.6032999275021507);
\draw [shift={(3.6969319065716615,-0.6789239091729918)},line width=1.pt]  plot[domain=3.301129175426223:3.8277639840192443,variable=\t]({1.*1.0110467237921361*cos(\t r)+0.*1.0110467237921361*sin(\t r)},{0.*1.0110467237921361*cos(\t r)+1.*1.0110467237921361*sin(\t r)});
\draw [shift={(4.030517600554396,-0.5968490254306129)},line width=1.pt]  plot[domain=2.736287967164821:3.3218431206910584,variable=\t]({1.*1.353725096197147*cos(\t r)+0.*1.353725096197147*sin(\t r)},{0.*1.353725096197147*cos(\t r)+1.*1.353725096197147*sin(\t r)});
\draw [line width=1.pt] (2.814990844104971,-1.6032999275021507)-- (2.9147078586581485,-1.319502566760497);
\draw [shift={(3.1643407696444847,-0.5040988269979044)},line width=1.pt]  plot[domain=1.4493777810318513:2.279231408029473,variable=\t]({1.*0.5807650540817942*cos(\t r)+0.*0.5807650540817942*sin(\t r)},{0.*0.5807650540817942*cos(\t r)+1.*0.5807650540817942*sin(\t r)});
\draw [shift={(13.276941027140758,-2.9935458954057155)},line width=1.pt]  plot[domain=3.0818283139640306:3.2573878040783075,variable=\t]({1.*8.711131683046663*cos(\t r)+0.*8.711131683046663*sin(\t r)},{0.*8.711131683046663*cos(\t r)+1.*8.711131683046663*sin(\t r)});
\draw [shift={(4.0575555653291415,-2.3473300794527105)},line width=1.pt]  plot[domain=-0.23590082562008785:1.3917046995445808,variable=\t]({1.*0.5387267254629335*cos(\t r)+0.*0.5387267254629335*sin(\t r)},{0.*0.5387267254629335*cos(\t r)+1.*0.5387267254629335*sin(\t r)});
\draw [shift={(2.772434659867685,-0.6789239091729917)},line width=1.pt]  plot[domain=5.597013976750135:6.123648785343156,variable=\t]({1.*1.011046723792136*cos(\t r)+0.*1.011046723792136*sin(\t r)},{0.*1.011046723792136*cos(\t r)+1.*1.011046723792136*sin(\t r)});
\draw [shift={(2.43884896588495,-0.596849025430613)},line width=1.pt]  plot[domain=-0.18025046710126524:0.4053046864249721,variable=\t]({1.*1.353725096197147*cos(\t r)+0.*1.353725096197147*sin(\t r)},{0.*1.353725096197147*cos(\t r)+1.*1.353725096197147*sin(\t r)});
\draw [shift={(3.305025796794861,-0.5040988269979039)},line width=1.pt]  plot[domain=0.8623612455603189:1.6922148725579413,variable=\t]({1.*0.5807650540817941*cos(\t r)+0.*0.5807650540817941*sin(\t r)},{0.*0.5807650540817941*cos(\t r)+1.*0.5807650540817941*sin(\t r)});
\draw [line width=1.pt] (4.15352208310559,-1.8172197964041001)-- (3.6543757223343745,-1.6032999275021507);
\draw [line width=1.pt] (3.6543757223343745,-1.6032999275021507)-- (3.5546587077811975,-1.319502566760497);
\draw [shift={(3.0146097902956135,-0.7510716483392288)},line width=1.pt]  plot[domain=1.0859925626300266:2.0043292739164125,variable=\t]({1.*0.3520720186509239*cos(\t r)+0.*0.3520720186509239*sin(\t r)},{0.*0.3520720186509239*cos(\t r)+1.*0.3520720186509239*sin(\t r)});
\draw [shift={(3.2461662722462745,-0.9222427553794739)},line width=1.pt]  plot[domain=4.017044043768784:4.690192576459457,variable=\t]({1.*0.51737804135549*cos(\t r)+0.*0.51737804135549*sin(\t r)},{0.*0.51737804135549*cos(\t r)+1.*0.51737804135549*sin(\t r)});
\draw [line width=1.pt] (2.298755166377215,-2.8713833758838883)-- (2.442744107429901,-2.5914048793925546);
\draw [line width=1.pt] (2.442744107429901,-2.5914048793925546)-- (2.5147385779562437,-2.2474312979889164);
\draw [line width=1.pt] (2.5147385779562437,-2.2474312979889164)-- (2.5221635292332794,-1.7287973481614471);
\draw [line width=1.pt] (2.814990844104971,-1.6032999275021507)-- (3.234683283219673,-2.3514233109714118);
\draw [line width=1.pt] (2.3387520944474054,-4.)-- (2.298755166377215,-2.8713833758838883);
\draw [line width=1.pt] (3.234683283219673,-2.3514233109714118)-- (3.234683283219673,-4.);
\draw [line width=1.pt] (2.3387520944474054,-4.)-- (3.234683283219673,-4.);
\draw [line width=1.pt] (3.234683283219673,-4.)-- (3.234683283219673,-2.3514233109714118);
\draw [line width=1.pt] (3.234683283219673,-2.3514233109714118)-- (2.814990844104971,-1.6032999275021507);
\draw [line width=1.pt] (2.814990844104971,-1.6032999275021507)-- (2.5221635292332794,-1.7287973481614471);
\draw [line width=1.pt] (2.5221635292332794,-1.7287973481614471)-- (2.5147385779562437,-2.2474312979889164);
\draw [line width=1.pt] (2.5147385779562437,-2.2474312979889164)-- (2.442744107429901,-2.5914048793925546);
\draw [line width=1.pt] (2.442744107429901,-2.5914048793925546)-- (2.298755166377215,-2.8713833758838883);
\draw [line width=1.pt] (2.298755166377215,-2.8713833758838883)-- (2.3387520944474054,-4.);
\draw [line width=1.pt] (4.13061447199194,-4.)-- (3.2346832832196726,-4.);
\draw [line width=1.pt] (3.2346832832196726,-4.)-- (3.2346832832196726,-2.3514233109714118);
\draw [line width=1.pt] (3.2346832832196726,-2.3514233109714118)-- (3.6543757223343745,-1.6032999275021507);
\draw [line width=1.pt] (3.6543757223343745,-1.6032999275021507)-- (3.9472030372060667,-1.7287973481614471);
\draw [line width=1.pt] (3.9472030372060667,-1.7287973481614471)-- (3.954627988483102,-2.2474312979889164);
\draw [line width=1.pt] (3.954627988483102,-2.2474312979889164)-- (4.026622459009445,-2.5914048793925546);
\draw [line width=1.pt] (4.026622459009445,-2.5914048793925546)-- (4.170611400062131,-2.8713833758838883);
\draw [line width=1.pt] (4.170611400062131,-2.8713833758838883)-- (4.13061447199194,-4.);
\draw [shift={(3.2232002941930697,-0.9222427553794739)},line width=1.pt]  plot[domain=4.734585384309925:5.407733917000597,variable=\t]({1.*0.5173780413554913*cos(\t r)+0.*0.5173780413554913*sin(\t r)},{0.*0.5173780413554913*cos(\t r)+1.*0.5173780413554913*sin(\t r)});
\draw [line width=1.pt] (3.,-1.5)-- (2.986702329184491,-1.7114724618483639);
\draw [line width=1.pt] (3.,-1.5)-- (2.986702329184491,-1.7114724618483639);
\draw [line width=1.pt] (2.986702329184491,-1.7114724618483639)-- (3.234683283219673,-1.5754829064097162);
\draw [line width=1.pt] (3.234683283219673,-1.5754829064097162)-- (3.,-1.5);
\draw [line width=1.pt] (3.469366566439346,-1.5)-- (3.4826642372548546,-1.7114724618483639);
\draw [line width=1.pt] (3.4826642372548546,-1.7114724618483639)-- (3.234683283219673,-1.5754829064097162);
\draw [line width=1.pt] (3.234683283219673,-1.5754829064097162)-- (3.469366566439346,-1.5);
\draw [line width=1.pt] (3.,-1.5)-- (2.878667609188586,-1.4220741060887048);
\draw [line width=1.pt] (3.469366566439346,-1.5)-- (3.59069895725076,-1.4220741060887048);
\draw [shift={(3.2226842047986137,-0.9315323644796477)},line width=1.pt]  plot[domain=3.708321871113308:4.789160871654481,variable=\t]({1.*0.15644884180554344*cos(\t r)+0.*0.15644884180554344*sin(\t r)},{0.*0.15644884180554344*cos(\t r)+1.*0.15644884180554344*sin(\t r)});
\draw [shift={(3.2466823616407314,-0.9315323644796487)},line width=1.pt]  plot[domain=4.635617089114904:5.71645608965608,variable=\t]({1.*0.15644884180554236*cos(\t r)+0.*0.15644884180554236*sin(\t r)},{0.*0.15644884180554236*cos(\t r)+1.*0.15644884180554236*sin(\t r)});
\draw [shift={(3.454756776143732,-0.7510716483392295)},line width=1.pt]  plot[domain=1.1372633796733802:2.0556000909597647,variable=\t]({1.*0.35207201865092463*cos(\t r)+0.*0.35207201865092463*sin(\t r)},{0.*0.35207201865092463*cos(\t r)+1.*0.35207201865092463*sin(\t r)});
\draw [line width=1.pt] (2.69872444707912,-0.839539429918211)-- (2.8507127737458426,-1.0315246846551254);
\draw [line width=1.pt] (2.9547047867283376,-1.1195179264095445)-- (2.8507127737458426,-1.0315246846551254);
\draw [line width=1.pt] (3.234683283219673,-1.247508096234154)-- (3.074695570938909,-1.2075111681639636);
\draw [line width=1.pt] (3.074695570938909,-1.2075111681639636)-- (2.9547047867283376,-1.1195179264095445);
\draw [line width=1.pt] (3.234683283219673,-1.4394933509710683)-- (3.234683283219673,-1.247508096234154);
\draw [line width=1.pt] (3.234683283219673,-1.247508096234154)-- (3.074695570938909,-1.2075111681639636);
\draw [line width=1.pt] (3.074695570938909,-1.2075111681639636)-- (2.9547047867283376,-1.1195179264095445);
\draw [line width=1.pt] (2.9547047867283376,-1.1195179264095445)-- (2.8507127737458426,-1.0315246846551254);
\draw [line width=1.pt] (2.8507127737458426,-1.0315246846551254)-- (2.69872444707912,-0.839539429918211);
\draw [line width=1.pt] (2.69872444707912,-0.839539429918211)-- (2.8007295195115076,-1.146937540511326);
\draw [line width=1.pt] (2.8007295195115076,-1.146937540511326)-- (2.9147078586581485,-1.319502566760497);
\draw [line width=1.pt] (2.9147078586581485,-1.319502566760497)-- (3.0342828582812356,-1.3942442983466161);
\draw [line width=1.pt] (3.0342828582812356,-1.3942442983466161)-- (3.234683283219673,-1.4394933509710683);
\draw [line width=1.pt] (3.234683283219673,-1.4394933509710683)-- (3.234683283219673,-1.247508096234154);
\draw [line width=1.pt] (3.234683283219673,-1.247508096234154)-- (3.394670995500437,-1.2075111681639636);
\draw [line width=1.pt] (3.394670995500437,-1.2075111681639636)-- (3.5146617797110085,-1.1195179264095445);
\draw [line width=1.pt] (3.5146617797110085,-1.1195179264095445)-- (3.6186537926935034,-1.0315246846551254);
\draw [line width=1.pt] (3.6186537926935034,-1.0315246846551254)-- (3.770642119360226,-0.8395394299182111);
\draw [line width=1.pt] (3.770642119360226,-0.8395394299182111)-- (3.6686370469278384,-1.146937540511326);
\draw [line width=1.pt] (3.6686370469278384,-1.146937540511326)-- (3.5546587077811975,-1.319502566760497);
\draw [line width=1.pt] (3.5546587077811975,-1.319502566760497)-- (3.4350837081581105,-1.3942442983466161);
\draw [line width=1.pt] (3.4350837081581105,-1.3942442983466161)-- (3.234683283219673,-1.4394933509710683);
\draw [line width=1.pt] (3.0266992572546805,-0.8795363579884015)-- (2.994701714798529,-1.1755136257078111);
\draw [line width=1.pt] (2.994701714798529,-1.1755136257078111)-- (3.090694342166986,-1.2315093250060778);
\draw [line width=1.pt] (3.090694342166986,-1.2315093250060778)-- (3.1066931133950617,-0.9755289853568587);
\draw [line width=1.pt] (3.234683283219673,-0.9035345148305158)-- (3.1066931133950617,-0.9755289853568587);
\draw [line width=1.pt] (3.0266992572546805,-0.8795363579884015)-- (3.234683283219673,-0.7995425018480204);
\draw [line width=1.pt] (3.234683283219673,-0.9035345148305158)-- (3.234683283219673,-0.7995425018480204);
\draw [line width=1.pt] (3.234683283219673,-0.7995425018480204)-- (3.0266992572546805,-0.8795363579884015);
\draw [line width=1.pt] (3.0266992572546805,-0.8795363579884015)-- (2.994701714798529,-1.1755136257078111);
\draw [line width=1.pt] (2.994701714798529,-1.1755136257078111)-- (3.090694342166986,-1.2315093250060778);
\draw [line width=1.pt] (3.090694342166986,-1.2315093250060778)-- (3.1066931133950617,-0.9755289853568587);
\draw [line width=1.pt] (3.1066931133950617,-0.9755289853568587)-- (3.234683283219673,-0.9035345148305158);
\draw [line width=1.pt] (3.4426673091846656,-0.8795363579884015)-- (3.234683283219673,-0.7995425018480204);
\draw [line width=1.pt] (3.234683283219673,-0.9035345148305158)-- (3.234683283219673,-0.7995425018480204);
\draw [line width=1.pt] (3.234683283219673,-0.7995425018480204)-- (3.4426673091846656,-0.8795363579884015);
\draw [line width=1.pt] (3.4426673091846656,-0.8795363579884015)-- (3.474664851640817,-1.1755136257078111);
\draw [line width=1.pt] (3.474664851640817,-1.1755136257078111)-- (3.37867222427236,-1.2315093250060778);
\draw [line width=1.pt] (3.37867222427236,-1.2315093250060778)-- (3.3626734530442843,-0.9755289853568587);
\draw [line width=1.pt] (3.3626734530442843,-0.9755289853568587)-- (3.234683283219673,-0.9035345148305158);
\draw [line width=2.6pt,color=qqwuqq] (5.657035258268461,-4.)-- (5.657035258268461,-3.3998754506956943);
\draw [line width=2.6pt,color=qqwuqq] (5.657035258268461,-3.3998754506956943)-- (6.039287361940807,-2.8330878486987676);
\draw [line width=2.6pt,color=qqwuqq] (6.039287361940807,-2.8330878486987676)-- (5.274783154596116,-2.8330878486987676);
\draw [line width=2.6pt,color=qqffqq] (5.657035258268461,-3.3998754506956943)-- (5.274783154596116,-2.8330878486987676);
\draw [line width=2.6pt,color=qqffqq] (7.120138137841922,-4.)-- (7.120138137841922,-3.3998754506956943);
\draw [line width=2.6pt,color=qqwuqq] (7.120138137841922,-3.3998754506956943)-- (7.502390241514267,-2.8330878486987676);
\draw [line width=2.6pt,color=qqwuqq] (7.120138137841922,-3.3998754506956943)-- (6.737886034169576,-2.8330878486987676);
\draw [line width=2.6pt,color=qqffqq] (8.583241017415382,-4.)-- (8.583241017415382,-3.3998754506956943);
\draw [line width=2.6pt,color=qqffqq] (8.583241017415382,-3.3998754506956943)-- (8.965493121087729,-2.8330878486987676);
\draw [line width=2.6pt,color=qqwuqq] (8.583241017415382,-3.3998754506956943)-- (8.200988913743036,-2.8330878486987676);
\draw [line width=2.6pt,color=qqwuqq] (7.502390241514267,-2.8330878486987676)-- (6.737886034169576,-2.8330878486987676);
\draw [line width=2.6pt,color=qqwuqq] (10.046343896988843,-4.)-- (10.046343896988843,-3.3998754506956943);
\draw [line width=2.6pt,color=qqffqq] (10.046343896988843,-3.3998754506956943)-- (10.428596000661189,-2.8330878486987676);
\draw [line width=2.6pt,color=qqffqq] (10.046343896988843,-3.3998754506956943)-- (9.664091793316496,-2.8330878486987676);
\draw [line width=1.pt] (1.4654432248958453,-4.)-- (10.95584028158856,-4.);
\begin{scriptsize}
\draw [fill=ududff] (-1.2165226491869265,-8.088) circle (2.5pt);
\draw [fill=ududff] (7.783477350813073,-8.088) circle (2.5pt);
\draw [fill=black] (3.0146986428687192,-0.5555597046549537) circle (2.5pt);
\draw [fill=black] (3.454667923570627,-0.5555597046549537) circle (2.5pt);

\draw [color=black,fill=white] (5.657035258268461,-4.) circle (2.5pt);
\draw [color=black,fill=white] (5.657035258268461,-3.3998754506956943) circle (2.5pt);
\draw [color=black,fill=white] (5.274783154596116,-2.8330878486987676) circle (2.5pt);
\draw [color=black,fill=white] (6.039287361940807,-2.8330878486987676) circle (2.5pt);
\draw [color=black,fill=white] (7.120138137841922,-4.) circle (2.5pt);
\draw [color=black,fill=white] (7.120138137841922,-3.3998754506956943) circle (2.5pt);
\draw [color=black,fill=white] (8.583241017415382,-4.) circle (2.5pt);
\draw [color=black,fill=white] (8.965493121087729,-2.8330878486987676) circle (2.5pt);
\draw [color=black,fill=white] (8.583241017415382,-3.3998754506956943) circle (2.5pt);
\draw [color=black,fill=white] (8.200988913743036,-2.8330878486987676) circle (2.5pt);
\draw [color=black,fill=white] (7.502390241514267,-2.8330878486987676) circle (2.5pt);
\draw [color=black,fill=white] (6.737886034169576,-2.8330878486987676) circle (2.5pt);
\draw [color=black,fill=white] (10.046343896988843,-4.) circle (2.5pt);
\draw [color=black,fill=white] (10.428596000661189,-2.8330878486987676) circle (2.5pt);
\draw [color=black,fill=white] (10.046343896988843,-3.3998754506956943) circle (2.5pt);
\draw [color=black,fill=white] (9.664091793316496,-2.8330878486987676) circle (2.5pt);
\end{scriptsize}
\end{tikzpicture}
\end{center}
\vspace{-0.5cm}
    \caption{The bartender serves us an interesting problem. Can you find a winning move before moving on to the next drink?}
    \label{fig:fig11}
\end{figure}
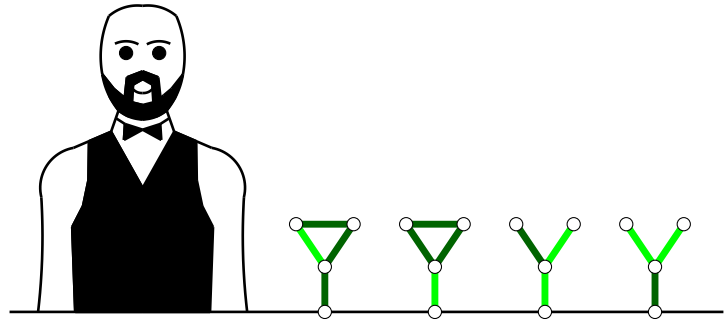

The game digraph for stalks with two edges of different colors is shown in Figure \ref{fig:fig12}. To avoid overloading the drawing, the nodes with a single edge were labeled directly with the Grundy value $1$. Note that although the players may choose to toggle, there are no cyclic zones. It is also interesting to observe that a green edge with a lime edge on top of it is a $\mathcal{P}$-position.

\begin{figure}[htb!]
\begin{center}\scalebox{0.75}{
\definecolor{aqaqaq}{rgb}{0.6274509803921569,0.6274509803921569,0.6274509803921569}
\definecolor{qqffqq}{rgb}{0.,1.,0.}
\definecolor{qqwuqq}{rgb}{0.,0.39215686274509803,0.}
\begin{tikzpicture}
\clip(4.5,-2.7) rectangle (17.5,5.2);
\draw [line width=0.4pt,color=gray,fill=gray!30] (9.,-2.) circle (0.5989134213566999cm);
\draw [line width=1.pt] (5.68,0.42)-- (6.32,0.42);
\draw [line width=3.6pt,color=qqwuqq] (6.,0.42)-- (6.,1.);
\draw [line width=3.6pt,color=qqffqq] (6.,1.)-- (6.,1.58);
\draw [line width=1.pt] (5.993836951477332,0.7223761150930783) circle (1.2cm);
\draw [line width=1.pt] (9.742506058887226,0.7235631413396626) circle (1.2cm);
\draw [line width=1.pt] (9.428669107409894,0.4211870262465843)-- (10.068669107409894,0.4211870262465842);
\draw [line width=3.6pt,color=qqffqq] (9.748669107409894,0.42118702624658433)-- (9.748669107409894,1.0011870262465843);
\draw [line width=3.6pt,color=qqwuqq] (9.748669107409894,1.0011870262465843)-- (9.748669107409894,1.5811870262465844);
\draw [line width=1.pt] (16.1,-1.) circle (1.2cm);
\draw [line width=1.pt] (14.6,3.) circle (1.2cm);
\draw [line width=1.pt] (15.786163048522669,-1.3023761150930784)-- (16.426163048522668,-1.3023761150930784);
\draw [line width=1.pt] (14.286163048522667,2.6976238849069216)-- (14.926163048522668,2.6976238849069216);
\draw [line width=3.6pt,color=qqffqq] (14.606163048522667,2.6976238849069216)-- (14.606163048522667,3.2776238849069217);
\draw [->,line width=1.pt] (6.883309676848088,-0.08312886367581651) -- (7.719689357942688,-0.8405528723569398);
\draw [line width=1.pt] (7.3014995173953885,-0.46184086801637814)-- (8.556069039037288,-1.597976881038063);
\draw [->,line width=1.pt] (9.593059944091019,-1.9164704304097155) -- (12.252394068982015,-1.5419163283123924);
\draw [line width=1.pt] (10.922727006536517,-1.729193379361054)-- (14.911728193873012,-1.1673622262150696);
\draw [line width=1.pt] (9.,-2.) circle (0.5989134213566998cm);
\draw [line width=1.pt] (9.748669107409894,3.846131078224102) circle (1.2cm);
\draw [line width=1.pt] (9.434832155932561,3.5437549631310237)-- (10.074832155932562,3.5437549631310237);
\draw [line width=3.6pt,color=qqwuqq] (9.75483215593256,3.5437549631310237)-- (9.754832155932561,4.123754963131024);
\draw [->,line width=1.pt] (15.021348129906077,1.8764049869171384) -- (15.35,1.);
\draw [line width=1.pt] (15.18567406495304,1.438202493458568)-- (15.678651870093935,0.1235950130828565);
\draw [->,line width=1.pt] (10.829100915915554,1.2327896299276024) -- (12.171253029443612,1.861781570669832);
\draw [line width=1.pt] (11.500176972679583,1.5472856002987172)-- (13.513405142971669,2.4907735114120615);
\draw [->,line width=1.pt] (10.900697597511105,0.4095689377943925) -- (12.921253029443614,-0.13821842933016715);
\draw [line width=1.pt] (11.910975313477358,0.13567525423211269)-- (14.941808461376123,-0.6860057964547268);
\draw [line width=1.pt] (7.193836891315965,0.7227560983307751)-- (8.542506119048587,0.7231831581019657);
\draw [->,line width=1.pt] (7.868171505182277,0.7229696282163705) -- (8.373922465582009,0.7231297756305668);
\draw [->,line width=1.pt] (7.868171505182277,0.7229696282163705) -- (7.362420544782543,0.7228094808021739);
\draw [->,line width=1.pt] (6.916339156239685,1.489832747228479) -- (7.871253029443609,2.2842535966585857);
\draw [line width=1.pt] (7.393796092841647,1.8870431719435323)-- (8.826166902647532,3.0786744460886926);
\draw [->,line width=1.pt] (10.70267849254499,3.1182120627600423) -- (12.924334553704961,1.4230655391120424);
\draw [line width=1.pt] (11.813506523124975,2.2706388009360423)-- (15.145990614864932,-0.2720809845359575);
\begin{scriptsize}
\draw [color=black,fill=white] (6.,0.42) circle (2.0pt);
\draw [color=black,fill=white] (6.,1.) circle (2.0pt);
\draw [color=black,fill=white] (6.,1.58) circle (2.0pt);
\draw [color=black,fill=white] (9.748669107409894,0.42118702624658433) circle (2.0pt);
\draw [color=black,fill=white] (9.748669107409894,1.0011870262465843) circle (2.0pt);
\draw [color=black,fill=white] (9.748669107409894,1.5811870262465844) circle (2.0pt);
\draw [color=black,fill=white] (14.606163048522667,2.6976238849069216) circle (2.0pt);
\draw [color=black,fill=white] (14.606163048522667,3.2776238849069217) circle (2.0pt);
\draw [color=black,fill=white] (9.75483215593256,3.5437549631310237) circle (2.0pt);
\draw [color=black,fill=white] (9.754832155932561,4.123754963131024) circle (2.0pt);
\end{scriptsize}

\draw (5.75,0.298914255574211) node[anchor=north west] {$0$};
\draw (5.8+3.7,0.298914255574211) node[anchor=north west] {$2$};
\draw (5.8+10.06,-1.4) node[anchor=north west] {$0$};
\draw (5.8+3.7,3.4) node[anchor=north west] {$1$};
\draw (14.35,2.55) node[anchor=north west] {$1$};
\draw (5.5+3.2,-1.7) node[anchor=north west] {$\mathbb{Z}^{+}$};

\end{tikzpicture}
}
\end{center}
\vspace{-0.5cm}
    \caption{Stalks with two edges of different colors.}
    \label{fig:fig12}
\end{figure}
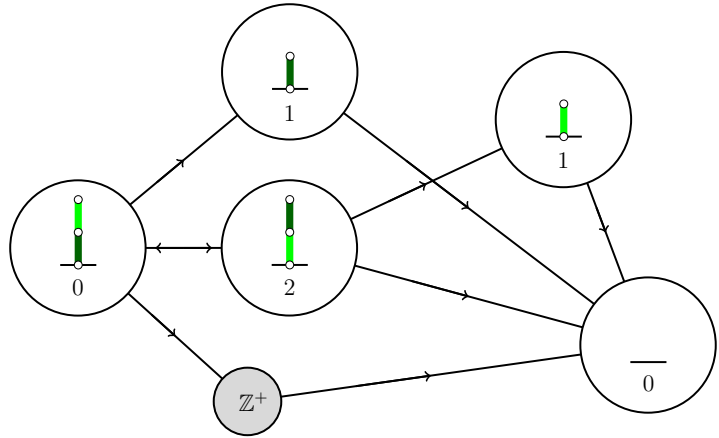

Having already determined the Grundy values of the stalks with two edges, we can now determine the values of the stalks with three edges (Figure \ref{fig:fig13}). Once again, the nodes corresponding to positions whose Grundy values are already known are simply labeled with those values.

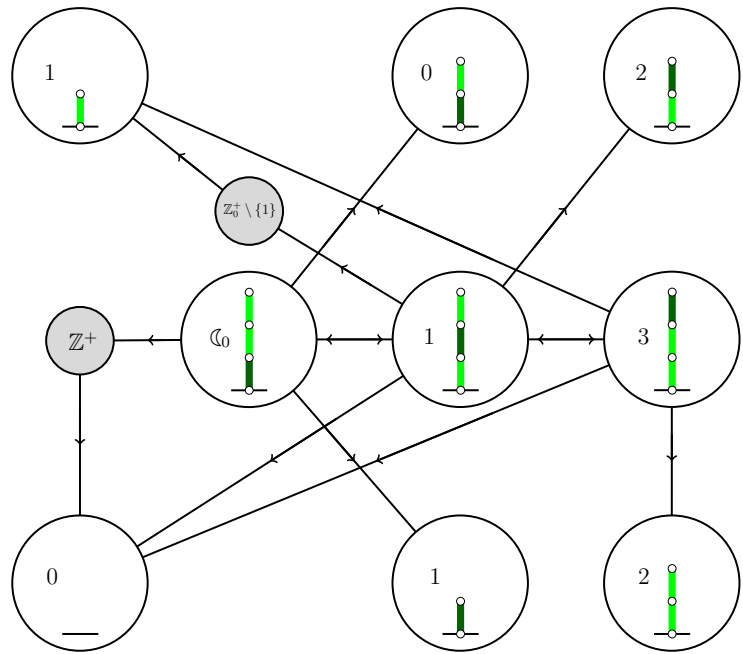
\begin{center}\scalebox{0.75}{
\definecolor{aqaqaq}{rgb}{0.6274509803921569,0.6274509803921569,0.6274509803921569}
\definecolor{qqffqq}{rgb}{0.,1.,0.}
\definecolor{qqwuqq}{rgb}{0.,0.39215686274509803,0.}
\begin{tikzpicture}
\clip(1.6,-4.3) rectangle (14.8,7.3);
\draw [line width=0.4pt,color=gray,fill=gray!30] (6.,3.6) circle (0.5989134213566999cm);
\draw [line width=0.4pt,color=gray,fill=gray!30] (3.,1.3) circle (0.5989134213566999cm);
\draw [line width=1.pt] (5.68,0.42)-- (6.32,0.42);
\draw [line width=3.6pt,color=qqwuqq] (6.,0.42)-- (6.,1.);
\draw [line width=3.6pt,color=qqffqq] (6.,1.)-- (6.,1.58);
\draw [line width=1.pt] (5.993836951477332,1.3223761150930786) circle (1.2cm);
\draw [line width=1.pt] (9.742506058887226,1.323563141339663) circle (1.2cm);
\draw [line width=1.pt] (9.428669107409894,0.4211870262465843)-- (10.068669107409894,0.4211870262465842);
\draw [line width=3.6pt,color=qqffqq] (9.748669107409894,0.42118702624658433)-- (9.748669107409894,1.0011870262465843);
\draw [line width=3.6pt,color=qqwuqq] (9.748669107409894,1.0011870262465843)-- (9.748669107409894,1.5811870262465844);
\draw [line width=3.6pt,color=qqffqq] (6.,1.58)-- (6.,2.16);
\draw [line width=3.6pt,color=qqffqq] (9.748669107409894,1.5811870262465844)-- (9.748669107409894,2.1611870262465844);
\draw [line width=1.pt] (13.49117516629712,1.3247501675862472) circle (1.2cm);
\draw [line width=3.6pt,color=qqffqq] (13.497338214819786,0.4223740524931687)-- (13.497338214819788,1.0023740524931686);
\draw [line width=3.6pt,color=qqffqq] (13.497338214819788,1.0023740524931686)-- (13.497338214819786,1.5823740524931684);
\draw [line width=3.6pt,color=qqwuqq] (13.497338214819786,1.5823740524931684)-- (13.497338214819788,2.1623740524931687);
\draw [line width=1.pt] (13.177338214819788,0.42237405249316856)-- (13.817338214819788,0.4223740524931685);
\draw [line width=1.pt] (3.,-3.) circle (1.2cm);
\draw [line width=1.pt] (9.74,-3.) circle (1.2cm);
\draw [line width=1.pt] (3.,6.) circle (1.2cm);
\draw [line width=1.pt] (9.74,6.) circle (1.2cm);
\draw [line width=1.pt] (2.6861630485226673,-3.9023761150930785)-- (3.326163048522668,-3.9023761150930785);
\draw [line width=1.pt] (9.426163048522668,-3.9023761150930785)-- (10.066163048522668,-3.9023761150930785);
\draw [line width=3.6pt,color=qqwuqq] (9.746163048522668,-3.9023761150930785)-- (9.746163048522668,-3.3223761150930784);
\draw [line width=1.pt] (13.49,6.) circle (1.2cm);
\draw [line width=1.pt] (2.6861630485226673,5.0976238849069215)-- (3.326163048522668,5.0976238849069215);
\draw [line width=3.6pt,color=qqffqq] (3.006163048522667,5.0976238849069215)-- (3.0061630485226676,5.677623884906922);
\draw [line width=3.6pt,color=qqwuqq] (9.746163048522668,5.0976238849069215)-- (9.746163048522668,5.677623884906922);
\draw [line width=1.pt] (9.426163048522668,5.0976238849069215)-- (10.066163048522668,5.0976238849069215);
\draw [line width=3.6pt,color=qqffqq] (9.746163048522668,5.677623884906922)-- (9.746163048522668,6.257623884906921);
\draw [line width=1.pt] (6.,3.6) circle (0.5989134213566998cm);
\draw [line width=1.pt] (3.,1.3) circle (0.5989134213566998cm);
\draw [->,line width=1.pt] (3.,0.701086578643299) -- (3.,-0.5494567106783511);
\draw [line width=1.pt] (3.,0.07581493398247402)-- (3.,-1.8);
\draw [->,line width=1.pt] (4.7938704670111685,1.3134074943567349) -- (4.196383580474334,1.3089418418992234);
\draw [line width=1.pt] (4.495127023742752,1.311174668127979)-- (3.598896693937501,1.3044761894417118);
\draw [->,line width=1.pt] (8.717271918857008,1.9471775719908893) -- (7.61448032880233,2.617967514216807);
\draw [line width=1.pt] (8.165876123829669,2.282572543103848)-- (6.51168873874765,3.2887574564427244);
\draw [->,line width=1.pt] (5.532327189705759,3.9741382482353913) -- (4.734684880518697,4.612252095585042);
\draw [line width=1.pt] (5.133506035112228,4.293195171910217)-- (3.9370425713316344,5.2503659429346925);
\draw [->,line width=1.pt] (6.779766027356854,0.41555998836883395) -- (7.866918475738668,-0.838811942453461);
\draw [line width=1.pt] (7.323342251547761,-0.21162597704231353)-- (8.954070924120481,-2.093183873275756);
\draw [->,line width=1.pt] (6.7439670996572065,2.259021598075641) -- (7.866918475738667,3.66118805754654);
\draw [line width=1.pt] (7.305442787697936,2.9601048278110906)-- (8.989869851820128,5.063354517017439);
\draw [line width=1.pt] (8.542506119048584,1.323183158101966)-- (7.193836891315973,1.3227560983307756);
\draw [->,line width=1.pt] (7.8681715051822785,1.3229696282163708) -- (7.36242054478255,1.3228094808021744);
\draw [->,line width=1.pt] (7.8681715051822785,1.3229696282163708) -- (8.373922465582009,1.323129775630567);
\draw [line width=1.pt] (12.291175226458455,1.3243701843485511)-- (10.942505998725869,1.3239431245773605);
\draw [->,line width=1.pt] (11.616840612592162,1.3241566544629557) -- (11.111089652192442,1.3239965070487592);
\draw [->,line width=1.pt] (11.616840612592162,1.3241566544629557) -- (12.122591572991881,1.3243168018771523);
\draw [->,line width=1.pt] (8.73234938514145,0.6758104921850077) -- (6.371253029443611,-0.8382184293301693);
\draw [line width=1.pt] (7.551801207292531,-0.08120396857258083)-- (4.010156673745773,-2.3522473508453463);
\draw [->,line width=1.pt] (10.492914512770762,2.2599856700841006) -- (11.616253029443612,3.6617815706698296);
\draw [line width=1.pt] (11.054583771107186,2.960883620376965)-- (12.739591546116461,5.063577471255559);
\draw [line width=1.pt] (13.176163048522668,5.0976238849069215)-- (13.816163048522668,5.0976238849069215);
\draw [line width=3.6pt,color=qqffqq] (13.496163048522668,5.0976238849069215)-- (13.496163048522668,5.677623884906922);
\draw [line width=3.6pt,color=qqwuqq] (13.496163048522668,5.677623884906922)-- (13.496163048522668,6.257623884906921);
\draw [line width=1.pt] (13.49,-3.) circle (1.2cm);
\draw [line width=1.pt] (13.176163048522664,-3.9023761150930794)-- (13.816163048522665,-3.9023761150930794);
\draw [line width=3.6pt,color=qqffqq] (13.496163048522664,-3.3223761150930793)-- (13.496163048522662,-2.7423761150930797);
\draw [line width=3.6pt,color=qqffqq] (13.496163048522662,-3.9023761150930794)-- (13.496163048522664,-3.3223761150930793);
\draw [->,line width=1.pt] (12.381742274262356,0.8674114986860584) -- (8.245587583148556,-0.8376249162068774);
\draw [line width=1.pt] (10.313664928705457,0.01489329123959049)-- (4.109432892034756,-2.542661331099813);
\draw [->,line width=1.pt] (12.395086744704177,1.813207097161257) -- (8.245587583148566,3.6623750837931217);
\draw [line width=1.pt] (10.320337163926371,2.7377910904771894)-- (4.096088421592957,5.511543070424986);
\draw [->,line width=1.pt] (13.490849089772505,0.12475021188869194) -- (13.49058758314856,-0.837624916206885);
\draw [line width=1.pt] (13.490718336460532,-0.3564373521590965)-- (13.490326076524616,-1.8000000443024617);
\draw (2.2185606072080737,6.342387994640347) node[anchor=north west] {$1$};
\draw (8.9,6.342387994640347) node[anchor=north west] {$0$};
\draw (12.7,6.342387994640347) node[anchor=north west] {$2$};
\draw (8.942649181835666,1.6668008229807003) node[anchor=north west] {$1$};
\draw (8.942649181835666-3.8,1.6668008229807003) node[anchor=north west] {$\moon\!\!_{0}$};
\draw (8.942649181835666+3.8,1.6668008229807003) node[anchor=north west] {$3$};

\draw (2.76-0.5,-2.6) node[anchor=north west] {$0$};
\draw (13.25-3.7-0.5,-2.6) node[anchor=north west] {$1$};
\draw (13.25-0.5,-2.6) node[anchor=north west] {$2$};

\draw (2.65,1.6) node[anchor=north west] {\scalebox{1}{$\mathbb{Z}^{+}$}};
\draw (5.4,3.9) node[anchor=north west] {\scalebox{0.6}{$\mathbb{Z}_0^{+}\setminus\{1\}$}};
\begin{scriptsize}
\draw [color=black,fill=white] (6.,0.42) circle (2.0pt);
\draw [color=black,fill=white] (6.,1.) circle (2.0pt);
\draw [color=black,fill=white] (6.,1.58) circle (2.0pt);
\draw [color=black,fill=white] (9.748669107409894,0.42118702624658433) circle (2.0pt);
\draw [color=black,fill=white] (9.748669107409894,1.0011870262465843) circle (2.0pt);
\draw [color=black,fill=white] (6.,2.16) circle (2.0pt);
\draw [color=black,fill=white] (9.748669107409894,1.5811870262465844) circle (2.0pt);
\draw [color=black,fill=white] (9.748669107409894,2.1611870262465844) circle (2.0pt);
\draw [color=black,fill=white] (13.497338214819786,0.4223740524931687) circle (2.0pt);
\draw [color=black,fill=white] (13.497338214819788,1.0023740524931686) circle (2.0pt);
\draw [color=black,fill=white] (13.497338214819786,1.5823740524931684) circle (2.0pt);
\draw [color=black,fill=white] (13.497338214819788,2.1623740524931687) circle (2.0pt);
\draw [color=black,fill=white] (9.746163048522668,-3.9023761150930785) circle (2.0pt);
\draw [color=black,fill=white] (9.746163048522668,-3.3223761150930784) circle (2.0pt);
\draw [color=black,fill=white] (3.006163048522667,5.0976238849069215) circle (2.0pt);
\draw [color=black,fill=white] (3.0061630485226676,5.677623884906922) circle (2.0pt);
\draw [color=black,fill=white] (9.746163048522668,5.0976238849069215) circle (2.0pt);
\draw [color=black,fill=white] (9.746163048522668,5.677623884906922) circle (2.0pt);
\draw [color=black,fill=white] (9.746163048522668,6.257623884906921) circle (2.0pt);
\draw [color=black,fill=white] (13.496163048522668,5.0976238849069215) circle (2.0pt);
\draw [color=black,fill=white] (13.496163048522668,5.677623884906922) circle (2.0pt);
\draw [color=black,fill=white] (13.496163048522668,6.257623884906921) circle (2.0pt);
\draw [color=black,fill=white] (13.496163048522662,-2.7423761150930797) circle (2.0pt);
\draw [color=black,fill=white] (13.496163048522662,-3.9023761150930794) circle (2.0pt);
\draw [color=black,fill=white] (13.496163048522664,-3.3223761150930793) circle (2.0pt);
\end{scriptsize}
\end{tikzpicture}
}
\end{center}
\clearpage
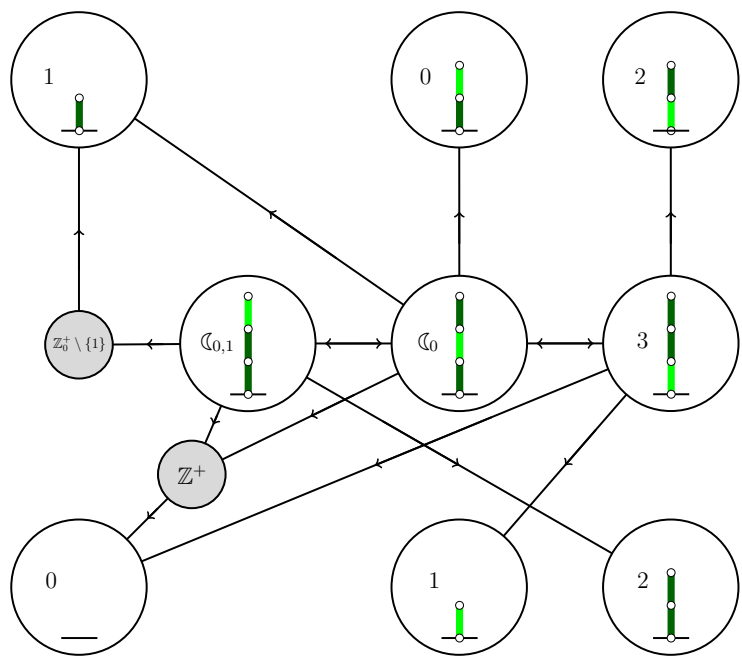
\begin{figure}[htb!]
\begin{center}\scalebox{0.75}{
\definecolor{aqaqaq}{rgb}{0.6274509803921569,0.6274509803921569,0.6274509803921569}
\definecolor{qqffqq}{rgb}{0.,1.,0.}
\definecolor{qqwuqq}{rgb}{0.,0.39215686274509803,0.}
\begin{tikzpicture}
\clip(1.6,-4.3) rectangle (14.8,7.3);
\draw [line width=0.4pt,color=gray,fill=gray!30] (3.,1.3) circle (0.5989134213566999cm);

\draw [line width=1.pt] (5.68,0.42)-- (6.32,0.42);
\draw [line width=3.6pt,color=qqwuqq] (6.,0.42)-- (6.,1.);
\draw [line width=3.6pt,color=qqwuqq] (6.,1.)-- (6.,1.58);
\draw [line width=3.6pt,color=qqffqq] (6.,1.58)-- (6.,2.16);

\draw [line width=1.pt] (5.993836951477332,1.3223761150930786) circle (1.2cm);
\draw [line width=1.pt] (9.742506058887226,1.323563141339663) circle (1.2cm);

\draw [line width=1.pt] (9.428669107409894,0.4211870262465843)-- (10.068669107409894,0.4211870262465842);
\draw [line width=3.6pt,color=qqwuqq] (9.748669107409894,0.42118702624658433)-- (9.748669107409894,1.0011870262465843);
\draw [line width=3.6pt,color=qqffqq] (9.748669107409894,1.0011870262465843)-- (9.748669107409894,1.5811870262465844);
\draw [line width=3.6pt,color=qqwuqq] (9.748669107409894,1.5811870262465844)-- (9.748669107409894,2.1611870262465844);

\draw [line width=1.pt] (13.49117516629712,1.3247501675862472) circle (1.2cm);

\draw [line width=3.6pt,color=qqffqq] (13.497338214819786,0.4223740524931687)-- (13.497338214819788,1.0023740524931686);
\draw [line width=3.6pt,color=qqwuqq] (13.497338214819788,1.0023740524931686)-- (13.497338214819786,1.5823740524931684);
\draw [line width=3.6pt,color=qqwuqq] (13.497338214819786,1.5823740524931684)-- (13.497338214819788,2.1623740524931687);
\draw [line width=1.pt] (13.177338214819788,0.42237405249316856)-- (13.817338214819788,0.4223740524931685);

\draw [line width=1.pt] (3.,-3.) circle (1.2cm);
\draw [line width=1.pt] (9.74,-3.) circle (1.2cm);
\draw [line width=1.pt] (3.,6.) circle (1.2cm);
\draw [line width=1.pt] (9.74,6.) circle (1.2cm);

\draw [line width=1.pt] (2.6861630485226673,-3.9023761150930785)-- (3.326163048522668,-3.9023761150930785);
\draw [line width=1.pt] (9.426163048522668,-3.9023761150930785)-- (10.066163048522668,-3.9023761150930785);
\draw [line width=3.6pt,color=qqffqq] (9.746163048522668,-3.9023761150930785)-- (9.746163048522668,-3.3223761150930784);

\draw [line width=1.pt] (13.49,6.) circle (1.2cm);

\draw [line width=1.pt] (2.6861630485226673,5.0976238849069215)-- (3.326163048522668,5.0976238849069215);
\draw [line width=3.6pt,color=qqwuqq] (3.006163048522667,5.0976238849069215)-- (3.0061630485226676,5.677623884906922);
\draw [line width=3.6pt,color=qqwuqq] (9.746163048522668,5.0976238849069215)-- (9.746163048522668,5.677623884906922);

\draw [line width=1.pt] (9.426163048522668,5.0976238849069215)-- (10.066163048522668,5.0976238849069215);
\draw [line width=3.6pt,color=qqffqq] (9.746163048522668,5.677623884906922)-- (9.746163048522668,6.257623884906921);
\draw [line width=1.pt] (3.,1.3) circle (0.5989134213566998cm);

\draw [line width=3.6pt,color=qqffqq] (13.496163048522668,5.0976238849069215)-- (13.496163048522668,5.677623884906922);
\draw [line width=3.6pt,color=qqwuqq] (13.496163048522668,5.677623884906922)-- (13.496163048522668,6.257623884906921);

\draw [line width=1.pt] (13.49,-3.) circle (1.2cm);

\draw [line width=1.pt] (13.176163048522664,-3.9023761150930794)-- (13.816163048522665,-3.9023761150930794);
\draw [line width=3.6pt,color=qqwuqq] (13.496163048522664,-3.3223761150930793)-- (13.496163048522662,-2.7423761150930797);
\draw [line width=3.6pt,color=qqwuqq] (13.496163048522662,-3.9023761150930794)-- (13.496163048522664,-3.3223761150930793);

\draw (2.2185606072080737,6.342387994640347) node[anchor=north west] {$1$};
\draw (8.9,6.342387994640347) node[anchor=north west] {$0$};
\draw (12.7,6.342387994640347) node[anchor=north west] {$2$};
\draw (8.85,1.6668008229807003) node[anchor=north west] {$\moon\!\!_{0}$};
\draw (8.85-3.85,1.6668008229807003) node[anchor=north west] {$\moon\!\!_{0,1}$};
\draw (8.942649181835666+3.8,1.6668008229807003) node[anchor=north west] {$3$};

\draw (2.76-0.5,-2.6) node[anchor=north west] {$0$};
\draw (13.25-3.7-0.5,-2.6) node[anchor=north west] {$1$};
\draw (13.25-0.5,-2.6) node[anchor=north west] {$2$};

\draw (2.4,1.6) node[anchor=north west] {\scalebox{0.6}{$\mathbb{Z}_0^{+}\setminus\{1\}$}};
\begin{scriptsize}
\draw [color=black,fill=white] (6.,0.42) circle (2.0pt);
\draw [color=black,fill=white] (6.,1.) circle (2.0pt);
\draw [color=black,fill=white] (6.,1.58) circle (2.0pt);
\draw [color=black,fill=white] (9.748669107409894,0.42118702624658433) circle (2.0pt);
\draw [color=black,fill=white] (9.748669107409894,1.0011870262465843) circle (2.0pt);
\draw [color=black,fill=white] (6.,2.16) circle (2.0pt);
\draw [color=black,fill=white] (9.748669107409894,1.5811870262465844) circle (2.0pt);
\draw [color=black,fill=white] (9.748669107409894,2.1611870262465844) circle (2.0pt);
\draw [color=black,fill=white] (13.497338214819786,0.4223740524931687) circle (2.0pt);
\draw [color=black,fill=white] (13.497338214819788,1.0023740524931686) circle (2.0pt);
\draw [color=black,fill=white] (13.497338214819786,1.5823740524931684) circle (2.0pt);
\draw [color=black,fill=white] (13.497338214819788,2.1623740524931687) circle (2.0pt);
\draw [color=black,fill=white] (9.746163048522668,-3.9023761150930785) circle (2.0pt);
\draw [color=black,fill=white] (9.746163048522668,-3.3223761150930784) circle (2.0pt);
\draw [color=black,fill=white] (3.006163048522667,5.0976238849069215) circle (2.0pt);
\draw [color=black,fill=white] (3.0061630485226676,5.677623884906922) circle (2.0pt);
\draw [color=black,fill=white] (9.746163048522668,5.0976238849069215) circle (2.0pt);
\draw [color=black,fill=white] (9.746163048522668,5.677623884906922) circle (2.0pt);
\draw [color=black,fill=white] (9.746163048522668,6.257623884906921) circle (2.0pt);
\draw [color=black,fill=white] (13.496163048522668,5.0976238849069215) circle (2.0pt);
\draw [color=black,fill=white] (13.496163048522668,5.677623884906922) circle (2.0pt);
\draw [color=black,fill=white] (13.496163048522668,6.257623884906921) circle (2.0pt);
\draw [color=black,fill=white] (13.496163048522662,-2.7423761150930797) circle (2.0pt);
\draw [color=black,fill=white] (13.496163048522662,-3.9023761150930794) circle (2.0pt);
\draw [color=black,fill=white] (13.496163048522664,-3.3223761150930793) circle (2.0pt);
\end{scriptsize}

\draw [line width=1.pt] (13.176163048522668,5.0976238849069215)-- (13.816163048522668,5.0976238849069215);

\draw [line width=0.4pt,color=gray,fill=gray!30] (5.,-1.) circle (0.5989134213566999cm);
\draw [line width=1.pt] (5.,-1.) circle (0.5989134213566998cm);
\draw [->,line width=1.pt] (4.7938704670111685,1.3134074943567349) -- (4.196383580474334,1.3089418418992234);
\draw [line width=1.pt] (4.495127023742752,1.311174668127979)-- (3.598896693937501,1.3044761894417118);
\draw [->,line width=1.pt] (7.033399691108034,0.7229519661204904) -- (9.741918475738661,-0.8388119424534572);
\draw [line width=1.pt] (8.387659083423348,-0.057929988166483426)-- (12.450437260369286,-2.400575851027405);
\draw [line width=1.pt] (7.193836891315973,1.3227560983307756)-- (8.542506119048584,1.323183158101966);
\draw [->,line width=1.pt] (7.8681715051822785,1.3229696282163708) -- (8.373922465582009,1.323129775630567);
\draw [->,line width=1.pt] (7.8681715051822785,1.3229696282163708) -- (7.36242054478255,1.3228094808021744);
\draw [line width=1.pt] (12.291175226458455,1.3243701843485511)-- (10.942505998725869,1.3239431245773605);
\draw [->,line width=1.pt] (11.616840612592162,1.3241566544629557) -- (11.111089652192442,1.3239965070487592);
\draw [->,line width=1.pt] (11.616840612592162,1.3241566544629557) -- (12.122591572991881,1.3243168018771523);
\draw [->,line width=1.pt] (3.,1.8989134213567007) -- (3.,3.349456710678351);
\draw [line width=1.pt] (3.,2.624185066017526)-- (3.,4.8);
\draw [->,line width=1.pt] (8.664893921086732,0.7955933690372579) -- (7.101362115687365,0.02955009397802555);
\draw [line width=1.pt] (7.8831280183870485,0.4125717315076417)-- (5.537830310287999,-0.7364931810812068);
\draw [->,line width=1.pt] (4.576504258415044,-1.423495741584954) -- (4.21251619791945,-1.7874838020805486);
\draw [line width=1.pt] (4.394510228167247,-1.6054897718327514)-- (3.8485281374238554,-2.151471862576143);
\draw [->,line width=1.pt] (5.521722648850835,0.21914989837436202) -- (5.378676154532693,-0.11511797248556535);
\draw [line width=1.pt] (5.450199401691764,0.05201596294439834)-- (5.235629660214552,-0.4493858433454927);
\draw [->,line width=1.pt] (8.756460834546687,2.007459928548946) -- (6.3712530294436185,3.661781570669828);
\draw [line width=1.pt] (7.5638569319951525,2.834620749609387)-- (3.98604522434055,5.3161032127907095);
\draw [->,line width=1.pt] (12.704892496263342,0.41824061837864707) -- (11.615587583148558,-0.8376249162068792);
\draw [line width=1.pt] (12.16024003970595,-0.2096921489141161)-- (10.526282670033774,-2.0934904507924057);
\draw [->,line width=1.pt] (12.381742274262356,0.8674114986860584) -- (8.245587583148556,-0.8376249162068774);
\draw [line width=1.pt] (10.313664928705457,0.01489329123959049)-- (4.109432892034756,-2.542661331099813);
\draw [->,line width=1.pt] (9.74186299017407,2.523562969032413) -- (9.741253029443614,3.6617815706698336);
\draw [line width=1.pt] (9.741558009808841,3.092672269851123)-- (9.740643068713158,4.8000001723072545);
\draw [->,line width=1.pt] (13.490873535461866,2.5247501296774346) -- (13.490587583148564,3.6623750837931253);
\draw [line width=1.pt] (13.490730559305215,3.09356260673528)-- (13.490301630835262,4.8000000379088155);

\draw (4.6,-0.7) node[anchor=north west] {\scalebox{1}{$\mathbb{Z}^{+}$}};
\end{tikzpicture}
}
\end{center}
\vspace{-0.5cm}
    \caption{Stalks with three edges of different colors.}
    \label{fig:fig13}
\end{figure}

There are stalks with three edges that, due to the presence of carry-on moves, have lunar values. Even so, although repetitive play is possible, their respective digraphs still do not exhibit cyclic zones.

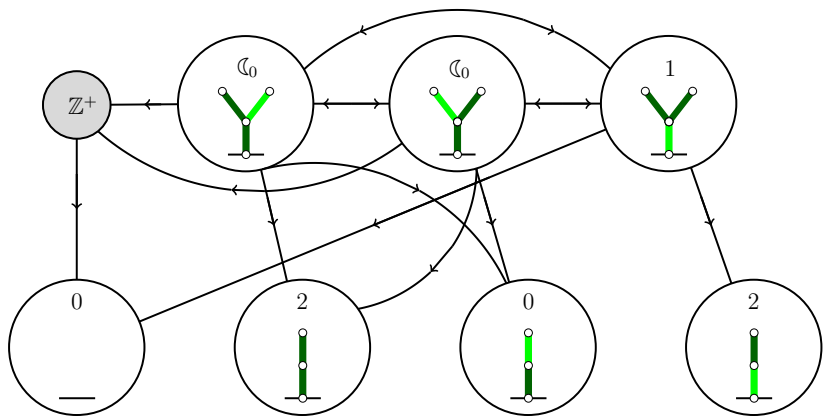
\begin{center}\scalebox{0.75}{
\definecolor{aqaqaq}{rgb}{0.6274509803921569,0.6274509803921569,0.6274509803921569}
\definecolor{qqffqq}{rgb}{0.,1.,0.}
\definecolor{qqwuqq}{rgb}{0.,0.39215686274509803,0.}
\begin{tikzpicture}
\clip(1.,-4.3) rectangle (16.4,3.1);
\draw [line width=0.4pt,color=gray,fill=gray!30] (3.,1.3) circle (0.5989134213566999cm);
\draw (5.887288676594245-0.2,2.2802364886905377) node[anchor=north west] {$\moon\!\!_{0}$};
\draw [line width=1.pt] (5.68,0.42)-- (6.32,0.42);
\draw [line width=3.6pt,color=qqwuqq] (6.,0.42)-- (6.,1.);
\draw [line width=1.pt] (5.993836951477332,1.3223761150930786) circle (1.2cm);
\draw [line width=1.pt] (9.742506058887226,1.323563141339663) circle (1.2cm);
\draw [line width=1.pt] (9.428669107409894,0.4211870262465843)-- (10.068669107409894,0.4211870262465842);
\draw [line width=3.6pt,color=qqwuqq] (9.748669107409894,0.42118702624658433)-- (9.748669107409894,1.0011870262465843);
\draw [line width=1.pt] (13.49117516629712,1.3247501675862472) circle (1.2cm);
\draw [line width=3.6pt,color=qqffqq] (13.497338214819786,0.4223740524931687)-- (13.497338214819788,1.0023740524931686);
\draw [line width=1.pt] (13.177338214819788,0.42237405249316856)-- (13.817338214819788,0.4223740524931685);
\draw [line width=1.pt] (3.,-3.) circle (1.2cm);
\draw [line width=1.pt] (2.6861630485226673,-3.9023761150930785)-- (3.326163048522668,-3.9023761150930785);
\draw [line width=1.pt] (3.,1.3) circle (0.5989134213566998cm);
\draw [->,line width=1.pt] (3.,0.701086578643299) -- (3.,-0.5494567106783511);
\draw [line width=1.pt] (3.,0.07581493398247402)-- (3.,-1.8);
\draw [->,line width=1.pt] (4.7938704670111685,1.3134074943567349) -- (4.196383580474334,1.3089418418992234);
\draw [line width=1.pt] (4.495127023742752,1.311174668127979)-- (3.598896693937501,1.3044761894417118);
\draw [line width=1.pt] (8.542506119048584,1.323183158101966)-- (7.193836891315973,1.3227560983307756);
\draw [->,line width=1.pt] (7.8681715051822785,1.3229696282163708) -- (7.36242054478255,1.3228094808021744);
\draw [->,line width=1.pt] (7.8681715051822785,1.3229696282163708) -- (8.373922465582009,1.323129775630567);
\draw [line width=1.pt] (12.291175226458455,1.3243701843485511)-- (10.942505998725869,1.3239431245773605);
\draw [->,line width=1.pt] (11.616840612592162,1.3241566544629557) -- (11.111089652192442,1.3239965070487592);
\draw [->,line width=1.pt] (11.616840612592162,1.3241566544629557) -- (12.122591572991881,1.3243168018771523);
\draw (9.647342749587647-0.2,2.2514238138016993) node[anchor=north west] {$\moon\!\!_{0}$};
\draw (13.39299048513663-0.12,2.222611138912861) node[anchor=north west] {$1$};
\draw (2.905176825599478-0.15,-1.9) node[anchor=north west] {$0$};
\draw (2.7,1.6) node[anchor=north west] {$\mathbb{Z}^{+}$};
\draw [line width=3.2pt,color=qqwuqq] (5.580604894521872,1.5417015006620476)-- (6.,1.);
\draw [line width=3.2pt,color=qqffqq] (6.419395105478128,1.5417015006620476)-- (6.,1.);
\draw [line width=3.2pt,color=qqffqq] (9.329274001931765,1.5428885269086319)-- (9.748669107409894,1.0011870262465843);
\draw [line width=3.2pt,color=qqwuqq] (10.168064212888023,1.5428885269086319)-- (9.748669107409894,1.0011870262465843);
\draw [line width=3.2pt,color=qqwuqq] (13.916733320297917,1.5440755531552162)-- (13.497338214819788,1.0023740524931686);
\draw [line width=3.2pt,color=qqwuqq] (13.077943109341659,1.5440755531552162)-- (13.497338214819788,1.0023740524931686);
\draw [shift={(9.74724040978494,-1.0479414212973992)},line width=1.pt]  plot[domain=0.8333549824263414:1.1282840804379441,variable=\t]({1.*4.034067430991071*cos(\t r)+0.*4.034067430991071*sin(\t r)},{0.*4.034067430991071*cos(\t r)+1.*4.034067430991071*sin(\t r)});
\draw [shift={(9.747240409785224,-1.047941421297077)},line width=1.pt]  plot[domain=2.015677385256428:2.3106064832680766,variable=\t]({1.*4.034067430990676*cos(\t r)+0.*4.034067430990676*sin(\t r)},{0.*4.034067430990676*cos(\t r)+1.*4.034067430990676*sin(\t r)});
\draw [shift={(9.747240409785205,-1.0479414212970595)},line width=1.pt]  plot[domain=1.2770516348355718:1.8669098308588559,variable=\t]({1.*4.034067430990652*cos(\t r)+0.*4.034067430990652*sin(\t r)},{0.*4.034067430990652*cos(\t r)+1.*4.034067430990652*sin(\t r)});
\draw [->,line width=1.pt] (8.570079069669216,2.8105545573344917) -- (8.01117739574482,2.5934573826051626);
\draw [->,line width=1.pt] (10.915258403785234,2.8133322032304315) -- (11.47467277157061,2.5975595694086078);
\draw [line width=1.pt] (7.,-3.) circle (1.2cm);
\draw [line width=1.pt] (11.,-3.) circle (1.2cm);
\draw [line width=1.pt] (6.686163048522666,-3.9023761150930785)-- (7.326163048522667,-3.9023761150930785);
\draw [line width=3.6pt,color=qqwuqq] (7.006163048522667,-3.9023761150930785)-- (7.006163048522667,-3.3223761150930784);
\draw [line width=3.6pt,color=qqwuqq] (7.006163048522667,-3.3223761150930784)-- (7.006163048522667,-2.7423761150930783);
\draw (6.9-0.15,-1.9) node[anchor=north west] {$2$};
\draw (10.915100444696535-0.15,-1.9) node[anchor=north west] {$0$};
\draw [line width=1.pt] (10.686163048522666,-3.9023761150930785)-- (11.326163048522666,-3.9023761150930785);
\draw [line width=3.6pt,color=qqwuqq] (11.006163048522668,-3.9023761150930785)-- (11.006163048522666,-3.3223761150930784);
\draw [line width=3.6pt,color=qqffqq] (11.006163048522666,-3.3223761150930784)-- (11.006163048522668,-2.7423761150930783);
\draw [line width=1.pt] (15.,-3.) circle (1.2cm);
\draw [line width=1.pt] (14.686163048522666,-3.9023761150930785)-- (15.326163048522666,-3.9023761150930785);
\draw [line width=3.6pt,color=qqffqq] (15.006163048522668,-3.9023761150930785)-- (15.006163048522666,-3.3223761150930784);
\draw [line width=3.6pt,color=qqwuqq] (15.006163048522666,-3.3223761150930784)-- (15.006163048522668,-2.7423761150930783);
\draw [->,line width=1.pt] (6.26589922104285,0.15362373692272524) -- (6.496918475738669,-0.8388119424534604);
\draw [line width=1.pt] (6.381408848390759,-0.34259410276536756)-- (6.727937730434487,-1.8312476218296458);
\draw [->,line width=1.pt] (10.077635194268451,0.17130949625164418) -- (10.371253029443611,-0.8382184293301669);
\draw [line width=1.pt] (10.22444411185603,-0.33345446653926136)-- (10.664870864618774,-1.847746354911978);
\draw [->,line width=1.pt] (13.886466377042495,0.19172551705675359) -- (14.245587583148563,-0.8376249162068836);
\draw [line width=1.pt] (14.066026980095529,-0.32294969957506503)-- (14.604708789254632,-1.8669753494705208);
\draw [->,line width=1.pt] (12.381742274262356,0.8674114986860584) -- (8.245587583148556,-0.8376249162068774);
\draw [line width=1.pt] (10.313664928705457,0.01489329123959049)-- (4.109432892034756,-2.542661331099813);
\draw [->,line width=1.pt] (6.334978683492921,-0.22868616265896535) -- (5.735809101044162,-0.20511890003626812);
\draw [shift={(6.203537690192571,4.057953263591205)},line width=1.pt]  plot[domain=4.743042321880924:5.353732509967569,variable=\t]({1.*4.288654136836191*cos(\t r)+0.*4.288654136836191*sin(\t r)},{0.*4.288654136836191*cos(\t r)+1.*4.288654136836191*sin(\t r)});
\draw [shift={(6.203537690192552,4.057953263591209)},line width=1.pt]  plot[domain=3.992419562792838:4.603109750879483,variable=\t]({1.*4.288654136836194*cos(\t r)+0.*4.288654136836194*sin(\t r)},{0.*4.288654136836194*cos(\t r)+1.*4.288654136836194*sin(\t r)});
\draw [->,line width=1.pt] (8.533945620529318,0.03187646003873024) -- (9.077101914321572,-0.22192864502143994);
\draw [shift={(7.206522880151855,-3.5169722172399736)},line width=1.pt]  plot[domain=1.212864752501719:1.8156774990265676,variable=\t]({1.*3.788981138234132*cos(\t r)+0.*3.788981138234132*sin(\t r)},{0.*3.788981138234132*cos(\t r)+1.*3.788981138234132*sin(\t r)});
\draw [shift={(7.206522880151766,-3.5169722172399918)},line width=1.pt]  plot[domain=0.4516566977116328:1.0544694442364722,variable=\t]({1.*3.788981138234188*cos(\t r)+0.*3.788981138234188*sin(\t r)},{0.*3.788981138234188*cos(\t r)+1.*3.788981138234188*sin(\t r)});
\draw [->,line width=1.pt] (9.660010073997823,-1.2040898016385486) -- (9.273711589110828,-1.6618696915712143);
\draw [shift={(7.503652953473375,0.2236771872164021)},line width=1.pt]  plot[domain=5.698337065222645:6.26428170876986,variable=\t]({1.*2.586193072007661*cos(\t r)+0.*2.586193072007661*sin(\t r)},{0.*2.586193072007661*cos(\t r)+1.*2.586193072007661*sin(\t r)});
\draw [shift={(7.503652953473331,0.2236771872164218)},line width=1.pt]  plot[domain=4.900260907454909:5.466205551002114,variable=\t]({1.*2.5861930720077093*cos(\t r)+0.*2.5861930720077093*sin(\t r)},{0.*2.5861930720077093*cos(\t r)+1.*2.5861930720077093*sin(\t r)});
\draw (14.920062254245062-0.17,-1.9) node[anchor=north west] {$2$};
\begin{scriptsize}
\draw [color=black, fill=white] (6.,0.42) circle (2.0pt);
\draw [color=black, fill=white] (9.748669107409894,0.42118702624658433) circle (2.0pt);
\draw [color=black, fill=white] (9.748669107409894,1.0011870262465843) circle (2.0pt);
\draw [color=black, fill=white] (13.497338214819786,0.4223740524931687) circle (2.0pt);
\draw [color=black, fill=white] (13.497338214819788,1.0023740524931686) circle (2.0pt);
\draw [color=black, fill=white] (5.580604894521872,1.5417015006620476) circle (2.0pt);
\draw [color=black, fill=white] (6.419395105478128,1.5417015006620476) circle (2.0pt);
\draw [color=black, fill=white] (6.,1.) circle (2.0pt);
\draw [color=black, fill=white] (9.329274001931765,1.5428885269086319) circle (2.0pt);
\draw [color=black, fill=white] (9.748669107409894,1.0011870262465843) circle (2.0pt);
\draw [color=black, fill=white] (10.168064212888023,1.5428885269086319) circle (2.0pt);
\draw [color=black, fill=white] (9.748669107409894,1.0011870262465843) circle (2.0pt);
\draw [color=black, fill=white] (13.916733320297917,1.5440755531552162) circle (2.0pt);
\draw [color=black, fill=white] (13.497338214819788,1.0023740524931686) circle (2.0pt);
\draw [color=black, fill=white] (13.077943109341659,1.5440755531552162) circle (2.0pt);
\draw [color=black, fill=white] (13.497338214819788,1.0023740524931686) circle (2.0pt);
\draw [color=black, fill=white] (7.006163048522667,-3.9023761150930785) circle (2.0pt);
\draw [color=black, fill=white] (7.006163048522667,-3.3223761150930784) circle (2.0pt);
\draw [color=black, fill=white] (7.006163048522667,-2.7423761150930783) circle (2.0pt);
\draw [color=black, fill=white] (11.006163048522668,-3.9023761150930785) circle (2.0pt);
\draw [color=black, fill=white] (11.006163048522666,-3.3223761150930784) circle (2.0pt);
\draw [color=black, fill=white] (11.006163048522668,-2.7423761150930783) circle (2.0pt);
\draw [color=black, fill=white] (15.006163048522668,-3.9023761150930785) circle (2.0pt);
\draw [color=black, fill=white] (15.006163048522666,-3.3223761150930784) circle (2.0pt);
\draw [color=black, fill=white] (15.006163048522668,-2.7423761150930783) circle (2.0pt);
\end{scriptsize}
\end{tikzpicture}
}
\end{center}
\clearpage
\begin{figure}[htb!]
\begin{center}\scalebox{0.75}{
\definecolor{aqaqaq}{rgb}{0.6274509803921569,0.6274509803921569,0.6274509803921569}
\definecolor{qqffqq}{rgb}{0.,1.,0.}
\definecolor{qqwuqq}{rgb}{0.,0.39215686274509803,0.}
\begin{tikzpicture}
\clip(1.,-4.3) rectangle (16.4,3.1);
\draw [line width=0.4pt,color=gray,fill=gray!30] (3.,1.3) circle (0.5989134213566999cm);
\draw (5.887288676594245-0.2,2.2802364886905377) node[anchor=north west] {$\moon\!\!_{0}$};
\draw [line width=1.pt] (5.68,0.42)-- (6.32,0.42);
\draw [line width=3.6pt,color=qqwuqq] (6.,0.42)-- (6.,1.);
\draw [line width=1.pt] (5.993836951477332,1.3223761150930786) circle (1.2cm);
\draw [line width=3.pt] (9.742506058887226,1.323563141339663) circle (1.2cm);
\draw [line width=1.pt] (9.428669107409894,0.4211870262465843)-- (10.068669107409894,0.4211870262465842);
\draw [line width=3.6pt,color=qqffqq] (9.748669107409894,0.42118702624658433)-- (9.748669107409894,1.0011870262465843);
\draw [line width=3.pt] (13.49117516629712,1.3247501675862472) circle (1.2cm);
\draw [line width=3.6pt,color=qqffqq] (13.497338214819786,0.4223740524931687)-- (13.497338214819788,1.0023740524931686);
\draw [line width=1.pt] (13.177338214819788,0.42237405249316856)-- (13.817338214819788,0.4223740524931685);
\draw [line width=1.pt] (3.,-3.) circle (1.2cm);
\draw [line width=1.pt] (2.6861630485226673,-3.9023761150930785)-- (3.326163048522668,-3.9023761150930785);
\draw [line width=1.pt] (3.,1.3) circle (0.5989134213566998cm);
\draw [->,line width=1.pt] (3.,0.701086578643299) -- (3.,-0.5494567106783511);
\draw [line width=1.pt] (3.,0.07581493398247402)-- (3.,-1.8);
\draw [->,line width=1.pt] (4.7938704670111685,1.3134074943567349) -- (4.196383580474334,1.3089418418992234);
\draw [line width=1.pt] (4.495127023742752,1.311174668127979)-- (3.598896693937501,1.3044761894417118);
\draw [line width=1.pt] (8.542506119048584,1.323183158101966)-- (7.193836891315973,1.3227560983307756);
\draw [->,line width=1.pt] (7.8681715051822785,1.3229696282163708) -- (7.36242054478255,1.3228094808021744);
\draw [->,line width=1.pt] (7.8681715051822785,1.3229696282163708) -- (8.373922465582009,1.323129775630567);
\draw [line width=1.pt] (12.291175226458455,1.3243701843485511)-- (10.942505998725869,1.3239431245773605);
\draw [->,line width=1.pt] (11.616840612592162,1.3241566544629557) -- (11.111089652192442,1.3239965070487592);
\draw [->,line width=1.pt] (11.616840612592162,1.3241566544629557) -- (12.122591572991881,1.3243168018771523);
\draw (9.647342749587647-0.45,2.2514238138016993) node[anchor=north west] {$\infty_{0,2}$};
\draw (13.39299048513663-0.45,2.222611138912861) node[anchor=north west] {$\infty_{0,2}$};
\draw (2.905176825599478-0.15,-1.9) node[anchor=north west] {$0$};
\draw (2.7,1.6) node[anchor=north west] {$\mathbb{Z}^{+}$};
\draw [line width=3.2pt,color=qqffqq] (5.580604894521872,1.5417015006620476)-- (6.,1.);
\draw [line width=3.2pt,color=qqffqq] (6.419395105478128,1.5417015006620476)-- (6.,1.);
\draw [line width=3.2pt,color=qqwuqq] (9.329274001931765,1.5428885269086319)-- (9.748669107409894,1.0011870262465843);
\draw [line width=3.2pt,color=qqffqq] (10.168064212888023,1.5428885269086319)-- (9.748669107409894,1.0011870262465843);
\draw [line width=3.2pt,color=qqwuqq] (13.916733320297917,1.5440755531552162)-- (13.497338214819788,1.0023740524931686);
\draw [line width=3.2pt,color=qqffqq] (13.077943109341659,1.5440755531552162)-- (13.497338214819788,1.0023740524931686);

\draw [shift={(9.74724040978494,-1.0479414212973992)},line width=1.pt]  plot[domain=0.8333549824263414:1.1282840804379441,variable=\t]({1.*4.034067430991071*cos(\t r)+0.*4.034067430991071*sin(\t r)},{0.*4.034067430991071*cos(\t r)+1.*4.034067430991071*sin(\t r)});
\draw [shift={(9.747240409785224,-1.047941421297077)},line width=1.pt]  plot[domain=2.015677385256428:2.3106064832680766,variable=\t]({1.*4.034067430990676*cos(\t r)+0.*4.034067430990676*sin(\t r)},{0.*4.034067430990676*cos(\t r)+1.*4.034067430990676*sin(\t r)});
\draw [shift={(9.747240409785205,-1.0479414212970595)},line width=1.pt]  plot[domain=1.2770516348355718:1.8669098308588559,variable=\t]({1.*4.034067430990652*cos(\t r)+0.*4.034067430990652*sin(\t r)},{0.*4.034067430990652*cos(\t r)+1.*4.034067430990652*sin(\t r)});

\draw [->,line width=1.pt] (8.570079069669216,2.8105545573344917) -- (8.01117739574482,2.5934573826051626);
\draw [->,line width=1.pt] (10.915258403785234,2.8133322032304315) -- (11.47467277157061,2.5975595694086078);
\draw [line width=1.pt] (7.,-3.) circle (1.2cm);
\draw [line width=1.pt] (11.,-3.) circle (1.2cm);
\draw [line width=1.pt] (6.686163048522666,-3.9023761150930785)-- (7.326163048522667,-3.9023761150930785);
\draw [line width=3.6pt,color=qqwuqq] (7.006163048522667,-3.9023761150930785)-- (7.006163048522667,-3.3223761150930784);
\draw [line width=3.6pt,color=qqffqq] (7.006163048522667,-3.3223761150930784)-- (7.006163048522667,-2.7423761150930783);
\draw (6.881325960259168-0.13,-1.9) node[anchor=north west] {$0$};
\draw (10.915100444696535-0.15,-1.9) node[anchor=north west] {$2$};
\draw [line width=1.pt] (10.686163048522666,-3.9023761150930785)-- (11.326163048522666,-3.9023761150930785);
\draw [line width=3.6pt,color=qqffqq] (11.006163048522668,-3.9023761150930785)-- (11.006163048522666,-3.3223761150930784);
\draw [line width=3.6pt,color=qqffqq] (11.006163048522666,-3.3223761150930784)-- (11.006163048522668,-2.7423761150930783);
\draw [line width=1.pt] (15.,-3.) circle (1.2cm);
\draw [line width=1.pt] (14.686163048522666,-3.9023761150930785)-- (15.326163048522666,-3.9023761150930785);
\draw [line width=3.6pt,color=qqffqq] (15.006163048522668,-3.9023761150930785)-- (15.006163048522666,-3.3223761150930784);
\draw [line width=3.6pt,color=qqwuqq] (15.006163048522666,-3.3223761150930784)-- (15.006163048522668,-2.7423761150930783);
\draw [line width=1.pt] (6.450899610002213,-0.5781985996476287)-- (6.727937730434487,-1.831247621829645);
\draw [->,line width=1.pt] (6.355801893019068,-0.14806970225423954) -- (6.450899610002213,-0.5781985996476287);
\draw [line width=1.pt] (6.355801893019068,-0.14806970225423954)-- (6.2879183379723305,0.158968990257927);
\draw [->,line width=1.pt] (10.077635194268451,0.17130949625164418) -- (10.371253029443611,-0.8382184293301669);
\draw [line width=1.pt] (10.22444411185603,-0.33345446653926136)-- (10.664870864618774,-1.847746354911978);
\draw [->,line width=1.pt] (13.886466377042495,0.19172551705675359) -- (14.245587583148563,-0.8376249162068836);
\draw [line width=1.pt] (14.066026980095529,-0.32294969957506503)-- (14.604708789254632,-1.8669753494705208);
\draw [->,line width=1.pt] (12.381742274262356,0.8674114986860584) -- (8.245587583148556,-0.8376249162068774);
\draw [line width=1.pt] (10.313664928705457,0.01489329123959049)-- (4.109432892034756,-2.542661331099813);
\draw [->,line width=1.pt] (11.367586843804855,-0.11853015131897837) -- (11.779128466571047,-0.46033116718166933);
\draw [line width=1.pt] (11.367586843804855,-0.11853015131897837)-- (10.618602671872981,0.5035297732899917);
\draw [line width=1.pt] (11.779128466571047,-0.46033116718166933)-- (14.013806733856967,-2.3163167094230426);

\draw [->,line width=1.pt] (12.892207023690888,0.28492462253415757) -- (11.940084553120618,-1.367986770397689);
\draw [line width=1.pt] (11.598968142606227,-1.9601744549479223)-- (12.003341650841005,-1.2581706692786767);
\draw [->,line width=1.pt] (8.73234938514145,0.6758104921850077) -- (4.961916269532205,-1.7419412314560105);
\draw [line width=1.pt] (4.010156673745773,-2.3522473508453463)-- (5.084967708723684,-1.663035752939483);

\draw (14.920062254245062-0.15,-1.9) node[anchor=north west] {$2$};
\begin{scriptsize}
\draw [color=black, fill=white] (6.,0.42) circle (2.0pt);
\draw [color=black, fill=white] (9.748669107409894,0.42118702624658433) circle (2.0pt);
\draw [color=black, fill=white] (9.748669107409894,1.0011870262465843) circle (2.0pt);
\draw [color=black, fill=white] (13.497338214819786,0.4223740524931687) circle (2.0pt);
\draw [color=black, fill=white] (13.497338214819788,1.0023740524931686) circle (2.0pt);
\draw [color=black, fill=white] (5.580604894521872,1.5417015006620476) circle (2.0pt);
\draw [color=black, fill=white] (6.419395105478128,1.5417015006620476) circle (2.0pt);
\draw [color=black, fill=white] (6.,1.) circle (2.0pt);
\draw [color=black, fill=white] (9.329274001931765,1.5428885269086319) circle (2.0pt);
\draw [color=black, fill=white] (9.748669107409894,1.0011870262465843) circle (2.0pt);
\draw [color=black, fill=white] (10.168064212888023,1.5428885269086319) circle (2.0pt);
\draw [color=black, fill=white] (9.748669107409894,1.0011870262465843) circle (2.0pt);
\draw [color=black, fill=white] (13.916733320297917,1.5440755531552162) circle (2.0pt);
\draw [color=black, fill=white] (13.497338214819788,1.0023740524931686) circle (2.0pt);
\draw [color=black, fill=white] (13.077943109341659,1.5440755531552162) circle (2.0pt);
\draw [color=black, fill=white] (13.497338214819788,1.0023740524931686) circle (2.0pt);
\draw [color=black, fill=white] (7.006163048522667,-3.9023761150930785) circle (2.0pt);
\draw [color=black, fill=white] (7.006163048522667,-3.3223761150930784) circle (2.0pt);
\draw [color=black, fill=white] (7.006163048522667,-2.7423761150930783) circle (2.0pt);
\draw [color=black, fill=white] (11.006163048522668,-3.9023761150930785) circle (2.0pt);
\draw [color=black, fill=white] (11.006163048522666,-3.3223761150930784) circle (2.0pt);
\draw [color=black, fill=white] (11.006163048522668,-2.7423761150930783) circle (2.0pt);
\draw [color=black, fill=white] (15.006163048522668,-3.9023761150930785) circle (2.0pt);
\draw [color=black, fill=white] (15.006163048522666,-3.3223761150930784) circle (2.0pt);
\draw [color=black, fill=white] (15.006163048522668,-2.7423761150930783) circle (2.0pt);
\end{scriptsize}
\end{tikzpicture}
}
\end{center}
\vspace{-0.5cm}
    \caption{``Glasses'' with three edges of different colors.}
    \label{fig:fig14}
\end{figure}
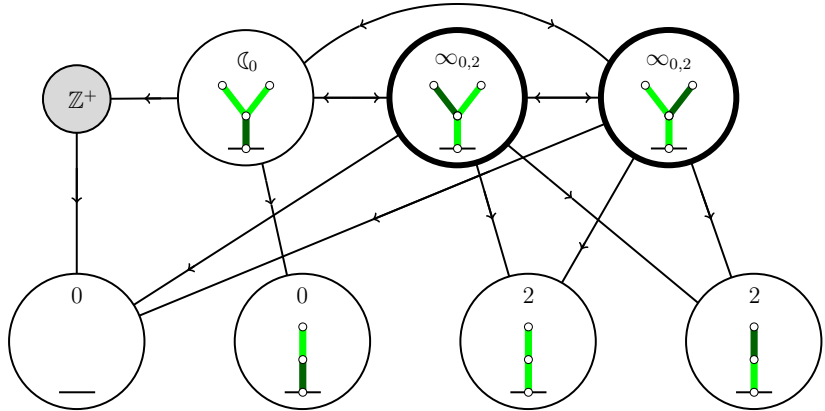

If we consider ``glasses'' with three edges, like those served by the bartender in Figure \ref{fig:fig11}, the situation is already different. The game digraphs of these positions are shown in Figure \ref{fig:fig14}; the case with two lime edges and one green edge contains a cyclic zone with two nodes, both with Grundy value $\infty_{0,2}$.

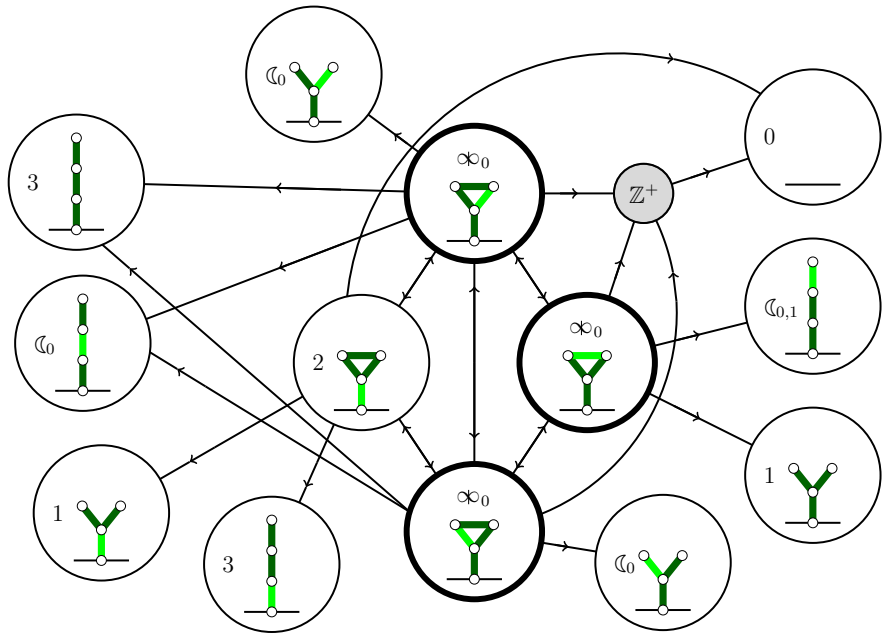
\begin{figure}[htb!]
\begin{center}
\scalebox{0.75}{
\definecolor{aqaqaq}{rgb}{0.6274509803921569,0.6274509803921569,0.6274509803921569}
\definecolor{qqffqq}{rgb}{0.,1.,0.}
\definecolor{qqwuqq}{rgb}{0.,0.39215686274509803,0.}
\begin{tikzpicture}
\clip(1.6,-5.9) rectangle (17.3,5.4);
\draw [line width=1.pt,color=gray,fill=gray!30] (13.,2.) circle (0.5252469384162191cm);
\draw [line width=1.pt] (16.,-3.) circle (1.2cm);
\draw [line width=1.pt] (7.155598004310521,4.114213889601654) circle (1.2cm);
\draw [line width=1.pt] (3.3890754935959952,-3.6677202845177903) circle (1.2cm);
\draw [line width=3.6pt,color=qqwuqq] (16.,-3.3006610325243906)-- (15.660053550746245,-2.8774978778865385);
\draw [line width=3.6pt,color=qqwuqq] (16.,-3.3006610325243906)-- (16.339946449253755,-2.8774978778865385);
\draw [line width=3.6pt,color=qqwuqq] (16.,-3.3006610325243906)-- (16.,-3.843459930647162);
\draw [line width=1.pt] (15.510242998265113,-3.843459930647162)-- (16.489757001734887,-3.843459930647162);
\draw [line width=3.6pt,color=qqwuqq] (7.155598004310521,3.8135528570772634)-- (7.155598004310521,3.270753958954492);
\draw [line width=3.6pt,color=qqwuqq] (7.155598004310521,3.8135528570772634)-- (6.815651555056766,4.2367160117151155);
\draw [line width=3.6pt,color=qqffqq] (7.155598004310521,3.8135528570772634)-- (7.495544453564277,4.2367160117151155);
\draw [line width=1.pt] (6.665841002575634,3.270753958954492)-- (7.645355006045408,3.270753958954492);
\draw [line width=1.pt] (2.8993184918611092,-4.511180215164952)-- (3.878832495330883,-4.511180215164952);
\draw [line width=3.6pt,color=qqffqq] (3.389075493595996,-3.968381317042181)-- (3.389075493595996,-4.511180215164952);
\draw [line width=3.6pt,color=qqwuqq] (3.389075493595996,-3.968381317042181)-- (3.7290219428497515,-3.545218162404329);
\draw [line width=3.6pt,color=qqwuqq] (3.389075493595996,-3.968381317042181)-- (3.0491290443422407,-3.545218162404329);
\draw [line width=1.pt] (16.,3.) circle (1.2cm);
\draw [line width=1.pt] (15.510242998265113,2.156540069352838)-- (16.489757001734887,2.156540069352838);
\draw [line width=3pt] (10.,-4.) circle (1.2cm);
\draw [line width=3pt] (10.,2.) circle (1.2cm);
\draw [line width=3pt] (12.,-1.) circle (1.2cm);
\draw [line width=1.pt] (8.,-1.) circle (1.2cm);
\draw [line width=1.pt] (11.510242998265113,-1.8434599306471622)-- (12.489757001734887,-1.8434599306471622);
\draw [line width=3.6pt,color=qqwuqq] (12.,-1.3006610325243906)-- (12.,-1.8434599306471622);
\draw [line width=3.6pt,color=qqwuqq] (12.,-1.3006610325243906)-- (12.339946449253755,-0.8774978778865385);
\draw [line width=3.6pt,color=qqwuqq] (12.,-1.3006610325243906)-- (11.660053550746245,-0.8774978778865385);
\draw [line width=3.6pt,color=qqffqq] (11.660053550746245,-0.8774978778865385)-- (12.339946449253755,-0.8774978778865385);
\draw [line width=1.pt] (9.510242998265113,-4.843459930647162)-- (10.489757001734887,-4.843459930647162);
\draw [line width=3.6pt,color=qqwuqq] (10.,-4.30066103252439)-- (10.,-4.843459930647162);
\draw [line width=3.6pt,color=qqwuqq] (10.,-4.30066103252439)-- (10.339946449253755,-3.8774978778865385);
\draw [line width=3.6pt,color=qqffqq] (10.,-4.30066103252439)-- (9.660053550746245,-3.8774978778865385);
\draw [line width=3.6pt,color=qqwuqq] (9.660053550746245,-3.8774978778865385)-- (10.339946449253755,-3.8774978778865385);
\draw [line width=1.pt] (9.510242998265113,1.1565400693528378)-- (10.489757001734887,1.1565400693528378);
\draw [line width=3.6pt,color=qqwuqq] (10.,1.6993389674756094)-- (10.,1.1565400693528378);
\draw [line width=3.6pt,color=qqffqq] (10.,1.6993389674756094)-- (10.339946449253755,2.1225021221134615);
\draw [line width=3.6pt,color=qqwuqq] (10.,1.6993389674756094)-- (9.660053550746245,2.1225021221134615);
\draw [line width=3.6pt,color=qqwuqq] (9.660053550746245,2.1225021221134615)-- (10.339946449253755,2.1225021221134615);
\draw [line width=1.pt] (7.510242998265113,-1.8434599306471622)-- (8.489757001734887,-1.8434599306471622);
\draw [line width=3.6pt,color=qqffqq] (8.,-1.3006610325243906)-- (8.,-1.8434599306471622);
\draw [line width=3.6pt,color=qqwuqq] (8.,-1.3006610325243906)-- (8.339946449253755,-0.8774978778865385);
\draw [line width=3.6pt,color=qqwuqq] (8.,-1.3006610325243906)-- (7.660053550746245,-0.8774978778865385);
\draw [line width=3.6pt,color=qqwuqq] (7.660053550746245,-0.8774978778865385)-- (8.339946449253755,-0.8774978778865385);
\draw [line width=1.pt] (6.40893054623936,-4.580313844382544) circle (1.2cm);
\draw [line width=1.pt] (3.05990670689117,-0.6580874495582656) circle (1.2cm);
\draw [line width=1.pt] (16.,0.) circle (1.2cm);
\draw [line width=1.pt] (15.510242998265113,-0.8434599306471622)-- (16.489757001734887,-0.8434599306471622);
\draw [line width=3.6pt,color=qqwuqq] (16.,-0.3006610325243906)-- (16.,-0.8434599306471622);
\draw [line width=3.6pt,color=qqwuqq] (16.,-0.3006610325243906)-- (16.,0.24213786559838102);
\draw [line width=3.6pt,color=qqffqq] (16.,0.24213786559838102)-- (16.,0.7849367637211526);
\draw [line width=1.pt] (2.5701497051562825,-1.5015473802054276)-- (3.5496637086260563,-1.5015473802054276);
\draw [line width=3.6pt,color=qqwuqq] (3.0599067068911694,-0.958748482082656)-- (3.0599067068911694,-1.5015473802054276);
\draw [line width=3.6pt,color=qqffqq] (3.0599067068911694,-0.958748482082656)-- (3.0599067068911694,-0.41594958395988435);
\draw [line width=3.6pt,color=qqwuqq] (3.0599067068911694,-0.41594958395988435)-- (3.0599067068911694,0.12684931416288725);
\draw [line width=1.pt] (5.919173544504474,-5.423773775029706)-- (6.898687547974248,-5.423773775029706);
\draw [line width=3.6pt,color=qqffqq] (6.408930546239361,-4.880974876906935)-- (6.408930546239361,-5.423773775029706);
\draw [line width=3.6pt,color=qqwuqq] (6.408930546239361,-4.880974876906935)-- (6.408930546239361,-4.338175978784164);
\draw [line width=3.6pt,color=qqwuqq] (6.408930546239361,-4.338175978784164)-- (6.408930546239361,-3.7953770806613925);
\draw [line width=1.pt] (13.344641601211482,-4.547128624023825) circle (1.2cm);
\draw [line width=3.6pt,color=qqffqq] (13.344641601211482,-4.8477896565482155)-- (13.004695151957726,-4.424626501910364);
\draw [line width=3.6pt,color=qqwuqq] (13.344641601211482,-4.8477896565482155)-- (13.684588050465237,-4.424626501910364);
\draw [line width=1.pt] (12.854884599476595,-5.3905885546709875)-- (13.834398602946369,-5.3905885546709875);
\draw [line width=3.6pt,color=qqwuqq] (13.344641601211482,-4.8477896565482155)-- (13.344641601211482,-5.3905885546709875);
\draw [line width=1.pt] (10.,-2.8)-- (10.,0.8);
\draw [->,line width=1.pt] (10.,-1.) -- (10.,0.35);
\draw [->,line width=1.pt] (10.,-1.) -- (10.,-2.35);
\draw [line width=1.pt] (8.665640235470272,-0.0015396467945872072)-- (9.334359764529724,1.0015396467945894);
\draw [->,line width=1.pt] (9.,0.5) -- (9.250769823397292,0.8761547350959423);
\draw [->,line width=1.pt] (9.,0.5) -- (8.749230176602705,0.12384526490405989);
\draw [line width=1.pt] (11.334359764529722,-0.0015396467945841255)-- (10.665640235470272,1.0015396467945907);
\draw [->,line width=1.pt] (11.,0.5) -- (10.749230176602705,0.8761547350959439);
\draw [->,line width=1.pt] (11.,0.5) -- (11.250769823397292,0.12384526490406272);
\draw [line width=1.pt] (9.334359764529726,-3.0015396467945923)-- (8.665640235470272,-1.9984603532054128);
\draw [->,line width=1.pt] (9.,-2.5) -- (8.749230176602705,-2.1238452649040602);
\draw [->,line width=1.pt] (9.,-2.5) -- (9.250769823397295,-2.8761547350959447);
\draw [line width=1.pt] (11.334359764529722,-1.9984603532054142)-- (10.665640235470272,-3.0015396467945896);
\draw [->,line width=1.pt] (11.,-2.5) -- (10.749230176602705,-2.876154735095943);
\draw [->,line width=1.pt] (11.,-2.5) -- (11.250769823397292,-2.123845264904061);
\draw [->,line width=1.pt] (13.498292997827614,2.1660976659425377) -- (14.179936520083515,2.3933121733611715);
\draw [line width=1.pt] (13.839114758955564,2.279704919651855)-- (14.861580042339416,2.6205266807798053);
\draw [->,line width=1.pt] (11.2,2.) -- (11.837376530791893,2.);
\draw [line width=1.pt] (11.518688265395946,2.)-- (12.474753061583787,2.);
\draw [->,line width=1.pt] (12.379473319220207,0.1384199576606214) -- (12.60668782663883,0.8200634799164915);
\draw [line width=1.pt] (12.493080572929518,0.4792417187885565)-- (12.833902334057454,1.5017070021723617);
\draw [->,line width=1.pt] (13.164171000174397,-0.7089572499564007) -- (14.,-0.5);
\draw [line width=1.pt] (13.58208550008719,-0.6044786249782025)-- (14.835828999825571,-0.2910427500436077);
\draw [->,line width=1.pt] (13.0733126291999,-1.5366563145999486) -- (14.,-2.);
\draw [line width=1.pt] (13.536656314599952,-1.7683281572999732)-- (14.926687370800103,-2.463343685400047);
\draw [line width=1.pt] (2.945633254654797,2.1987488563510587) circle (1.2cm);
\draw [line width=1.pt] (2.4558762529199107,1.3552889257038965)-- (3.4353902563896845,1.3552889257038965);
\draw [line width=3.6pt,color=qqwuqq] (2.9456332546547976,1.898087823826668)-- (2.9456332546547976,1.3552889257038965);
\draw [line width=3.6pt,color=qqwuqq] (2.9456332546547976,1.898087823826668)-- (2.9456332546547976,2.4408867219494397);
\draw [line width=3.6pt,color=qqwuqq] (2.9456332546547976,2.4408867219494397)-- (2.9456332546547976,2.9836856200722113);
\draw [->,line width=1.pt] (8.800475977042336,2.0337952414914873) -- (6.472816627327396,2.0993744281755293);
\draw [line width=1.pt] (7.636646302184866,2.0665848348335083)-- (4.145157277612457,2.1649536148595714);
\draw [->,line width=1.pt] (8.87938157214269,1.5707980205699936) -- (6.529953353445579,0.6709562752208661);
\draw [line width=1.pt] (7.704667462794134,1.1208771478954298)-- (4.1805251347484695,-0.22888547012826144);
\draw [->,line width=1.pt] (9.036906708584448,2.7158570472031185) -- (8.577799002155265,3.0571069448008217);
\draw [line width=1.pt] (8.807352855369857,2.88648199600197)-- (8.118691295726082,3.398356842398525);
\draw [line width=1.pt] (13.,2.) circle (0.5252469384162191cm);
\draw [->,line width=1.pt] (11.184259450000571,-4.193725463180069) -- (11.672320800605736,-4.273564312011911);
\draw [line width=1.pt] (11.428290125303153,-4.23364488759599)-- (12.160382151210902,-4.353403160843753);
\draw [->,line width=1.pt] (13.583265964053293,0.08081865403246966) -- (13.519934345924606,0.6047138158230274);
\draw [shift={(9.736765419148554,-0.11839407792675105)},line width=1.pt]  plot[domain=-1.1919996642137018:0.051744410689094064,variable=\t]({1.*3.851655767916846*cos(\t r)+0.*3.851655767916846*sin(\t r)},{0.*3.851655767916846*cos(\t r)+1.*3.851655767916846*sin(\t r)});
\draw [shift={(9.736765419148739,-0.11839407792671236)},line width=1.pt]  plot[domain=0.18886020710924725:0.44037443680116023,variable=\t]({1.*3.8516557679166565*cos(\t r)+0.*3.8516557679166565*sin(\t r)},{0.*3.8516557679166565*cos(\t r)+1.*3.8516557679166565*sin(\t r)});
\draw [->,line width=1.pt] (5.398899149496775,-2.504906335232014) -- (4.9540086922955995,-2.762304454723146);
\draw [->,line width=1.pt] (7.210795212571324,-2.775912936927613) -- (7.01819756391084,-3.209306982841085);
\draw [line width=1.pt] (5.398899149496775,-2.504906335232014)-- (6.961316360543092,-1.600946168241856);
\draw [line width=1.pt] (4.9540086922955995,-2.762304454723146)-- (4.4277591330529,-3.0667741162759365);
\draw [line width=1.pt] (7.210795212571324,-2.775912936927613)-- (7.512680171092563,-2.0965944484419134);
\draw [line width=1.pt] (7.01819756391084,-3.209306982841085)-- (6.896250375146798,-3.4837193959406294);
\draw [->,line width=1.pt] (4.331430080887537,0.28493434403221507) -- (3.8883759779073825,0.6703261913778793);
\draw [->,line width=1.pt] (5.234733263102365,-1.4235275245174166) -- (4.73333196087114,-1.1152556146499855);
\draw [line width=1.pt] (3.8883759779073825,0.6703261913778793)-- (3.4054966747739197,1.0903601001665044);
\draw [line width=1.pt] (4.73333196087114,-1.1152556146499855)-- (4.2492531338424095,-0.8176339214663589);
\draw [line width=1.pt] (5.234733263102365,-1.4235275245174166)-- (8.852792276049836,-3.647985173452758);
\draw [line width=1.pt] (8.852792276049836,-3.647985173452758)-- (4.331430080887537,0.28493434403221507);
\begin{scriptsize}
\draw [color=black,fill=white] (16.,-3.843459930647162) circle (2.5pt);
\draw [color=black,fill=white] (16.,-3.3006610325243906) circle (2.5pt);
\draw [color=black,fill=white] (15.660053550746245,-2.8774978778865385) circle (2.5pt);
\draw [color=black,fill=white] (16.339946449253755,-2.8774978778865385) circle (2.5pt);
\draw [color=black,fill=white] (7.155598004310521,3.270753958954492) circle (2.5pt);
\draw [color=black,fill=white] (6.815651555056766,4.2367160117151155) circle (2.5pt);
\draw [color=black,fill=white] (7.155598004310521,3.8135528570772634) circle (2.5pt);
\draw [color=black,fill=white] (7.495544453564277,4.2367160117151155) circle (2.5pt);
\draw [color=black,fill=white] (3.389075493595996,-4.511180215164952) circle (2.5pt);
\draw [color=black,fill=white] (3.7290219428497515,-3.545218162404329) circle (2.5pt);
\draw [color=black,fill=white] (3.389075493595996,-3.968381317042181) circle (2.5pt);
\draw [color=black,fill=white] (3.0491290443422407,-3.545218162404329) circle (2.5pt);
\draw [color=black,fill=white] (12.,-1.8434599306471622) circle (2.5pt);
\draw [color=black,fill=white] (12.339946449253755,-0.8774978778865385) circle (2.5pt);
\draw [color=black,fill=white] (12.,-1.3006610325243906) circle (2.5pt);
\draw [color=black,fill=white] (11.660053550746245,-0.8774978778865385) circle (2.5pt);
\draw [color=black,fill=white] (10.,-4.843459930647162) circle (2.5pt);
\draw [color=black,fill=white] (10.339946449253755,-3.8774978778865385) circle (2.5pt);
\draw [color=black,fill=white] (10.,-4.30066103252439) circle (2.5pt);
\draw [color=black,fill=white] (9.660053550746245,-3.8774978778865385) circle (2.5pt);
\draw [color=black,fill=white] (10.,1.1565400693528378) circle (2.5pt);
\draw [color=black,fill=white] (10.339946449253755,2.1225021221134615) circle (2.5pt);
\draw [color=black,fill=white] (10.,1.6993389674756094) circle (2.5pt);
\draw [color=black,fill=white] (9.660053550746245,2.1225021221134615) circle (2.5pt);
\draw [color=black,fill=white] (8.,-1.8434599306471622) circle (2.5pt);
\draw [color=black,fill=white] (8.339946449253755,-0.8774978778865385) circle (2.5pt);
\draw [color=black,fill=white] (8.,-1.3006610325243906) circle (2.5pt);
\draw [color=black,fill=white] (7.660053550746245,-0.8774978778865385) circle (2.5pt);
\draw [color=black,fill=white] (16.,-0.3006610325243906) circle (2.5pt);
\draw [color=black,fill=white] (16.,-0.8434599306471622) circle (2.5pt);
\draw [color=black,fill=white] (16.,0.24213786559838102) circle (2.5pt);
\draw [color=black,fill=white] (16.,0.7849367637211526) circle (2.5pt);
\draw [color=black,fill=white] (3.0599067068911694,-0.958748482082656) circle (2.5pt);
\draw [color=black,fill=white] (3.0599067068911694,-1.5015473802054276) circle (2.5pt);
\draw [color=black,fill=white] (3.0599067068911694,-0.41594958395988435) circle (2.5pt);
\draw [color=black,fill=white] (3.0599067068911694,0.12684931416288725) circle (2.5pt);
\draw [color=black,fill=white] (6.408930546239361,-4.880974876906935) circle (2.5pt);
\draw [color=black,fill=white] (6.408930546239361,-5.423773775029706) circle (2.5pt);
\draw [color=black,fill=white] (6.408930546239361,-4.338175978784164) circle (2.5pt);
\draw [color=black,fill=white] (6.408930546239361,-3.7953770806613925) circle (2.5pt);
\draw [color=black,fill=white] (13.004695151957726,-4.424626501910364) circle (2.5pt);
\draw [color=black,fill=white] (13.684588050465237,-4.424626501910364) circle (2.5pt);
\draw [color=black,fill=white] (13.344641601211482,-4.8477896565482155) circle (2.5pt);
\draw [color=black,fill=white] (13.344641601211482,-5.3905885546709875) circle (2.5pt);
\draw [color=black,fill=white] (2.9456332546547976,1.898087823826668) circle (2.5pt);
\draw [color=black,fill=white] (2.9456332546547976,1.3552889257038965) circle (2.5pt);
\draw [color=black,fill=white] (2.9456332546547976,2.4408867219494397) circle (2.5pt);
\draw [color=black,fill=white] (2.9456332546547976,2.9836856200722113) circle (2.5pt);
\end{scriptsize}
\draw (1.9314563660569894,2.4844324869419907) node[anchor=north west] {$3$};
\draw (14.98719828406257,-2.7150095898129805) node[anchor=north west] {$1$};
\draw (14.98719828406257,3.2843466525966014) node[anchor=north west] {$0$};
\draw (14.98719828406257,0.2846685313918107) node[anchor=north west] {$\moon\!\!_{0,1}$};
\draw (12.330340519566905,-4.257701195004016) node[anchor=north west] {$\moon\!\!_{0}$};
\draw (6.1452899172732325,4.398512811901238) node[anchor=north west] {$\moon\!\!_{0}$};
\draw (2.045729818293362,-0.372403818967334) node[anchor=north west] {$\moon\!\!_{0}$};
\draw (2.3742659934729335,-3.3863661217016716) node[anchor=north west] {$1$};
\draw (5.388228296207264,-4.300553739592655) node[anchor=north west] {$3$};
\draw (8.987842041653003+0.55,2.284453945528338+0.6) node[anchor=north west] {$\nym_{0}$};
\draw (6.988056627516481,-0.7152241756764529) node[anchor=north west] {$2$};
\draw (10.987627455789525+0.55,-0.7152241756764529+0.6) node[anchor=north west] {$\nym_{0}$};
\draw (8.987842041653003+0.55,-3.714902296881244+0.6) node[anchor=north west] {$\nym_{0}$};
\draw (12.6,2.3) node[anchor=north west] {$\mathbb{Z}^{+}$};

\draw [shift={(12.542807531653969,-0.3404158618672925)},line width=1.pt]  plot[domain=1.4821004386769168:3.035127634390425,variable=\t]({1.*4.825969297260067*cos(\t r)+0.*4.825969297260067*sin(\t r)},{0.*4.825969297260067*cos(\t r)+1.*4.825969297260067*sin(\t r)});
\draw [shift={(12.542807531654004,-0.3404158618671967)},line width=1.pt]  plot[domain=1.0179753179849524:1.3630962308936114,variable=\t]({1.*4.825969297259968*cos(\t r)+0.*4.825969297259968*sin(\t r)},{0.*4.825969297259968*cos(\t r)+1.*4.825969297259968*sin(\t r)});
\draw [->,line width=1.pt] (12.970290151493145,4.466583021826388) -- (13.537970531055228,4.381832571106088);
\end{tikzpicture}
}
\end{center}
\vspace{-0.5cm}
    \caption{``Full Glasses''.}
    \label{fig:fig15}
\end{figure}
\clearpage
To find the solution to the problem shown in Figure \ref{fig:fig11}, it is still necessary to analyze ``full glasses'' with four edges. For this purpose, a game digraph is shown in Figure \ref{fig:fig15}. It is worth noting that there is a cyclic zone with three nodes, all of them having Grundy value $\nym_0$, i.e., a nymphet.

In the disjunctive sum under analysis, the Grundy values of the components are, from left to right, $\nym_0$, $2$, $\infty_{0,2}$, and $\moon_{\!\!\!0}$. By applying Theorem \ref{th:cycliccarrysum}, the Grundy value of the sum is $\infty_{0,2}$, and therefore the next player wins by moving out of the cyclic zone to $0$. This can be done by removing the green edges at the bottom of the leftmost and rightmost components. Since these moves are carry-on moves, the nymphet and the moon are removed from the game, while the player retains the right to move. Next, it is enough to remove any of the top edges of the third component. This ends the cyclic situation, obtaining $*2 + *2 = 0$ and thus winning the game. This is the solution to the problem ``served'' by the bartender.

\section{Final remarks}
\label{sec:final}

In conclusion, the main features of the unified theory proposed here are summarized in Figure \ref{fig:fig16}. This figure shows, side by side, how SFPT and LNST intersect, as well as how the generalization underlying cycles and carry-on moves works. In the game digraphs shown, the arrows are moves that force the opponent to play in the same component.

Starting from the left side of the figure, although independent, ordinary {\sc nim} values play a central role in both theories. What differs in LNST concerns forms such as $G=\{\infty\}$, $G'=\{0,\circlearrowright^{0}\}$ or $G''=\{\circlearrowright^{0},\circlearrowright^{*}\}$. In the first case, the next player has a ``checkmate'' available, namely a move to a terminal gray node. In the second, the next player can either move to zero or make a carry-on move to zero, retaining the right to play. In the third, the next player has carry-on moves to distinct {\sc nim} values. In LNST, $G+X\in\mathcal{N}$, $G'+X\in\mathcal{N}$, and $G''+X\in\mathcal{N}$ hold for any impartial game $X$. These outcomes are explained differently: in the first case, the next player delivers checkmate in $G$; in the second case, in $G'$, the next player either moves to zero or makes a carry-on move to zero, depending on whether $X\in\mathcal{P}$ or $X\in\mathcal{N}$; in the third case, the next player stabilizes the second component at a {\sc nim} value (using forcing sequences if necessary) and then makes a carry-on move in the first component to a different {\sc nim} value. Due to these game situations, in LNST, $G$, $G'$, and $G''$ are \emph{indistinguishable} and, as such, belong to the same equivalence class. Consequently, besides the {\sc nim} values, there is a single lunar value, denoted by $\moon\!\!$. Nevertheless, for the analysis of Figure~\ref{fig:fig16}, it is important to recognize that these represent distinct game situations.

SFPT already accounts for cyclic zones. A node within a cyclic zone always belongs to the equivalence class of a form of the type $H=\{H,*n_1,\ldots,*n_k\}$. In these forms, the fact that $H$ is an option of itself indicates the presence of a {\sc pass} move; for this reason, in the literature, a form of this type is often represented as $H=\{\text{{\sc pass}},*n_1,\ldots,*n_k\}$.

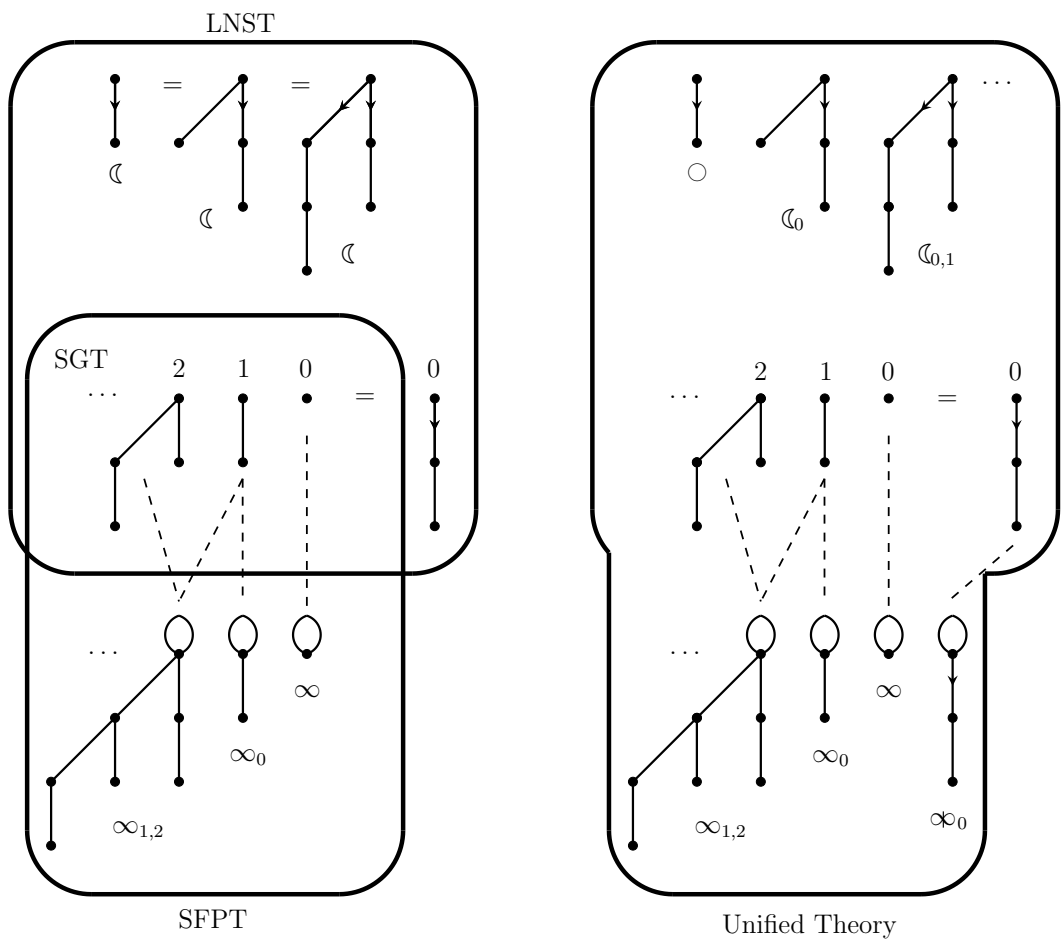
\begin{figure}[htb!]
\scalebox{0.85}{
\begin{tikzpicture}
\clip(2.2,-1) rectangle (19.,14.8);
\draw [->, midway,line width=1.pt] (9.,8.) -- (9.,7.);
\draw [line width=1.pt] (9.,7.)-- (9.,6.);
\draw [line width=1.pt] (6.,8.)-- (6.,7.);
\draw [line width=1.pt] (5.,8.)-- (5.,7.);
\draw [line width=1.pt] (5.,8.)-- (4.,7.);
\draw [line width=1.pt] (4.,7.)-- (4.,6.);
\draw [shift={(8.647639510300412,6.258147401543826)},line width=2pt]  plot[domain=-1.5707963267948966:0.,variable=\t]({1.*1.*cos(\t r)+0.*1.*sin(\t r)},{0.*1.*cos(\t r)+1.*1.*sin(\t r)});
\draw [shift={(8.647639510300412,12.575075470770422)},line width=2pt]  plot[domain=0.:1.5707963267948966,variable=\t]({1.*1.*cos(\t r)+0.*1.*sin(\t r)},{0.*1.*cos(\t r)+1.*1.*sin(\t r)});
\draw [line width=2pt] (9.647639510300412,6.258147401543826)-- (9.647639510300412,12.575075470770422);
\draw [line width=2pt] (2.3653247455627158,6.258147401543827)-- (2.365324745562716,12.575075470770422);
\draw [shift={(3.365324745562716,12.575075470770422)},line width=2pt]  plot[domain=0.:1.5707963267948966,variable=\t]({-1.*1.*cos(\t r)+0.*1.*sin(\t r)},{0.*1.*cos(\t r)+1.*1.*sin(\t r)});
\draw [shift={(3.3653247455627158,6.258147401543826)},line width=2pt]  plot[domain=-1.5707963267948966:0.,variable=\t]({-1.*1.*cos(\t r)+0.*1.*sin(\t r)},{0.*1.*cos(\t r)+1.*1.*sin(\t r)});
\draw [line width=2pt] (3.365324745562715,13.575075470770422)-- (8.647639510300412,13.575075470770422);
\draw [line width=2pt] (8.647639510300412,5.258147401543826)-- (3.3653247455627153,5.258147401543826);
\draw [shift={(7.5054004621662855,1.2392182506514626)},line width=2pt]  plot[domain=-1.5707963267948966:0.,variable=\t]({1.*1.*cos(\t r)+0.*1.*sin(\t r)},{0.*1.*cos(\t r)+1.*1.*sin(\t r)});
\draw [shift={(7.5054004621662855,8.300332366389684)},line width=2pt]  plot[domain=0.:1.5707963267948966,variable=\t]({1.*1.*cos(\t r)+0.*1.*sin(\t r)},{0.*1.*cos(\t r)+1.*1.*sin(\t r)});
\draw [line width=2pt] (2.6249245292295615,1.239218250651463)-- (2.6249245292295624,8.300332366389684);
\draw [shift={(3.6249245292295624,8.300332366389684)},line width=2pt]  plot[domain=0.:1.5707963267948966,variable=\t]({-1.*1.*cos(\t r)+0.*1.*sin(\t r)},{0.*1.*cos(\t r)+1.*1.*sin(\t r)});
\draw [shift={(3.6249245292295615,1.2392182506514628)},line width=2pt]  plot[domain=-1.5707963267948966:0.,variable=\t]({-1.*1.*cos(\t r)+0.*1.*sin(\t r)},{0.*1.*cos(\t r)+1.*1.*sin(\t r)});
\draw [line width=2pt] (3.624924529229562,9.300332366389684)-- (7.5054004621662855,9.300332366389684);
\draw [line width=2pt] (7.5054004621662855,0.23921825065146263)-- (3.6249245292295615,0.23921825065146288);
\draw [line width=2pt] (8.505400462166286,1.2392182506514626)-- (8.505400462166286,8.300332366389684);
\draw [->, midway,line width=1.pt] (4.,13.) -- (4.,12.);
\draw [shift={(7.1,4.3)},line width=1.pt]  plot[domain=1.892546881191538:4.390638425988048,variable=\t]({1.*0.3162277660168376*cos(\t r)+0.*0.3162277660168376*sin(\t r)},{0.*0.3162277660168376*cos(\t r)+1.*0.3162277660168376*sin(\t r)});
\draw [shift={(6.9,4.3)},line width=1.pt]  plot[domain=-1.249045772398258:1.249045772398258,variable=\t]({1.*0.3162277660168374*cos(\t r)+0.*0.3162277660168374*sin(\t r)},{0.*0.3162277660168374*cos(\t r)+1.*0.3162277660168374*sin(\t r)});
\draw [line width=0.8pt,dash pattern=on 4pt off 4pt] (7.,4.75247)-- (7.,7.41769);
\draw [line width=1.pt] (6.,4.)-- (6.,3.);
\draw [shift={(5.9,4.3)},line width=1.pt]  plot[domain=-1.2490457723982553:1.2490457723982553,variable=\t]({1.*0.3162277660168376*cos(\t r)+0.*0.3162277660168376*sin(\t r)},{0.*0.3162277660168376*cos(\t r)+1.*0.3162277660168376*sin(\t r)});
\draw [shift={(6.1,4.3)},line width=1.pt]  plot[domain=1.8925468811915354:4.390638425988051,variable=\t]({1.*0.3162277660168374*cos(\t r)+0.*0.3162277660168374*sin(\t r)},{0.*0.3162277660168374*cos(\t r)+1.*0.3162277660168374*sin(\t r)});
\draw [line width=0.8pt,dash pattern=on 4pt off 4pt] (5.995935886720087,6.7427336643886076)-- (5.995935886720087,4.821695265253944);
\draw [shift={(4.9,4.3)},line width=1.pt]  plot[domain=-1.249045772398258:1.249045772398258,variable=\t]({1.*0.3162277660168374*cos(\t r)+0.*0.3162277660168374*sin(\t r)},{0.*0.3162277660168374*cos(\t r)+1.*0.3162277660168374*sin(\t r)});
\draw [shift={(5.1,4.3)},line width=1.pt]  plot[domain=1.8925468811915407:4.390638425988046,variable=\t]({1.*0.31622776601683794*cos(\t r)+0.*0.31622776601683794*sin(\t r)},{0.*0.31622776601683794*cos(\t r)+1.*0.31622776601683794*sin(\t r)});
\draw [line width=1.pt] (5.,4.)-- (5.,3.);
\draw [line width=1.pt] (5.,4.)-- (4.,3.);
\draw [line width=1.pt] (4.,3.)-- (4.,2.);
\draw [line width=0.8pt,dash pattern=on 4pt off 4pt] (5.995935886720087,6.7427336643886076)-- (5.,4.821695265253944);
\draw [line width=1.pt] (4.,3.)-- (3.,2.);
\draw [line width=1.pt] (3.,2.)-- (3.,1.);
\draw [line width=1.pt] (5.,3.)-- (5.,2.);
\draw [line width=0.8pt,dash pattern=on 4pt off 4pt] (4.455643836963446,6.7427336643886076)-- (5.,4.821695265253944);
\draw [line width=1.pt] (6.,12.)-- (6.,11.);
\draw [line width=1.pt] (5.,12.)-- (6.,13.);
\draw [->, midway,line width=1.pt] (6.,13.) -- (6.,12.);
\draw [->, midway,line width=1.pt] (8.,13.) -- (7.,12.);
\draw [->, midway,line width=1.pt] (8.,13.) -- (8.,12.);
\draw [line width=1.pt] (7.,12.)-- (7.,11.);
\draw [line width=1.pt] (7.,11.)-- (7.,10.);
\draw [line width=1.pt] (8.,12.)-- (8.,11.);
\draw [shift={(17.747639510300413,12.575075470770422)},line width=2pt]  plot[domain=0.:1.5707963267948966,variable=\t]({1.*1.*cos(\t r)+0.*1.*sin(\t r)},{0.*1.*cos(\t r)+1.*1.*sin(\t r)});
\draw [shift={(12.465324745562718,12.575075470770422)},line width=2pt]  plot[domain=0.:1.5707963267948966,variable=\t]({-1.*1.*cos(\t r)+0.*1.*sin(\t r)},{0.*1.*cos(\t r)+1.*1.*sin(\t r)});
\draw [shift={(16.605400462166287,1.2392182506514626)},line width=2pt]  plot[domain=-1.5707963267948966:0.,variable=\t]({1.*1.*cos(\t r)+0.*1.*sin(\t r)},{0.*1.*cos(\t r)+1.*1.*sin(\t r)});
\draw [shift={(12.724924529229563,1.2392182506514628)},line width=2pt]  plot[domain=-1.5707963267948966:0.,variable=\t]({-1.*1.*cos(\t r)+0.*1.*sin(\t r)},{0.*1.*cos(\t r)+1.*1.*sin(\t r)});
\draw [shift={(16.2,4.3)},line width=1.pt]  plot[domain=1.8925468811915318:4.390638425988053,variable=\t]({1.*0.3162277660168362*cos(\t r)+0.*0.3162277660168362*sin(\t r)},{0.*0.3162277660168362*cos(\t r)+1.*0.3162277660168362*sin(\t r)});
\draw [shift={(16.,4.3)},line width=1.pt]  plot[domain=-1.2490457723982438:1.2490457723982455,variable=\t]({1.*0.3162277660168379*cos(\t r)+0.*0.3162277660168379*sin(\t r)},{0.*0.3162277660168379*cos(\t r)+1.*0.3162277660168379*sin(\t r)});
\draw [shift={(15.2,4.3)},line width=1.pt]  plot[domain=1.8925468811915354:4.390638425988051,variable=\t]({1.*0.3162277660168374*cos(\t r)+0.*0.3162277660168374*sin(\t r)},{0.*0.3162277660168374*cos(\t r)+1.*0.3162277660168374*sin(\t r)});
\draw [shift={(15.,4.3)},line width=1.pt]  plot[domain=-1.2490457723982553:1.2490457723982553,variable=\t]({1.*0.3162277660168376*cos(\t r)+0.*0.3162277660168376*sin(\t r)},{0.*0.3162277660168376*cos(\t r)+1.*0.3162277660168376*sin(\t r)});
\draw [shift={(14.2,4.3)},line width=1.pt]  plot[domain=1.8925468811915407:4.390638425988046,variable=\t]({1.*0.31622776601683794*cos(\t r)+0.*0.31622776601683794*sin(\t r)},{0.*0.31622776601683794*cos(\t r)+1.*0.31622776601683794*sin(\t r)});
\draw [shift={(14.,4.3)},line width=1.pt]  plot[domain=-1.249045772398258:1.249045772398258,variable=\t]({1.*0.3162277660168374*cos(\t r)+0.*0.3162277660168374*sin(\t r)},{0.*0.3162277660168374*cos(\t r)+1.*0.3162277660168374*sin(\t r)});
\draw [shift={(17.2,4.3)},line width=1.pt]  plot[domain=1.892546881191538:4.390638425988048,variable=\t]({1.*0.3162277660168376*cos(\t r)+0.*0.3162277660168376*sin(\t r)},{0.*0.3162277660168376*cos(\t r)+1.*0.3162277660168376*sin(\t r)});
\draw [shift={(17.,4.3)},line width=1.pt]  plot[domain=-1.2490457723982598:1.2490457723982615,variable=\t]({1.*0.3162277660168362*cos(\t r)+0.*0.3162277660168362*sin(\t r)},{0.*0.3162277660168362*cos(\t r)+1.*0.3162277660168362*sin(\t r)});
\draw [line width=1.pt] (18.1,7.)-- (18.1,6.);
\draw [line width=1.pt] (15.1,8.)-- (15.1,7.);
\draw [line width=1.pt] (14.1,8.)-- (14.1,7.);
\draw [line width=1.pt] (14.1,8.)-- (13.1,7.);
\draw [line width=1.pt] (13.1,7.)-- (13.1,6.);
\draw [line width=2pt] (18.747639510300413,6.258147401543826)-- (18.747639510300413,12.575075470770422);
\draw [line width=2pt] (11.465324745562718,6.258147401543827)-- (11.465324745562718,12.575075470770422);
\draw [line width=2pt] (12.465324745562716,13.575075470770422)-- (17.747639510300413,13.575075470770422);
\draw [line width=2pt] (16.605400462166287,0.23921825065146263)-- (12.724924529229563,0.23921825065146288);
\draw [line width=0.8pt,dash pattern=on 4pt off 4pt] (16.1,4.75247)-- (16.1,7.41769);
\draw [line width=1.pt] (15.1,4.)-- (15.1,3.);
\draw [line width=0.8pt,dash pattern=on 4pt off 4pt] (15.095935886720088,6.7427336643886076)-- (15.09593588672009,4.821695265253944);
\draw [line width=1.pt] (14.1,4.)-- (14.1,3.);
\draw [line width=1.pt] (14.1,4.)-- (13.1,3.);
\draw [line width=1.pt] (13.1,3.)-- (13.1,2.);
\draw [line width=0.8pt,dash pattern=on 4pt off 4pt] (15.095935886720088,6.7427336643886076)-- (14.1,4.821695265253944);
\draw [line width=1.pt] (13.1,3.)-- (12.1,2.);
\draw [line width=1.pt] (12.1,2.)-- (12.1,1.);
\draw [line width=1.pt] (14.1,3.)-- (14.1,2.);
\draw [line width=0.8pt,dash pattern=on 4pt off 4pt] (13.555643836963448,6.7427336643886076)-- (14.1,4.821695265253944);
\draw [line width=1.pt] (15.1,12.)-- (15.1,11.);
\draw [line width=1.pt] (14.1,12.)-- (15.1,13.);
\draw [line width=1.pt] (16.1,12.)-- (16.1,11.);
\draw [line width=1.pt] (16.1,11.)-- (16.1,10.);
\draw [line width=1.pt] (17.1,12.)-- (17.1,11.);
\draw [line width=1.pt] (17.1,3.)-- (17.1,2.);
\draw [line width=0.8pt,dash pattern=on 4pt off 4pt] (17.086200894832594,4.873615221987312)-- (18.072680072766612,5.721641181965677);
\draw [->, midway,line width=1.pt] (9.,8.) -- (9.,7.);
\draw [->, midway,line width=1.pt] (4.,13.) -- (4.,12.);
\draw [->, midway,line width=1.pt] (6.,13.) -- (6.,12.);
\draw [->, midway,line width=1.pt] (8.,13.) -- (7.,12.);
\draw [->, midway,line width=1.pt] (8.,13.) -- (8.,12.);
\draw [line width=2pt] (17.747639510300413,5.258147401543826)-- (17.605400462166287,5.258147401543825);
\draw [shift={(17.747639510300413,6.2581474015438285)},line width=2pt]  plot[domain=-1.5707963267948966:0.,variable=\t]({1.*1.*cos(\t r)+0.*1.*sin(\t r)},{0.*1.*cos(\t r)+1.*1.*sin(\t r)});
\draw [line width=2pt] (17.605400462166287,5.258147401543825)-- (17.605400462166287,1.2392182506514626);
\draw [->, midway,line width=1.pt] (18.1,8.) -- (18.1,7.);
\draw [->, midway,line width=1.pt] (17.1,13.) -- (16.1,12.);
\draw [->, midway,line width=1.pt] (17.1,13.) -- (17.1,12.);
\draw [->, midway,line width=1.pt] (15.1,13.) -- (15.1,12.);
\draw [->, midway,line width=1.pt] (13.1,13.) -- (13.1,12.);
\draw [line width=2pt] (11.724924529229567,5.5859811128519485)-- (11.724924529229563,1.239218250651463);
\draw [shift={(12.465324745562679,6.258147401543817)},line width=2pt]  plot[domain=3.1415926535897833:3.8787234044495214,variable=\t]({1.*1.*cos(\t r)+0.*1.*sin(\t r)},{0.*1.*cos(\t r)+1.*1.*sin(\t r)});
\draw [->, midway,line width=1.pt] (17.1,4.) -- (17.1,3.);
\begin{scriptsize}
\draw [fill=black] (7.,8.) circle (2pt);
\draw [fill=black] (9.,8.) circle (2pt);
\draw [fill=black] (9.,7.) circle (2pt);
\draw [fill=black] (9.,6.) circle (2pt);
\draw [fill=black] (6.,8.) circle (2pt);
\draw [fill=black] (6.,7.) circle (2pt);
\draw [fill=black] (5.,8.) circle (2pt);
\draw [fill=black] (5.,7.) circle (2pt);
\draw [fill=black] (4.,7.) circle (2pt);
\draw [fill=black] (4.,6.) circle (2pt);
\draw [fill=black] (4.,13.) circle (2pt);
\draw [fill=black] (4.,12.) circle (2pt);
\draw [fill=black] (7.,4.) circle (2pt);
\draw [fill=black] (6.,4.) circle (2pt);
\draw [fill=black] (6.,3.) circle (2pt);
\draw [fill=black] (5.,4.) circle (2pt);
\draw [fill=black] (5.,3.) circle (2pt);
\draw [fill=black] (4.,3.) circle (2pt);
\draw [fill=black] (4.,2.) circle (2pt);
\draw [fill=black] (3.,2.) circle (2pt);
\draw [fill=black] (3.,1.) circle (2pt);
\draw [fill=black] (5.,2.) circle (2pt);
\draw [fill=black] (6.,12.) circle (2pt);
\draw [fill=black] (6.,11.) circle (2pt);
\draw [fill=black] (6.,13.) circle (2pt);
\draw [fill=black] (5.,12.) circle (2pt);
\draw [fill=black] (8.,13.) circle (2pt);
\draw [fill=black] (7.,12.) circle (2pt);
\draw [fill=black] (7.,11.) circle (2pt);
\draw [fill=black] (7.,10.) circle (2pt);
\draw [fill=black] (8.,12.) circle (2pt);
\draw [fill=black] (8.,11.) circle (2pt);
\draw [fill=black] (16.1,4.) circle (2pt);
\draw [fill=black] (18.1,7.) circle (2pt);
\draw [fill=black] (18.1,6.) circle (2pt);
\draw [fill=black] (15.1,8.) circle (2pt);
\draw [fill=black] (15.1,7.) circle (2pt);
\draw [fill=black] (14.1,8.) circle (2pt);
\draw [fill=black] (14.1,7.) circle (2pt);
\draw [fill=black] (14.1,8.) circle (2pt);
\draw [fill=black] (13.1,7.) circle (2pt);
\draw [fill=black] (13.1,7.) circle (2pt);
\draw [fill=black] (13.1,6.) circle (2pt);
\draw [fill=black] (15.1,4.) circle (2pt);
\draw [fill=black] (15.1,3.) circle (2pt);
\draw [fill=black] (14.1,4.) circle (2pt);
\draw [fill=black] (13.1,3.) circle (2pt);
\draw [fill=black] (13.1,2.) circle (2pt);
\draw [fill=black] (13.1,3.) circle (2pt);
\draw [fill=black] (12.1,2.) circle (2pt);
\draw [fill=black] (12.1,1.) circle (2pt);
\draw [fill=black] (14.1,3.) circle (2pt);
\draw [fill=black] (14.1,2.) circle (2pt);
\draw [fill=black] (15.1,12.) circle (2pt);
\draw [fill=black] (15.1,11.) circle (2pt);
\draw [fill=black] (14.1,12.) circle (2pt);
\draw [fill=black] (15.1,13.) circle (2pt);
\draw [fill=black] (16.1,12.) circle (2pt);
\draw [fill=black] (16.1,11.) circle (2pt);
\draw [fill=black] (16.1,11.) circle (2pt);
\draw [fill=black] (16.1,10.) circle (2pt);
\draw [fill=black] (17.1,12.) circle (2pt);
\draw [fill=black] (17.1,11.) circle (2pt);
\draw [fill=black] (17.1,3.) circle (2pt);
\draw [fill=black] (17.1,2.) circle (2pt);
\draw [fill=black] (16.1,8.) circle (2pt);
\draw [fill=black] (18.1,8.) circle (2pt);
\draw [fill=black] (13.1,13.) circle (2pt);
\draw [fill=black] (13.1,12.) circle (2pt);
\draw [fill=black] (17.1,13.) circle (2pt);
\draw [fill=black] (17.1,4.) circle (2pt);
\end{scriptsize}

\draw (3.747937355940069,11.824025247167047) node[anchor=north west] {$\moon$};
\draw (12.8,11.824025247167047) node[anchor=north west] {$\fullmoon$};

\draw (5.158540421511043,11.11872371438156) node[anchor=north west] {$\moon$};
\draw (14.263342026560062,11.11872371438156) node[anchor=north west] {$\moon_{\!\!\!0}$};

\draw (7.381308888471367,10.563031597641478) node[anchor=north west] {$\moon$};
\draw (16.4,10.563031597641478) node[anchor=north west] {$\moon_{\!\!\!0,1}$};

\draw (4.6,13.1) node[anchor=north west] {$=$};
\draw (4.6+2,13.1) node[anchor=north west] {$=$};
\draw (17.408848582438882,13.1) node[anchor=north west] {$\ldots$};

\draw (4.7527599611804625,8.748302311023958) node[anchor=north west] {$2$};
\draw (5.763629659364042,8.748302311023958) node[anchor=north west] {$1$};
\draw (6.734064569620279,8.748302311023958) node[anchor=north west] {$0$};
\draw (8.735586572023767,8.748302311023958) node[anchor=north west] {$0$};
\draw (3.4184119595781373,8.222650067968496) node[anchor=north west] {$\ldots$};
\draw (7.603412510058157,8.222650067968496) node[anchor=north west] {$=$};

\draw (12.520583639735026,8.188836707629266) node[anchor=north west] {$\ldots$};
\draw (13.853950876536391,8.708330436253174) node[anchor=north west] {$2$};
\draw (14.875621876163411,8.708330436253174) node[anchor=north west] {$1$};
\draw (15.84534350292804,8.708330436253174) node[anchor=north west] {$0$};
\draw (16.71116638396789,8.188836707629266) node[anchor=north west] {$=$};
\draw (17.83673612931969,8.708330436253174) node[anchor=north west] {$0$};

\draw (3.810405456474163,1.5220005236224352) node[anchor=north west] {$\infty_{1,2}$};
\draw (5.645949964278639,2.647570268974238) node[anchor=north west] {$\infty_{0}$};
\draw (6.652528896775988,3.647805867286225) node[anchor=north west] {$\infty$};
\draw (3.4313862108911586,4.178763452871637) node[anchor=north west] {$\ldots$};
\draw (12.917828406683844,1.5239755249445763) node[anchor=north west] {$\infty_{1,2}$};
\draw (14.758481370046603,2.6389864546739417) node[anchor=north west] {$\infty_{0}$};
\draw (15.749602196472704,3.647805867286225) node[anchor=north west] {$\infty$};
\draw (12.528459510587876,4.178763452871637) node[anchor=north west] {$\ldots$};
\draw (16.593828076542962,1.6805633521632894) node[anchor=north west] {$\nym_0$};

\draw (2.9040124924278277,8.902244310010737) node[anchor=north west] {SGT};
\draw (4.847535600590422,0.13614529090236482) node[anchor=north west] {SFPT};
\draw (5.27268128050099,14.1659527279511) node[anchor=north west] {LNST};
\draw (13.350449198801774,0.07541019377228372) node[anchor=north west] {Unified Theory};

\end{tikzpicture}}
\vspace{-1cm}
    \caption{Unified theory.}
    \label{fig:fig16}
\end{figure}

In SFPT, a component within a cyclic zone behaves as a situation in which the next player can either pass or exit by moving to a {\sc nim} value. If no such exit exists, the game is simply $H=\{H\}$ (or $H=\{\text{{\sc pass}}\}$). Consequently, the game digraphs of these forms consist of a root with a loop and possible exits to digraphs corresponding to ordinary {\sc nim} values.

The right side of the figure already pertains to the unified theory proposed here. There are now two fundamental ideas to highlight. The first is that there is no longer a single lunar value. Game forms such as $G$, $G'$, and $G''$ no longer belong to the same equivalence class. Given the possibility of checkmate, we still have $G+X\in\mathcal{N}$, and the Grundy value of
$G$ is now full moon, denoted by $\fullmoon$. However, if $X$ is in a cyclic zone, it may happen that we have $G'+X\in\mathcal{D}$. For example, this occurs whenever the Grundy value of $X$ is $\infty_1$; in this case the carry-on move in $G'$ does not allow a safe exit from the cyclic zone. Hence $X$ suffices to distinguish $G$ from $G'$. On the other hand, $G''+X\in\mathcal{N}$, since the next player can force the opponent to move to $*$ in $G''$, and then exit the cyclic zone victoriously to $*+*$. Therefore, $G''$ does not belong to the same equivalence class as $G'$. All of this occurs because the way a moon is ``neutralized'' before leaving a cyclic zone depends on the forcing sequences available there. This explains why the lunar equivalence class in LNST ``splits'' in the unified theory.

The second idea concerns game forms such as $*n$ and $\{\circlearrowright^{*n}\}$. In both theories, we have $*n=\{\circlearrowright^{*n}\}$; the clearest example is  the pair $0$ and $\{\circlearrowright^{0}\}$ where both are $\mathcal{P}$-positions. When comparing $*n + X$ with $\{\circlearrowright^{*n}\} + X$, making or not the carry-on in $\{\circlearrowright^{*n}\}$ before playing in $X$ is irrelevant. At any moment the opponent can do it, leaving everything the same.

What is particularly interesting happens when one examines the games $J=\{J,0,*,\ldots,*(n-1)\}$ and $J'=\{J',\circlearrowright^{*n}\}$. By adding a loop at the root, the options of $*n$ and $\{\circlearrowright^{*n}\}$ have now become exits from a cyclic zone. And, when the {\sc pass} move exists, leaving the cyclic zone \emph{while keeping the right to play} and leaving \emph{without keeping the right to play} are completely different. For example, $J+*(n+1)\in\mathcal{D}$, but $J'+*(n+1)\in\mathcal{N}$. In the second case, the next player makes the carry-on move in the first component, \emph{keeping the right to play}, and then plays in the second component, achieving $*n+*n$. Moreover, sometimes the only way to draw the game is to use the {\sc pass} move. For example, the only reason why
$J'$ is not equal to $*n$ is the fact that in $J'+*n$ the next player can draw by \emph{passing} in the first component. Thus, the equivalence class of $J'$ is new with respect to SFLT, and its Grundy value is denoted by $\nym_{*n}$ (nymphet). In sum, although $*n$ and $\{\circlearrowright^{*n}\}$ are equivalent, the cyclical situations created by the inclusion of the {\sc pass} move in these game forms are not, giving rise to a new type of cyclic value.

\bibliographystyle{bibstyle}
\bibliography{bib}

@book{ANW019,
  title = {Lessons in play},
  author = {Albert, Michael and Nowakowski, Richard and Wolfe, David},
  publisher = {Routledge Cavendish},
  edition = {2},
  month = {apr},
  year = {2019},
  address = {London, England},
  language = {en},
}

@book{Ber000,
  author = {Berlekamp, Elwyn R.},
  title = {The Dots and Boxes game: sophisticated child’s play},
  publisher = {A K Peters, Ltd., Natick, MA},
  year = {2000},
}

@book{BCG001,
  author = {Berlekamp, Elwyn R. and Conway, John H. and Guy, Richard K.},
  title = {Winning ways for your mathematical plays. {V}ol. 1},
  edition = {Second},
  publisher = {A K Peters, Ltd., Natick, MA},
  year = {2001},
  pages = {xx+276},
  isbn = {1-56881-130-6},
  mrclass = {91-02 (00A08 91A46)},
  mrnumber = {1808891},
  mrreviewer = {Lawrence\ Somer},
}

@article{Bou02,
  author = {Charles Bouton},
  title = {Nim, a game with a complete mathematical theory},
  journal = {The Annals of Mathematics},
  volume = {3(2)},
 publisher = {Princeton},
  year = {1902},
  pages = {35--39},
}

@book{Con001,
  author = {Conway, John H.},
  title = {On numbers and games},
  edition = {Second},
  publisher = {A K Peters, Ltd., Natick, MA},
  year = {2001},
  pages = {xii+242},
  isbn = {1-56881-127-6},
  mrclass = {00A08 (05-01 91A05)},
  mrnumber = {1803095},
}

@article{Fra73,
  author = {Fraenkel, Aviezri and Perl, Yehoshua},
  title = {Constructions in combinatorial games with cycles},
  journal = {Infinite and finite sets},
  volume = {2},
 publisher = {Colloquia Mathematica Societatis J\'{a}nos Bolyai},
  year = {1973},
  pages = {667--699},
}

@article{Lar021,
  author = {Larsson, Urban and Nowakowski, Richard and  Santos, Carlos P.},
  title = {Impartial games with entailing moves},
  journal = {Integers: Electronic Journal of Combinatorial Number Theory},
  volume = {21B},
  year = {2021},
}

@article{Lar024,
  author = {Larsson, Urban and Nowakowski, Richard and  Santos, Carlos P.},
  title = {Affine normal play},
  series = {Mathematical Sciences Research Institute Publications},
  journal = {Games of No Chance 6},
  volume = {71},
  year = {2025},
  pages = {135--190},
}

@article{Smi66,
  author = {Smith, Cedric},
  title = {Graphs and composite games},
  journal = {Journal of Combinatorial Theory series A},
  volume = {1},
  year = {1966},
  pages = {51--81},
}

@article{Spra35,
  author = {Sprague, Roland},
  title = {Über mathematische Kampfspiele},
  journal = {Tohoku Mathematical Journal},
  volume = {41},
 publisher = {Tohoku University},
  year = {1935},
  pages = {438--444},
}

@book{Sie013,
  author = {Siegel, Aaron N.},
  title = {Combinatorial game theory},
  series = {Graduate Studies in Mathematics},
  volume = {146},
  publisher = {American Mathematical Society, Providence, RI},
  year = {2013},
  pages = {xiv+523},
  isbn = {978-0-8218-5190-6},
  mrclass = {91-01 (05-01 91A46)},
  mrnumber = {3097920},
  mrreviewer = {Nhan\ Bao\ Ho},
  doi = {10.1090/gsm/146},
  url = {https://doi.org/10.1090/gsm/146},
}

@article{Sue24,
  author = {Suetsugu, Koki},
  title = {\href{https://arxiv.org/abs/2404.00632}{What happens when we add impartial loopy games and impartial entailing games?}},
  journal = {Preprint},
  year = {2024},
  note    = {arXiv:2404.00632 [math.CO]},
}

@article{Gru39,
  author = {Grundy, Patrick},
  title = {Mathematics and games},
  journal = {Eureka},
  volume = {2},
 publisher = {Cambridge University},
  year = {1939},
  pages = {2--6},
}

@article{Wyt07,
  author = {Wythoff, Willem},
  title = {A modification of the game of Nim},
  journal = { Nieuw Archief voor Wiskunde},
  volume = {7},
  year = {1907},
  pages = {199--202},
}

\end{document}